\documentclass[10pt,onecolumn]{IEEEtran} 

\usepackage{amsmath,xurl}
\usepackage{dsfont}
\usepackage{graphicx}
\usepackage{tikz,ifthen}
\usepackage{hyperref}
\usepackage{epsfig,amssymb,amsbsy,verbatim,array,enumerate}
\usepackage{pstricks,psfrag,theorem,cite,paralist,bm,caption,subcaption}
\usepackage{enumitem,wasysym}
\usepackage{amssymb}

\usepackage{parskip}
\let\intern=\iftrue

\def\la{\lambda}
\def\E{\mathbb{E}}
\def\P{\mathbb{P}}
\def\R{\mathbb{R}}

\def\N{\mathbb{N}}

\def\H{\mathbb{H}}

\def\Ical{\mathcal{I}}
\def\Lcal{\mathcal{L}}
\def\Hcal{\mathcal{H}}

\def\Rcal{\mathcal{R}}

\def\Vcal{\mathcal{V}}

\def\Scal{\mathcal{S}}

\def\Tcal{\mathcal{T}}
\def\Ical{\mathcal{I}}

\def\x{\mathsf{x}}

\def\s{\sigma}
\def\g{\gamma}
\def\a{\alpha}
\def\b{\beta}

\def\G{\Gamma}

\def\nn{\nonumber}

\def\phi{\varphi}

\newcommand{\sst}{\scriptscriptstyle}

\newcommand{\half}   {{\frac{1}{2}}}

\newcommand{\one}{{\mathds{1}}}

\newtheorem{theorem}{Theorem}
\newtheorem{lemma}{Lemma}
\newtheorem{corollary}{Corollary}
\newtheorem{definition}{Definition}

\newtheorem{proposition}{Proposition}
\newtheorem{remark}{Remark}

\newlength{\figwidth}


\newenvironment{spacedtable}[2]{%
\begingroup
\renewcommand{\arraystretch}{#1}
  \setlength{\tabcolsep}{#2}
}{%
  \endgroup
}
\begin{document}
\title{Seasonal Statistics of Shannon Rate in a
Dynamical Poisson-Voronoi Cellular Network} 

\author{Sanjoy Kumar Jhawar\textsuperscript{1,2}
\thanks{\hspace{-0.11in}\textsuperscript{1}T\'{e}l\'{e}com Paris, \textsuperscript{2}INRIA Paris and \textsuperscript{3}{\'E}cole Normale Sup{\'e}ri{\'e}ure Paris. \\ \hspace{0.45in}E-mail: {\tt sanjoy.jhawar@telecom-paris.fr}.} and 
Fran\c{c}ois Baccelli\textsuperscript{1,2,3}
\thanks{\hspace{-0.11in}E-mail: {\tt francois.baccelli@inria.fr}.}}
\maketitle
\begin{abstract}
In this work we consider a dynamical cellular communication network in which mobile base stations (BSs) are modeled as a homogeneous Poisson point process on $\R^2$. Each base station moves at a constant speed in a random direction. A typical user connects to the nearest base station and it experiences variable signal and interference powers depending on the distance of all the stations. Along the motion of the stations, the user swaps its serving station, and such an event is called a {\em handover}. We are interested in the performance evaluation of the system under some classical and tropical metrics of interest at different time of events, inducing handovers, maximal proximity of serving station, nearest interferer at closest or farthest distance with respect to the user or at any typical time epoch. The main results of the paper are closed or integral form expressions
for the basic metrics of interest, in particular coverage probability
and average Shannon rate at these epochs. We can make an analogy with ``seasons'' based on the fluctuations of signal and interference power. Strong or mild signal or interference power correspond to different seasons of Shannon rate along the evolution of the system. We also provide a complete comparison study of the metrics at
interest at these epochs. We also provide a recipe to lift our results for a 3D wireless communication model for UAVs and HAPS.  
\end{abstract}
\begin{IEEEkeywords}
Wireless communication, handover,  stochastic geometry, dynamical communications, performance evaluation, Shannon rate, dimensioning, seasonal statistics, time series of Shannon rate, tropical SINR.
\end{IEEEkeywords}
\captionsetup{font=footnotesize}
\section{Introduction}
Modern wireless communication systems rely or will rely on the large scale deployment of vehicular and in particular non-terrestrial base stations (BS). Due to the mobility of such BSs, it is extremely important to asses the performance at all epochs of this process and find out ways to mitigate challenges posed by certain types of ``good'' or ``bad'' epochs or configurations in such systems. The most important performance metric for the quality of service is the distribution of the Shannon rate experienced by a typical user in the system, which is a function of signal-to-interference-plus-noise-ratio (SINR) at any time instant. Our results essentially deal with the SINR distribution at various ``typical epochs'' of interest, and their comparison.

To characterize the time evolution of Shannon rate, we define a dynamical communication network model with random locations of BSs on the Euclidean plane. This randomness is complemented by their random motion, and this makes the time evolution of the SINR and Shannon rate as a stochastic process. As a first step we derive our results only for the dynamical communication model on 2D, even though our motivation comes from 3D non-terrestrial networks with a spherical structure. The present paper is hence essentially a first step in the direction
of the analysis of this higher dimension problem. 

We use stochastic geometry as a central tool in our analysis, which plays an important role in developing suitable models to understand many statistical features of key performance metrics in such systems. As we will see, it allows one to derive ensemble averages of the system properties, including at such specific epochs. 

In this article we focus on determining performance metrics of dynamical cellular communication systems at specific typical epochs of interest. For example, consider a time epoch when the power of the signal is maximum. This corresponds to the time when the serving BS is at its minimal distance from the user in its motion. The SINR at this typical epoch is measured using the distance to the serving BS and that of all others. To measure the performance of the system, we take an average over all possible configurations of the system at such epochs, instead of taking averages at all times. 

We first define the dynamical cellular communication model on $\R^2$. The BSs are distributed as a homogeneous Poisson point process and all move at constant speed $v$ and with a uniformly chosen and independent random direction. A user located at the origin ($o$) connects to the nearest BS at all time instants and changes its connection to whichever is the closest. This phenomenon is called a {\em handover}. Due to the motion of every station, the user encounters changes in the distance to the serving station, along the spatio-temporal evolution of the system. These temporal variations result in a continuous change in the performance metrics in terms of signal and interference power. We are in particular interested in determining the behavior of the system performance at specific time epochs whenever the user sees an abrupt change, for example at a handover epoch, at a time when the signal power is locally maximum, or at a time when the nearest interferer is at its closest or farthest distance. 

Our investigation is based on analyzing the fluctuations of the SINR with the additive total interference by considering interference as noise to the signal received from the serving station. As we will see, this is governed by the joint distribution of the distance to the nearest station, contributing to the signal power, as well as the distance to all other stations, contributing to the  interference power. On the other hand, their comparison at different typical epochs is mainly governed by the comparison of the distribution of the distance to the serving BS or the closest interferer depending on the typical epochs. As a special case, we also use the maxitive (or tropical) interference, with a goal to capture the effect of the strongest interferer on the performance evaluation.

Moving beyond the cellular scenario, our analysis can be adapted for the performance evaluation of a 3D wireless communication model for unmanned aerial vehicle (UAV) and high altitude platform station (HAPS). Here, the BSs are moving on a  plane at an altitude, say $a$, parallel to the plane of the user. As we will see later in Section~\ref{section:3D}, with a little change in the distance variables using Pythagoras theorem, we are able extend our results to the distribution of typical distances, coverage probabilities and average Shannon rates, to this practically important case. This allows us to lift our focus beyond 2D communication system. These results should be useful to the current research community~\cite{5G-6G},~\cite{Banagar-Chetlur-Dhillon},~\cite{HAPS} focusing on 3D wireless communication systems. 

\subsection{Motivation}
Most of the literature about performance evaluation of wireless communication systems only deals with a snapshot analysis. For example, in a cellular system, under Poisson setting, it is
is well known that the random variable for the
distance to the nearest BS is distributed as Rayleigh. This is not the case at a handover epoch, as the system is conditioned on the existence of two BSs at equal distance from the typical user. As a result, the previous works are inadequate to capture the performance of the system at such specific epochs, for example at a time when the signal or interference power is locally maximum or minimum. Our objective is to formally analyze a spatio-temporal model with the goal to capture the evolution of system performance at the typical epochs of interest listed before. As already mentioned, knowing a potentially extreme behavior at specific epochs can help the operator to prepare and dimension the system accordingly in all these best and perhaps more importantly worst case scenarios. From the system design perspective, it is also important to compare the performance at these time epochs. 

\subsection{Illustration for seasonal statistics}
As the title of this article suggests, one can better visualize this interplay between the evolving strong or weak signal and interference powers, as the evolution of ``seasons'' over time, where there is a unique source of signal power, and a joint source of interference power. The entire analysis is based on the identification of typical events corresponding to {\em max-signal}, {\em min-signal}, {\em min-interference} or {\em max-interference}, and  probabilistic features of these seasonal statistics. Our comparison results focus on the coverage probability and  the average Shannon rate. Even though the time evolution of the interplay between strong or weak signal and interference power resembles the dynamics of seasons, in contrast,  our spatio-temporal model is highly non-periodic. For instance, it can be seen to possibly possess two max-interference seasons without even having a max-signal in between. Also, there can be two max-signal seasons and a max-interference season in between, and vice versa.

In our context, the temporal evolution or the time series of the Shannon rate, can be better visualized in terms of Figure~\ref{fig:Shannon-rate}. In the tropical case, the interference is considered to be from the maximum interferer only, and the effect of such a reduction can be seen in the entire analysis, from seasonal statistics to their comparison. Our analysis and results provide insight about the worst case scenarios at extreme seasons and can help the system designer to dimension the network based on this, rather on the typical time. This also prepares the system to be more proactive than reactive for the time around such typical epochs. This essentially allows one to derive recommendation to the operator to dimension the system based on the extreme behavior at those typical epochs, rather than the average one. Let us stress two important facts pertaining to the term ``season"
(a) throughout the paper, this term has to be understood 
as some extreme point in the season (e.g.``peak of winter", here
a handover event, where by construction, signal power is locally minimal
and interference power locally maximal) rather than a time interval (the whole winter,
or some interval centered on the handover);
(b) seasons are in no way periodic here, with periodicity
replaced by a Markov structure here. 

As we will see later, in our context, seasons mean specific epochs where an
inflection, for instance a peak of the evolution of the Shannon rate is observed, rather than intervals where this rate is better or worse. Archetypes will be when signal power is locally maximal or when interference power is locally maximal. With this convention in mind, seasons can be refined further than what we do in the present paper, based e.g. on
certain derivatives. For instance, we could distinguish a local maximum of signal power when
interference power is decreasing, from a local maximum where interference power is increasing.
The techniques developed here clearly allow one to analyze such refinements, but this is left for future research. 

The Figure~\ref{fig:trajectory} depicts a realization of the location and trajectories of the BSs.
\begin{figure}[ht!]
\centering    \includegraphics[width=0.4\textwidth]{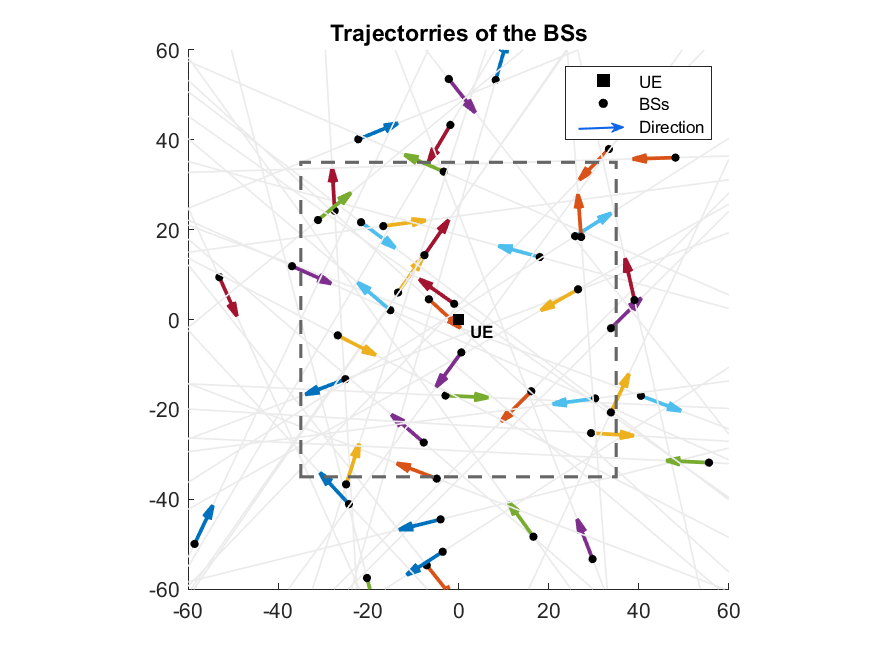}
\caption{Trajectory of BSs (grey lines) in a finite window with respect to a fixed user at the origin. The black dots are the initial locations of the BSs, distributed as a homogeneous PPP and the arrows are for the direction of motion.} 
\label{fig:trajectory} 
\end{figure}
The evolution of the distance to the nearest and second nearest BSs, and the corresponding time series of the Shannon rates are shown in Figure~\ref{fig:Shannon-rate}, for the simple case of tropical SIR with power law attenuation for different attenuation exponent $\a$ and without fading. In the interference limited no-fading case, the Shannon rate at time $t$ is defined as
\[
\Rcal(t):=B\ln\left(1+ \frac{H(t)^{-\a}}{R(t)^{-\a}}\right) \mbox{ bps},
\]
where $H(t)$ and $R(t)$ are the distance to the nearest and second nearest BS, respectively, at time $t$, $\a$ is the exponent of the path-loss attenuation and $B=b/{\ln 2}$, with $b$ being the bandwidth. Depending on the time-evolution of the ratio of the distance to the serving BS and nearest interferer there are high or low ``seasons'' of Shannon rate along the evolution of the system. We are interested in  the characteristics of it at the seasons alluded before and depicted with vertical lines in Figure~\ref{fig:Shannon-rate}. The non-periodicity of the seasons and their sequential but random appearance are also evident even in the simple tropical model without noise and fading, whereas handover epochs are the worst seasons in terms of Shannon rate experienced by the user. It can be seen as a one-to-one ``tug of war'' between the signal and tropical interference power over time. The local extreme nature of Shannon rate is very clear at handover and max-signal epochs, where as there is only a slight
change of its slope at max-interference and min-interference epochs.
The time series of the Shannon rate turns out to be even more complex and infinitely dependent, while moving beyond tropical scenario with or without fading and noise. As we will see later in Remark~\ref{remark:DDM}, in a ``dual dynamical model'' with a mobile user, which encounters the identical seasonality along its motion in a straight line, with BSs being fixed.
\begin{figure}[ht!]
\centering
\includegraphics[width=0.6\textwidth]{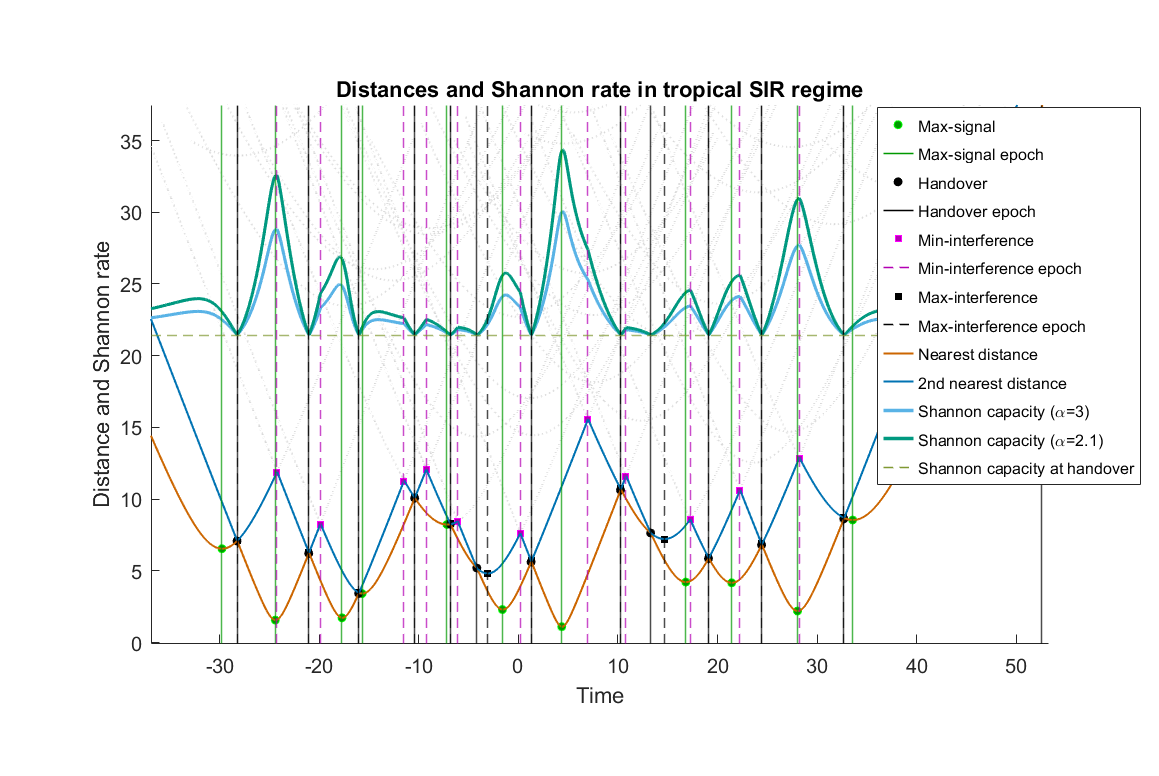}
\caption{Time series of distance (unit: meters) to the nearest (in red) and second nearest (in blue) BS(s) and that of Shannon rates (unit: bps) and typical epochs for $\a=2.1$ (in light green) and $3$ (in sky blue) in the tropical case without fading. The distance to all other BSs are shown with light grey curves in the background. The Shannon rate plots are vertically elevated for better visibility. The seasons are marked with vertical lines, and they appear non-periodically, for example: max-signal, handover, max-signal, min-interference, handover, max-interference and so on (counting from the left to right). } 
\label{fig:Shannon-rate} 
\end{figure}
\subsection{Contributions}
Below, we list out the main contributions made in this article. Most of our results are based on the power law attenuation function  $\ell(r)= r^{-\a}$ and Rayleigh fading, where $\a\geq 2$ is the path-loss exponent.  
\begin{enumerate}
    \item To set the stage for the performance evaluation, we identify and formalize the specific time epochs of interest within the spatio-temporal evolution of the system in Section~\ref{section:prelims}. This also includes the determination of the intensity of such epochs and the corresponding distance characteristics, which are established in Appendix~\ref{section:A-prelim}.
    \item To determine the distribution of the interference power, we derive the distribution of the point process of distance of the BSs, other than the serving one at these epochs, which in our case turns out to be conditionally Poisson, in Lemma~\ref{lemma:etappp}.
    \item This leads to a  determination of the coverage probability and average Shannon rate at each of the typical epochs for both non-tropical and tropical scenarios in Section~\ref{section:results}, mainly in the form of Theorem~\ref{theorem:GPM-WF}-~\ref{theorem:GPM-NF2}.
    \item An analytic comparison of the performance metrics (coverage probability and average Shannon rate) is also shown to hold between different typical epochs for both non-tropical and tropical scenarios in Subsection~\ref{subsection:comparison}. These comparisons are also validated by plotting these closed form or integral form expressions with respect to the user QoS parameter $\tau$ in Section~\ref{sec-numerical}.
    \item We also prove a scale-invariance property for the coverage probabilities and average Shannon rates at these typical epochs in the interference limited case, in Subsection~\ref{subsection:SI}.
    \item We manage to establish some closed or integral form expressions for the coverage probabilities, as presented in Section~\ref{sec-closedforms}. A significant take away is that, in the interference limited regime with fading factor and with path-loss exponent $\a=4$, the coverage probability at any of the typical epoch can be represented in terms of a non-linear function of the coverage probability at a typical time.
    \item As we will see later in Subsection~\ref{subsection:example-AF}, it can be adapted to a more general class of attenuation functions and fading factors. Some of the comparison results hold true for the SINR defined using attenuation function other than the power law $\ell(r)=r^{-\a}$, for which we provide an example of bounded attenuation function $\ell(r)=(1+r)^{-\a}$ and a counter example as well of step attenuation function with given radius. 
    \item Finally, in Section~\ref{section:3D} we state similar results about the distance distributions for a 3D mobility model for UAVs and HAPS. This allows us to extend our planar results to this 3D model, on the coverage probabilities, average Shannon rates and their comparisons among different typical epochs in Section~\ref{section:3DCP}.
    \end{enumerate}
\subsection{Organization of the article}
This article starts with the initial description of the underlying system model and a set of questions of interest in Section~\ref{section:system}. Section~\ref{section:prelims} is dedicated to mathematical preliminaries, including a complete description of the system variables of interest. The proofs of the main results for the last section are presented in Appendix~\ref{section:A-prelim}. Section~\ref{section:results} contains all the major results about the performance evaluation of the system at different typical events or equivalently seasons. This also includes the comparison among these QoS at typical events. Section~\ref{sec-closedforms} highlights some of the closed form or simple integral form expressions of the coverage probabilities in some of the regimes, due to specific values of the path-loss exponent $\a$. We present the plots of the closed form or integral form expressions of coverage probabilities with respect to different threshold parameter $\tau$ in Section~\ref{sec-numerical}. The proof of the main results about the performance evaluation and their expressions at different typical events are presented in Appendix~\ref{proofs-wf}, for the case with fading. The proofs for the case without fading is presented in Appendix~\ref{section:PWFa}. Appendix~\ref{section:CPE} is dedicated to the proof of the comparison among the coverage probabilities at those typical epochs in cases with and without fading, for non-tropical and tropical scenarios. The comparison of additive interference, as well as the tropical interference at different typical times is presented in Appendix~\ref{Int-order}. Lastly, the proof for scale invariance property in the interference limited regime is given in Appendix~\ref{section:SI}.
\subsection{Literature review}
The first closed form expressions for average Shannon rate in cellular wireless communication systems using stochastic geometry were obtained in~\cite{Andrews-etal}, when base stations are modelled as a Poisson point process.
The analysis of cellular networks using tools from stochastic geometry gained significant traction following~\cite{Andrews-etal}, and also~\cite{Andrews-Gupta-Dhillon}.
Building on this foundation, performance evaluation of heterogeneous networks was explored in ~\cite{Dhillon2012},~\cite{Jo2012} and~\cite{ElSawy2013}. 

This was extended to the setting of satellite networks using stochastic geometry approach by Okati et. al. in~\cite{Okati-thesis} and~\cite{okati}  addressing the unique geometric and statistical challenges posed by non-terrestrial infrastructures. This line of work laid important groundwork for analyzing low Earth orbit (LEO) constellations, where the spatial distribution of satellites and their motion relative to ground users introduce non-trivial interference and coverage characteristics.
The characterization of satellite constellations themselves has also received considerable attention. Choi et al.~\cite{Choi2022ANA} proposed a constellation model grounded in stochastic geometry, enabling the analysis of coverage and connectivity for large-scale LEO deployments. This was further generalized to multi-altitude constellation architectures~\cite{Multi}, where satellites operating at different heights create layered interference environments. LEO satellite based
communication systems are also studied in~\cite{Talgat-Alouni2020}. In a recent
work~\cite{Lee-Park-2026}, the authors evaluated the coverage probability
for NTN system with LEO satellites in a heterogeneous environment.

The dynamic nature of satellite on the orbits has motivated the development of dedicated mobility and motion models. Baccelli and Zuyev~\cite{Baccelli_Zuyev} introduced a dual dynamical model that jointly accounts for the motion of user and the stochastic geometry of static terrestrial network. The question about handover probability in both the dual dynamical models and their equivalence is studied by Banagar et al. in~\cite{HO-Prob}. Correlation structure of SNR in such a mobility model was studied in~\cite{SNR}. Among many~\cite{UAV-Channel}, \cite{Wang-etal},~\cite{UAV-Rural}, including the last article~\cite{SNR}, and references therein contain a rich set of results on performance evaluation of wireless communication model for UAVs and HAPS. Complementing this, the random direction model~\cite{RDM_nain} provides a tractable abstraction for user or node mobility, and has found application in both terrestrial and non-terrestrial network analysis. Random geometric graph with similar mobility of random direction model is studied in~\cite{diaz2007dynamic}. 

The work~\cite{Nitin-Baccelli} and \cite{Lin2012} examines the impact of mobility on network performance, characterizing how node movement affects interference statistics and spatial reuse. Gong et al.~\cite{Gong2010} further incorporated fading in terms of mobility and channel randomness shapes coverage and outage probabilities. These contributions highlight that a complete performance analysis must account for both the spatial and temporal dimensions of network dynamics.
A closely related thread concerns the temporal correlation of interference, which is of central importance when evaluating link-layer protocols and re-transmission schemes. Ganti and Haenggi~\cite{Ganti-Haenggi-SP} and, Krishnan and Dhillon~\cite{KD} provided a rigorous treatment of this temporal correlation of interference and mitigation of the adversity provided by the later in Poisson networks. In a new paradigm, the effect mobility on the association policies in terms of Doppler shift and the performance evaluation is explored in~\cite{balakrishnan2026doppler}. Finally, questions of resource sharing and data rate allocation in spatially distributed networks have been addressed by Madadi et al.~\cite{Pranav-etal}, who studied shared data rate models that account for the stochastic geometry of the network alongside fairness and efficiency constraints. In a two-tier cellular system handover management has been studied in~\cite{Arshad-Alouni17}. The notion of maximal power association in angular setting is studied in~\cite{Armeniakos-etal}. The current research is based on the handover analysis for the mobility model in~\cite{FB-SKJ} and as an initial ingredient about distance characteristics and Palm distributions are borrowed from there. The distance characteristics of similar interest in Poisson Voronoi tessellations are also studied in~\cite{Multi-point}.  
\section{system Model and problem formulation}~\label{section:system}
\subsection{System Model}
In this article, we work on a  dynamical wireless communication network model based on a Poisson point process on $\R^2$, as introduced in~\cite{FB-SKJ}.
Consider a collection of BSs located at points of a homogeneous Poisson point process $\Phi_\la\equiv \Phi= \sum_{i}\delta_{X_i}$ on $\R^2$ with intensity $\la$. Suppose each of the BS moves at constant speed $v$, in a uniformly and independently chosen random direction. The BSs do not change their direction along the rest of the course. Suppose the directions are given by a collection of i.i.d. random variables $\{\Theta_i\}_i$. A user located at the origin stays connected to the nearest BS at any time, using a {\em distance based association rule}. During this connection period, the serving BS provides the downlink signal to the user and the rest of the BSs contribute to the total interference experienced by the user. The user changes its connection from one BS to next over time. Such a change of BS is called a {\em handover}. The time epochs of these handover events form a point process on the time axis.

For any BS starting at a location $X\in\Phi$ and moving in the direction $\Theta$, it is located at $X^t:=X+Vt$, at time $t$, where $V=v(\cos \Theta,\sin \Theta)$ and $v$ is the modulus of the speed. We denote the point process corresponding to the new locations as $\Phi^t:=\sum_{i}\delta_{X^t_i}$. It is well known from the displacement theorem \cite[Section 5.5]{kingman} that for all $t$, $\Phi^t$ is a homogeneous Poisson point process with intensity $\la$. We always write the location of the nearest BS as $X_0^t$.
The time varying distance of a BS starting at $X=(\vert|X|\vert\cos\Psi,\vert|X|\vert\sin\Psi)$  is given by the function
\begin{equation}
d_X(t):=\vert|X^t\vert|=\left(||X||^2+2vt ||X||\cos\Xi+v^2t^2\right)^{\half},
\label{eq:distancef}
\end{equation}
where $\Xi:= (\Theta-\Psi) \;\,\mbox{mod}(2\pi)$. Note that $\Xi\stackrel{d}{\sim} \mbox{U}[-\pi, \pi)$. The curve of the distance function will be used throughout the article. As we will see later, these curves are the branches of certain hyperbolas on the upper half space $\H^+:=\R\times \R^+$.  

\begin{remark}[Dual dynamical model]~\label{remark:DDM}
As it was shown in~\cite{FB-SKJ}, this dynamical model is equivalent to the dual dynamical model, where the BSs are fixed but a mobile user is moving on a straight line at a constant speed. All the characteristics related to handovers, as seen in~\cite{FB-SKJ}, at other typical epochs, and also the corresponding performance evaluations, are the same for both dynamical models.
\end{remark}

In the first model, in principle, at any time instant, the power of the signal emitted from each BS is a function of the distance from the user, using the power law  attenuation function $\ell(\vert|X\vert|)= \vert|X\vert|^{-\a}$ for a BS located at $X$, where $\a\geq 2$ is the path-loss exponent. In a more realistic scenario, one can incorporate a fading factor, given by a random variable with a given distribution. In the modelling of wireless communications, the performance of the serving BS is usually measured in terms of the SINR experienced by the user at time $t$, which is 
\begin{equation}
    \mbox{SINR}(t):= \frac{\rho \vert|X_0^t\vert|^{-\a}}{\s^2+\sum_{Y\in \Phi^t\setminus \{X_0^t\}} \rho_Y \vert|Y\vert|^{-\a}},
    \label{eq:SINRt}
\end{equation}
where $\s^2$ is the thermal noise power, $Y\in \Phi^t\setminus \{X_0^t\}$ are the locations of all other BSs at time $t$ and $\rho, \rho_Y$ are the i.i.d. fading factors. These can be considered as Rayleigh fading, but more general classes of fading distributions can cover a large set of use cases. All other performance metrics, namely, {\em signal-to-noise-ratio} (SNR) and {\em signal-to-interference-ratio} (SIR) are defined similarly both in the fading and no-fading cases. 

In addition, we also define a more restrictive version of the performance metrics using tropical interference, by this we mean that we replace the additive interference by maxitive interference over $\Phi^t$, see~\cite{Baccelli09now} and~\cite{Haenggi-Ghanti}. In this case we are interested in the {\em signal-to-tropical interference-plus-noise-ratio} (STINR) defined as
\begin{equation}
     \mbox{STINR}(t):= \frac{\rho \vert|X_0^t\vert|^{-\a}}{\s^2+\max_{Y\in \Phi^t\setminus \{X_0^t\}} \rho_Y \vert|Y\vert|^{-\a}},
    \label{eq:STINRt}
\end{equation}
and similarly for the no-fading case. We can define {\em signal-to-tropical interference-ratio} (STIR) similarly, both in the fading and no-fading cases. The Shannon rate is the metric of interest here, which is proportional to $\ln(1+\mbox{SINR})$.
\subsection{Problem formulation: performance metrics}
In this article, building upon the theory developed  in the previous work~\cite{FB-SKJ}, we focus on the performance evaluation of the system using various metrics of interest. Suppose the user QoS requirement is $\tau>0$. At any time $t$, we say that a typical user is $\tau$-covered by the system with coverage probability
\begin{equation}
    p^c(\tau,\mu,\la,\a)(t):=\P(\Scal(t)>\tau),
\end{equation}
at time $t$, where $\Scal\in \{\mbox{SNR, SIR, SINR, STIR, STINR}\}$. In all these cases, we define the average Shannon rate as
\begin{align}
\Rcal(\mu,\la,\a)(t):=\E[B\ln(1+\Scal(t))] \mbox{ bps},
\label{eq:Rcal}
\end{align}
at time $t$, where $B=b/{\ln 2}$, with $b$ being the bandwidth. We use natural log $\ln$ in the definition of the average Shannon rate in (\ref{eq:Rcal}), instead of $\log_2$ for convenience in later computational tractability.

The main results in the article are about the coverage probability and average Shannon rate experienced by the user at the origin, at the different typical epochs alluded to above, namely, at a typical {\em min signal \& max interference (mS-MI) epoch}, {\em max signal (MS) epoch}, {\em max interference (MI) epoch}, {\em min interference (mI) epoch} and at any {\em typical time epoch}. All these scenarios are described in Figure~\ref{figure:BPP} using the collection of distance functions (\ref{eq:distancef}) described in~\cite{FB-SKJ}. We describe these scenarios in Table~\ref{tab:abb} in terms of proximity of the serving BS and all others, and also with in our seasonal analogy as follows:
\begin{enumerate}
\item \textbf{\em Min signal \& max interference (mS-MI) epoch}: At a  handover epoch the user experiences a comparable interference power component compared to the signal received from the serving station, which is a time of weakest signal and strongest interference.
\item \textbf{\em Max signal (MS) epoch}: At such an epoch, the power of the signal is maximal, as the serving BS is at its nearest position.  
\item \textbf{\em Max interference (MI) epoch}: At such an epoch, the nearest interferer is at the closest position with respect to the user. 
\item \textbf{\em Min interference (mI) epoch}: Such an epoch corresponds to the time when the nearest interferer gets swapped with another interferer. We will also call this event a {\em tropical interference handover}. 
\item \textbf{\em Typical time epoch (t)}: It correspond to the system at any other time epoch. At this time, there exists a BS at a Rayleigh distance. This leads to classical snapshot  analysis of the system as in~\cite{Andrews-etal}. 
\end{enumerate}
Table~\ref{tab:abb} summarizes the scenarios at different typical epochs, their correspondence with seasons, along with the notations for the variables for the distances used later in the signal and interference components. 
\begin{table}[ht!]
\centering
\begin{spacedtable}{1.2}{2pt}
\begin{tabular}{|c|c|c|c|c|}
\hline
Abbreviation & Meaning & Implication & Distance & Ext  \\
&{\em (Seasons)}&&variables&\\\hline 
{\em mS-MI} & Min signal and & Signal handover & $H_{\sst \Vcal}$ & $\sst{\Vcal}$   \\
& max interference  &  & $H_{\sst \Vcal}$, $\{D_i\}_i$ & \\
\hline
{\em MS} & Max signal & Nearest position & $H_{\sst S}$ & $\sst S$ \\
& & of the serving BS & $R_{\sst S}, \{D_i\}_i$ & \\
\hline
& & Nearest position & $R_{\sst I}$ & $\sst I$ \\ 
{\em MI} & Max interference & of the closest & $H_{\sst I}$, $\{D_i\}_i$ & \\
&&interferer &  & \\
\hline
& & Tropical & $R_{\sst \Ical}$ & $\sst \Ical$\\
{\em mI} & Min interference &interference & $H_{\sst \Ical}$, $H_{\sst \Ical}$, $\{D_i\}_i$& \\
&&handover & & \\
\hline
{\em t} & Typical time & -  & $R$ & $t$\\ 
& &  & $\{D_i\}_i$ & \\
\hline
\end{tabular}
\end{spacedtable}
\captionsetup{width=0.98\linewidth}
\caption{(a). The first three column together represent the abbreviated name, meaning and implication of each typical epoch. (b). In each cell of the variable column, the top variable corresponds to distance to the serving BS and the bottom variables correspond to distance to the interferers. (c). The last column contains the subscript corresponding to the typical epoch. A typical epoch, denoted using the subscript $t$, correspond to the system at any typical time.}
\label{tab:abb}
\end{table}
This article has three major parts. In the first part, we determine the probability distribution of several distance characteristics at various typical epochs, using the analysis in~\cite{FB-SKJ} and especially the Palm distribution, Theorem~5.4, therein. Next, using these distributions, we establish the coverage probability using a threshold parameter $\tau$ corresponding to the user QoS, under different metrics of interest. This also includes the average Shannon rate perceived by the user with respect to different metrics. The last part focuses on the comparison of coverage probabilities and average Shannon rates for various regime at different typical epochs. Moreover, this also includes a scale invariance property of the performance metrics in the interference limited regime. 

\subsection{Strategy}
The results in this article can be considered as a continuation of the theory developed in~\cite{FB-SKJ}. Along with the Palm probability distributions of quantities or distances useful for the distribution of SINR, in  
Section~\ref{section:prelims}, we identify and develop characteristics of few more ingredients of interest with the help of the spatio-temporal evolution of the distances of serving BS and other BSs. The distribution of these distances allows us to determine the distribution of SINR and Shannon rate at those typical events. Leveraging the properties of the Poisson point process, the probability generating functional (PGFl) of max-shot noise in particular, we provide integral form expressions or even closed form expressions in specific cases. We manage to compare the key performance indicators in terms of Laplace order or stochastic domination, with the help of the same ordering among some of the distances at typical epochs.  
\section{Preliminaries}~\label{section:prelims}
In the evaluation of the performance metrics at different time epochs, we need the probability distribution of system characteristics, for example, the distance to the serving station, to the nearest interferer(s) and that of all others at these epochs. This invites us to make use of the spatio-temporal evolution of the distance functions of all network participants, provided by the {\em radial bird particle process} developed in~\cite{FB-SKJ}. In the following subsections, we recall some of the constructions from there and devise new tools for deriving the distribution of the distances of interest at these different typical time epochs. Without loss of generality, we assume that $v=1$, since the distance distribution is invariant with respect to the speed parameter $v$, as seen in~\cite[Lemma~5.6]{FB-SKJ}, and also observed in (\ref{eq:headb}) later in this article. We borrow some of the notation and names from the paper~\cite{FB-SKJ}, or explicitly define some others otherwise. We consider this section as a preliminary to our main objective of this article. 
\subsection{Radial bird particle process}
Consider the particle process 
$\Hcal_c:= \sum_i\delta_{C_i}$, on the upper half plane $\H^+$ where $C_i:=\{\left(t,d_{X_i}(t)\right): t\in \R\}$, and $d_{X_i}(\cdot)$ is the distance function of the BS starting at $X_i$, for a homogeneous Poisson point process $\Phi_\la=\Phi= \sum_i\delta_{X_i}$. We call the infinite closed set $C_i$ a {\em radial bird} and 
\[
\Hcal_c:= \sum_i\delta_{C_i}
\]
the {\em radial bird particle process}. For each $X\in \Phi$, we define
\[
T_X:=\arg\inf_{t\in \R}\{d_X(t)\} \mbox{ and } H_X:=d_X(T_X).
\]
It turns out that
\begin{align}
(T_X, H_X) &= \left(-\vert|X|\vert \cos\Xi, \vert|X|\vert\,\vert\sin\Xi\vert \right),\nn
\end{align}
which is the time-space location corresponding to global minimum of the curve of $d_X(\cdot)$, where $\Xi$ is the relative angle of motion of the BS starting at $X$. In case of $v\neq 1$, we have 
\begin{align}
(T_X, H_X) &:=\left(-\frac{\vert|X|\vert}{v} \cos\Xi,\vert|X|\vert \,\vert\sin\Xi\vert \right).
\label{eq:headb}
\end{align}
Using~\cite[Lemma~4.3]{FB-SKJ}, $\Hcal:=\sum_i \delta_{(T_i,H_i)}$, forms a homogeneous Poisson point process on $\H^+$ with intensity measure $\nu$, where $(T_i,H_i)\equiv (T_{X_i},H_{X_i})$ and  ${\rm d}\nu:={\rm d} t\otimes 2\la\, {\rm d} h$. We call $\Hcal$ the {\em head point process}.  

We rewrite the radial bird particle process as
$\Hcal_c:= \sum_i\delta_{(T_i,H_i),C_i}$
by associating the infinite set $C_i$ to each point $(T_i,H_i)$ of $\Hcal$. The point process $\Hcal$ and the intricate properties of the distance function as a branch of a hyperbola work as the backbone of our analysis in the rest of the article. The following Figure~\ref{figure:BPP} represents a local realization of the radial bird particle process $\Hcal_c$, in which the intersection points on the vertical line at typical epochs represent the distance of the serving BS and the interferers.  
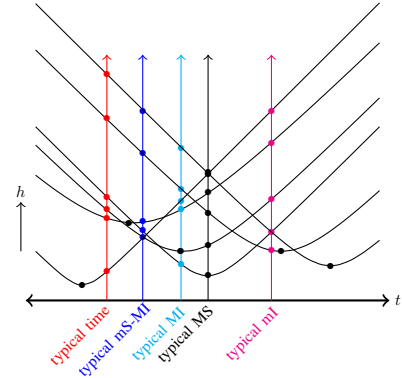
\begin{figure}[ht!]
    \centering
\begin{tikzpicture}[scale=1,every node/.style={scale=0.6}]
\pgftransformxscale{0.75}  
\pgftransformyscale{0.75}    
    \draw[thick,<->] (-2.2, 0) -- (5.2, 0) node[right] {$t$};
    \draw[->] (-2.3, 1) -- (-2.3, 2) node[above] {$h$};
    \draw[domain=-2:5, smooth] plot (\x, {(\x*\x+2.5)^0.5});
    \draw[](-0.1,1.57) node{$\bullet$};
    \draw[domain=-2:5, smooth] plot (\x, {(\x*\x-3*\x+2.5)^0.5});
    \draw[](1.51,0.5) node{$\bullet$};
    \draw[->] (1.51, 0) -- (1.51, 5);
    \node[rotate=45, anchor=south] at (1.25,-0.88) {typical MS};
    \draw[](1.51,1.12) node{$\bullet$};
    \draw[](1.51,1.78) node{$\bullet$};
    \draw[](1.51,2.2) node{$\bullet$};
    \draw[](1.51,2.57) node{$\bullet$};
    \draw[](1.51,2.60) node{$\bullet$};
    \draw[cyan, ->] (0.96, 0) -- (0.96, 5);
     \node[cyan,rotate=45, anchor=south] at (0.7,-0.88) {typical MI};
    \draw[cyan](0.96,0.73) node{$\bullet$};
    \draw[cyan](0.96,1) node{$\bullet$};
    \draw[cyan](0.96,1.85) node{$\bullet$};
    \draw[cyan](0.96,2.02) node{$\bullet$};
    \draw[cyan](0.96,2.25) node{$\bullet$};
    \draw[cyan](0.96,3.1) node{$\bullet$};
    \draw[blue, ->] (0.18, 0) -- (0.18, 5);
     \node[blue,rotate=45, anchor=south] at (-0.2,-1) {typical mS-MI};
    \draw[blue](0.18,1.28) node{$\bullet$};
    \draw[blue](0.18,1.42) node{$\bullet$};
    \draw[blue](0.18,1.6) node{$\bullet$};
    \draw[blue](0.18,3) node{$\bullet$};
    \draw[blue](0.18,3.85) node{$\bullet$};
    \draw[domain=-2:5, smooth] plot (\x, {(\x*\x-2*\x+2)^0.5});
    \draw[](0.95,1) node{$\bullet$};
    \draw[domain=-2:5, smooth] plot (\x, {(\x*\x+2.1*\x+1.2)^0.5});
    \draw[](-1.05,0.31) node{$\bullet$};
    \draw[domain=-2:5, smooth] plot (\x, {(\x*\x-6*\x+10)^0.5});
    \draw[](3,1) node{$\bullet$};
     \draw[magenta,->] (2.8, 0) -- (2.8, 5);
    \node[magenta,rotate=45, anchor=south] at (2.6,-0.88) {typical mI};
    \draw[magenta](2.8,1.02) node{$\bullet$};
    \draw[magenta](2.8,1.38) node{$\bullet$};
    \draw[magenta](2.8,2.05) node{$\bullet$};
    \draw[magenta](2.8,3.2) node{$\bullet$};
    \draw[magenta](2.8,3.85) node{$\bullet$};
    \draw[domain=-2:5, smooth] plot (\x, {(\x*\x-8*\x+16.5)^0.5});
    \draw[](4,0.7) node{$\bullet$};
     \draw[red, ->] (-0.55, 0) -- (-0.55, 5);
     \node[red,rotate=45, anchor=south] at (-0.9,-0.9) {typical time};
    \draw[red](-0.55,0.6) node{$\bullet$};
    \draw[red](-0.55,1.67) node{$\bullet$};
    \draw[red](-0.55,1.85) node{$\bullet$};
    \draw[red](-0.55,2.1) node{$\bullet$};
    \draw[red](-0.55,3.7) node{$\bullet$};
    \draw[red](-0.55,4.6) node{$\bullet$};
    \end{tikzpicture}
\captionsetup{width=0.9\linewidth}
    \caption{The blue nodes ($\blue{\bullet}$), black nodes ($\bullet$), red nodes ($\red{\bullet}$), cyan nodes ($\cyan{\bullet}$) and magenta nodes ($\magenta{\bullet}$), respectively, on individual vertical blue, black, red, cyan and magenta lines, form the point processes corresponding to the distance of all the mobile BSs at {\em min-signal max-interference}, {\em max-signal}, {\em typical time}, {\em max-interference} and {\em min-interference} epochs, respectively.}
    \label{figure:BPP}
\end{figure}
\subsection{Lower envelope}
The lower envelope is defined as the random closed subset  
\[
\Lcal_e:=\{(t, L(t)): t\in \R\},
\]
of $\H$, where $L(t):=\inf_{i\in \N}\{d_{X_i}(t)\}$ for $t\in \R$. The lower envelope $\Lcal_e$ correspond to the lowest part of the radial bird particle process in Figure~\ref{figure:BPP}. Physically, $\Lcal_e$ represents the time evolution of the distance to the nearest BS from the user, which contributes to the signal component. It also encodes the information about the distance of the nearest interferer at handover events.
\begin{figure}[ht!]
    \centering
\begin{tikzpicture}[scale=1,every node/.style={scale=0.6}]
\pgftransformxscale{0.75}  
\pgftransformyscale{0.75}    
    \draw[thick,<->] (-2.2, 0) -- (5.2, 0) node[right] {$t$};
    \draw[->] (-2.3, 1) -- (-2.3, 2) node[above] {$h$};
    \draw[domain=-2:-0.25, smooth, green,very thick] plot (\x, {(\x*\x+2.5)^0.5});
    \draw[domain=-2:2.95, smooth] plot (\x, {(\x*\x+2.5)^0.5});
    \draw[](-0.1,1.57) node{$\bullet$};
    \draw[domain=0.2:0.5, smooth, green, very thick] plot (\x, {(\x*\x-3*\x+2.5)^0.5});
    \draw[domain=0.5:2.5, smooth, purple, very thick] plot (\x, {(\x*\x-3*\x+2.5)^0.5});
    \draw[domain=2.5:2.8, smooth, green, very thick] plot (\x, {(\x*\x-3*\x+2.5)^0.5});
    \draw[domain=-1.7:4.7, smooth] plot (\x, {(\x*\x-3*\x+2.5)^0.5});
    \draw[](1.51,0.5) node{$\bullet$};
    \draw[->] (1.51, 0) -- (1.51, 3.2);
    \node[rotate=0, anchor=south] at (1.45,-0.45) {MS};
    \draw[](1.51,1.12) node{$\bullet$};
    \draw[](1.51,1.78) node{$\bullet$};
    \draw[](1.51,2.2) node{$\bullet$};
    \draw[](1.51,2.57) node{$\bullet$};
    \draw[](1.51,2.60) node{$\bullet$};
    \draw[cyan, ->] (0.96, 0) -- (0.96, 3.2);
     \node[cyan,rotate=0, anchor=south] at (0.9,-0.45) {MI};
    \draw[cyan](0.96,0.73) node{$\bullet$};
    \draw[cyan](0.96,1) node{$\bullet$};
    \draw[cyan](0.96,1.85) node{$\bullet$};
    \draw[cyan](0.96,2.02) node{$\bullet$};
    \draw[cyan](0.96,2.25) node{$\bullet$};
    \draw[cyan](0.96,3.1) node{$\bullet$};
    \draw[blue, ->] (0.18, 0) -- (0.18, 3.2);
     \node[blue,rotate=0, anchor=south] at (0.1,-0.45) {mS-MI};
    \draw[blue](0.18,1.28) node{$\bullet$};
    \draw[blue](0.18,1.42) node{$\bullet$};
    \draw[blue](0.18,1.6) node{$\bullet$};
    \draw[blue](0.18,3) node{$\bullet$};
    %
    \draw[domain=-0.25:0.2, smooth, very thick, green] plot (\x, {(\x*\x-2*\x+2)^0.5});
     \draw[domain=0.2:0.47, smooth, very thick, purple] plot (\x, {(\x*\x-2*\x+2)^0.5});
    \draw[domain=0.47:2, smooth, very thick, green] plot (\x, {(\x*\x-2*\x+2)^0.5});
    \draw[domain=-2:4.3, smooth] plot (\x, {(\x*\x-2*\x+2)^0.5});
    \draw[](0.95,1) node{$\bullet$};
    \draw[domain=-2:0.2, smooth, purple, very thick] plot (\x, {(\x*\x+2.1*\x+1.2)^0.5});
    \draw[domain=0.2:0.26, smooth, green, very thick] plot (\x, {(\x*\x+2.1*\x+1.2)^0.5});
    \draw[domain=-2:2.2, smooth] plot (\x, {(\x*\x+2.1*\x+1.2)^0.5});
    \draw[](-1.05,0.31) node{$\bullet$};
     \draw[domain=2:2.5, smooth, green, very thick] plot (\x, {(\x*\x-6*\x+10)^0.5});
     \draw[domain=2.5:3.25, smooth, purple, very thick] plot (\x, {(\x*\x-6*\x+10)^0.5});
     \draw[domain=3.25:5, smooth, green, very thick] plot (\x, {(\x*\x-6*\x+10)^0.5});
    \draw[domain=-0.2:5, smooth] plot (\x, {(\x*\x-6*\x+10)^0.5});
    \draw[](3,1) node{$\bullet$};
     \draw[magenta,->] (2.8, 0) -- (2.8, 3.2);
    \node[magenta,rotate=0, anchor=south] at (2.8,-0.45) {mI};
    \draw[magenta](2.8,1.02) node{$\bullet$};
    \draw[magenta](2.8,1.38) node{$\bullet$};
    \draw[magenta](2.8,2.05) node{$\bullet$};
    \draw[magenta](2.8,3.2) node{$\bullet$};
    %
    \draw[domain=2.82:3.25, smooth, green, very thick] plot (\x, {(\x*\x-8*\x+16.5)^0.5});
     \draw[domain=3.25:5, smooth, purple, very thick] plot (\x, {(\x*\x-8*\x+16.5)^0.5});
    \draw[domain=0.8:5, smooth] plot (\x, {(\x*\x-8*\x+16.5)^0.5});
    \draw[](4,0.7) node{$\bullet$};
     \draw[red, ->] (-0.55, 0) -- (-0.55, 3.2);
     \node[red,rotate=0, anchor=south] at (-1.2,-0.48) {typical time};
    \draw[red](-0.55,0.6) node{$\bullet$};
    \draw[red](-0.55,1.67) node{$\bullet$};
    \draw[red](-0.55,1.85) node{$\bullet$};
    \draw[red](-0.55,2.1) node{$\bullet$};
    %
    \end{tikzpicture}
\captionsetup{width=0.9\linewidth}
    \caption{A realization of the 1st lower envelope (thick curve in purple curve) and 2nd lower envelope (thick curve in green).}
    \label{figure:envelope1-2}
\end{figure}
\subsection{The k-th lower envelope and its characterization}
Other than the lower envelope $\Lcal_e$, it will be important to look at the second layer or the $k$-th layer above the lower envelope. The second layer is defined as
    \[
    \Lcal^{(2)}_e:=\left\{\left(t, L^{(2)}(t)\right):t\in \R\right\},
    \]
    where,
    \[
    L^{(2)}(t):=\inf\big\{\{d_{X_i}(t)\}_{i\in \N}\setminus \{L(t)\}\big\},
    \]
and we call it the {\em second lower envelope}. This is obtained by peeling off the part of the  curves which lies on $\Lcal_e$, see Figure~\ref{figure:envelope1-2}. This object is interesting since the dominant contribution to the  interference power is given by the distances measured with respect to $\Lcal^{(2)}_e$. For $k\geq 2$ the $k$-th envelope is defined as
\[
\Lcal^{(k)}_e :=\left\{\left(t, L^{(k)}(t)\right):t\in \R\right\},
\]
where,
\[
L^{(k)}(t):=\inf\left\{\{d_{X_i}(t)\}_{i\in \N}\setminus \left\{L^{(i)}(t)\right\}_{1\leq i\leq k-1}\right\},
\]
where $L^{(1)}(t):=L(t)$ and $\Lcal^{(1)}_e:=\Lcal_e$. For $(s,h)\in \H^+$, $U_h^s$ denotes the open upper half ball of radius $h$ and center $(s,o)$. Let us recall two main geometric properties from~\cite[Subsection~4.2 and the Observation~4.17]{FB-SKJ} which can be summarized as follows:
\begin{enumerate}
\item \label{partial} Let $(s,h), (T,H)\in\mathbb{H}^+$. Let $C_{(T,H)}$ be the radial bird with its head at $(T,H)$. Suppose $(s,h)\in C_{(T,H)}$, for the set particle $C_{(T,H)}$ located at a point $(T,H)\in \Hcal$, corresponding to a BS starting at $X \in \Phi$. Then $(T,H)\in \partial U_h^s$, see Figure~\ref{figure:RB_properties} (picture~\subref{sfig:dist2head}).  
\begin{figure}[ht!]
\begin{subfigure}[t]{0.45\linewidth}
    \centering
\begin{tikzpicture}[scale=1,every node/.style={scale=0.7}]
\pgftransformxscale{0.9}  
    \pgftransformyscale{0.9}    
    \draw[->] (-2, 0) -- (3.2, 0) node[right] {$t$};
    \draw[-] (0, 0) -- (0, 1.58);
    \draw[-] (0, 0) -- (1.51, 0.5);
    \draw[blue, domain=-1:3, smooth] plot (\x, {(\x*\x-3*\x+2.5)^0.5});
    \draw[blue](1.51,0.5) node{$\bullet$};
    \draw[](2.2,0.35) node{$(T,H)$};
    \draw[red] (-1.58,0) arc (180:0:1.58);
    \draw[](0,-0.31) node{$(s,0)$};
    \draw[](0.3,1.9) node{$(s,h)$};
    \end{tikzpicture}
    \caption{$(s,h)\in C_{(T,H)} \Leftrightarrow (T,H)\in \partial U_h^s$.}
\label{sfig:dist2head}
    \end{subfigure}
    \qquad
      \begin{subfigure}[t]{0.45\linewidth}
    \centering
\begin{tikzpicture}[scale=1,every node/.style={scale=0.7}]
\pgftransformxscale{0.9}  
\pgftransformyscale{0.9}    
    \draw[->] (-2, 0) -- (3, 0) node[right] {$t$};
    \draw[<-] (0, 3) -- (0, 0) node[below] {$(s,0)$};
    \draw[<->] (0.1, 0) -- (0.1, 1.1);
    \draw[](0.25,0.5) node{$\hat{h}$};
    \draw[red, domain=-0.5:3.7, smooth] plot (\x, {(\x*\x-5*\x+2.5*2.5+0.2)^0.5});
    \draw[red](2.5,0.447) node{$\bullet$};
    \draw[](-0.6,1.8) node{$(s,h)$};
    \draw[](0,1.58) node{$\bullet$};
    \draw[](1,0.5) node{$U^{s}_{h}$};
    \draw[] (-1.58,0) arc (180:0:1.58);
   \draw[blue, domain=-2:1.5, smooth] plot (\x, {(\x*\x+2.1*\x+1.2)^0.5});
    \draw[blue](-1.05,0.31) node{$\bullet$};
    \draw[](-0.9,0.8) node{$(T,H)$};
    \end{tikzpicture}
    \caption{All birds with head point inside and outside the half-ball of radius $h$, intersect the vertical line $t=s$ below and above the level $h$, respectively, as presented in Property~\ref{in-out}}
    \label{sfig:in-out}
    \end{subfigure}
    \caption{Properties of the distance function.}
\label{figure:RB_properties}
\end{figure}
\item \label{in-out} 
Let $\hat{h}$ be the height at which the radial bird $C_{(T,H)}$ intersects the vertical line $t=s$. 
Then $(T,H)\in U_{h}^{s}$ if and only if $\hat{h}<h$, as depicted in Figure~\ref{figure:RB_properties} (picture~\subref{sfig:in-out}).
\end{enumerate}
\begin{remark}[$v\neq 1$]
Without loss of generality, we perform all our analysis for $v=1$. In case of $v\neq 1$, we have open upper half ellipse $E^{s,v}_h$, in place of the open upper half ball $U^s_h$, leading to slight change in the geometric arguments.
\end{remark}
This leads to the following characterization of points lying on the $k$-th lower envelope $\Lcal_e^{(k)}$, and the proof of which is given in Subsection~\ref{subsection:Lk}.
\begin{proposition}
Suppose $k\geq 1$. For any point $(s,h)\in \cup_{i\in \N}C_i$ that lies on $\Lcal^{(k)}_e$:
\begin{enumerate}
\item \label{nb} there exist exactly $(k-1)$ many head points in $U^s_h$, i.e., $\left\{(s,h)\in \Lcal^{(k)}_e\right\}= \left\{\Hcal(U^s_h)=k-1\right\}$,
\item \label{b} the fact that the point $(s,h)$ is an intersection of two radial birds, implies that $\left\{(s,h)\in \Lcal^{(k)}_e\right\}= \left\{\Hcal\left(\overline{U^s_h}\right)=k+1\right\}$.
\end{enumerate}
\label{proposition:Lk}
\end{proposition}
\subsection{Characteristics of min signal \& max interference (mS-MI)}
The time epochs of local minimum signal and local maximum interference (mS-MI) correspond to the handover point process, as seen in~\cite{FB-SKJ}. Therein the handover point process is defined as 
\begin{align}
\Vcal&:=\sum_{(T_i,H_i)\in \Hcal}\;\;\sum_{(T_j,H_j)\in \Hcal\,\mbox{:}\, T_j <  T_i} \delta_{\hat S}\, \one_{A(\hat S, \hat H)},
\label{eq:Vcal}
\end{align}
where $A(\hat S, \hat H):= \left\{\Hcal(U^{\hat S}_{\hat H})=0\right\}$ and $(\hat S,\hat H)$ is the intersection of two radial birds $C_i$ and $C_j$ corresponding to $(T_i,H_i)$ and $(T_j,H_j)$, respectively. The intensity of the handover point process $\Vcal$ is calculated in~\cite[Theorem~4.26]{FB-SKJ} as $\la_{\sst \Vcal}=\frac{4\sqrt{\la}}{\pi}$.

At a handover event, there exist two BSs at equal distance from the user. One of them corresponds to the serving BS and the other one to the nearest interferer. From~\cite[Lemma~5.6]{FB-SKJ}, the typical handover distance, denoted by $H_{\sst \Vcal}$, follows a Nakagami distribution with parameter  $\left(\frac{3}{2}, \frac{3}{2\la\pi}\right)$, with density  
\begin{equation}
f_{H_{\sst \Vcal}}(h)=4\pi\la^{3/2} h^2 e^{-\la\pi h^2},\text{ for } h\geq 0.
\label{eq:pdfhatH0}
\end{equation}
\subsection{Characteristics of max signal (MS)}
At a typical time of maximum signal power, the serving BS is at its closest distance to the user. We write the corresponding head point process $\Hcal_{\sst S}$ as
\begin{equation}
\Hcal_{\sst S}:=\sum_{j\in \N}\delta_{(T^{\sst S}_j, H^{\sst S}_j)} = \sum_{i\in \N}\delta_{(T_i, H_i)} \one_{\left\{(T_i, H_i)\in \Lcal_e \right\}}.\nn
\end{equation}
We denote the point process of max-signal epoch on the time axis by
\begin{equation}
\Vcal_{\sst S}:= \sum_{j\in \N\,\mbox{:}\, (T^{\sst S}_j, H^{\sst S}_j)\in \Hcal_{\sst S}}\delta_{T^{\sst S}_j} = \sum_{i\in \N: (T_i, H_i)\in \Hcal}\delta_{T_i} \one_{\left\{(T_i, H_i)\in \Lcal_e \right\}}.\nn
\end{equation}
The intensity of this point process is $\la_{\sst S}=\sqrt{\la}$, as derived in~\cite[Lemma~5.8]{FB-SKJ}. Let $H_{\sst S}$ denote the distance to the nearest position of a BS. It is known from~\cite[Lemma~5.11]{FB-SKJ} that $H_{\sst S}$ follows a Nakagami distribution with parameters $\left(\frac{1}{2}, \frac{1}{2\la\pi}\right)$, which has the probability density
\begin{align}
f_{H_{\sst S}}(h)= 2 \la^\half \; e^{-\la\pi h^2}, \text{ for } h\geq 0.\nn
\end{align}
Let $R_{\sst S}$ be the distance to the nearest interferer at a typical max-signal epoch. The following result is about the joint density of $H^2_{\sst S}$ and $R^2_{\sst S}$ and the proof of which is given in Appendix~\ref{subsection:JHRS}.
\begin{proposition}
For any $\beta, \gamma\geq 0$, the joint distribution of $H^2_{\sst S}$ and $R^2_{\sst S}$ is given by the joint Laplace transform 
\[
\E^0_{\Vcal_{\sst S}}\left[e^{-\g H_{\sst S}^2-\b R_{\sst S}^2}\right] {=} \left(1{+}\frac{\beta}{\la\pi}\right)^{-1}\!\! \left(1{+}\frac{\g+\beta}{\la\pi}\right)^{-1/2}.
\]
\label{proposition:JHRS}
\end{proposition}
\vspace{-0.2in}
The conditional probability density of $R_{\sst S}$, given $H_{\sst S}=h$ can be obtained using inverse Laplace transform. 
\begin{corollary}
The conditional pdf of the random variable $R_{\sst S}$ given $H_{\sst S}=h$ is
\[
f_{R_{\sst S}\vert H_{\sst S}=h}(r)=2\la\pi r  e^{-\la\pi(r^2-h^2)} \one_{\{r>h\}.}
\]
\end{corollary}
\begin{corollary}
Under the Palm probability measure with respect to $\Vcal_{\sst S}$, the random variable $R_{\sst S}$ follows a Nakagami distribution with parameters $\left(\frac{3}{2}, \frac{3}{2\la\pi}\right)$.
\label{cor:RS}
\end{corollary}
\subsection{Characteristics of a typical time}
Suppose $R$ denotes the distance to the serving BS at a typical time. Using Poissonianity, it is well known that $R$ is distributed as Rayleigh with pdf $f_R(r):=2\la\pi r \, e^{-\la\pi r^2}$.
\subsection{Characteristics of max interference (MI)}
Recall that the head point process is $\Hcal:=\sum_{i\in \N}\delta_{(T_i,H_i)}$. Suppose $\Hcal_{\sst I}\leq \Hcal$, i.e.,  $Supp(\Hcal_{\sst I})\subset Supp(\Hcal)$, is the collection of points from $\Hcal$ restricted to the second lower envelope $\Lcal^{(2)}_e$, i.e.,  $\Hcal_{\sst I}:=\Hcal\vert_{\Lcal^{(2)}_e}$. In this case, in short, we use $\sst I$ in the subscript or superscript in our notation, to mean max interference. We write the corresponding head point process $\Hcal_{\sst I}$ as
\begin{equation}
\Hcal_{\sst I}:=\sum_{j\in \N}\delta_{(T^{\sst I}_j, H^{\sst I}_j)} = \sum_{i\in \N}\delta_{(T_i, H_i)} \one{\left\{(T_i, H_i)\in \Lcal^{(2)}_e \right\}},
\label{eq:Hcal_I}
\end{equation}
where we use the notation $(T^{\sst I}_j, H^{\sst I}_j)$ for a head point on $\Lcal_e^{(2)}$. Observe that from Proposition~\ref{proposition:Lk} with $k=2$,
\begin{equation}
\left\{(T_i, H_i)\in \Lcal^{(2)}_e\right\}=\left\{\Hcal(U^{T_i}_{H_i})=1\right\},
\label{eq:Hcal_I1}
\end{equation}
see Figure~\ref{fig:max-T-Interference}. Also notice that, for $(T_i, H_i)\in \Lcal^{(2)}_e$, it is not necessary for the radial bird at $(T_i, H_i)$ to incur a handover. The sequence of points in $\Hcal_{\sst I}$ gives rise to a stationary point process, say $\Vcal_{\sst I}$ on $\R$, consisting of the abscissas of the points in $\Hcal_{\sst I}$. 
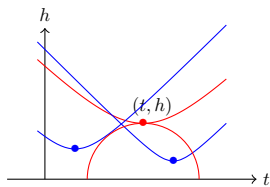
\begin{figure}[ht!]
\centering
\begin{center}
\begin{tikzpicture}
[scale=1, every node/.style={scale=0.7}]
\pgftransformxscale{0.28}  
\pgftransformyscale{0.28}   
\draw[->] (-9, 0) -- (7.5, 0) node[right] {$t$};
\draw[->] (-6.5, 0) -- (-6.5, 10) node[above] {$h$};
\draw[red, domain=-7:5.5, smooth] plot (\x, {(\x*\x+3.7^2)^0.5});
\draw[red](0,3.7) node{$\bullet$};   
\draw[](0.6,4.8) node{$(t,h)$};
\draw[blue, domain=-7:5.5, smooth] plot (\x, {(\x*\x-4*\x+2^2+ 1.2^2)^0.5});
\draw[blue](2,1.2) node{$\bullet$};    
\draw[blue, domain=-7:5.5, smooth] plot (\x, {(\x*\x+9*\x+4.5^2+2^2)^0.5});
\draw[blue](-4.5,2) node{$\bullet$}; 
\draw[red] (-3.7,0) arc (180:0:3.7);
\end{tikzpicture}
\captionsetup{width=0.85\linewidth}
\caption{The radial bird in red correspond to a BS that partly contributes the maximum interference to the user. The upper half ball $U^t_h$ contains exactly one point of $\Hcal$.}
\label{fig:max-T-Interference}
\end{center}
\end{figure}
Formally,
\begin{align}
\Vcal_{\sst I}&:= \sum_{j\in \N\,\mbox{:}\, (T^{\sst I}_j, H^{\sst I}_j)\in \Hcal_{\sst I}}\delta_{T^{\sst I}_j},\nn
\end{align}
is the point process of the abscissas of the points from $\Hcal_{\sst I}$ on the time axis, which we call the {\em max interference point process}. The time-stationarity of the point processes $\Hcal_{\sst I}$ and $\Vcal_{\sst I}$ is inherited from that of $\Hcal$. We can re-write $\Vcal_{\sst I}$ using (\ref{eq:Hcal_I}) and (\ref{eq:Hcal_I1}) as 
\begin{align}
\Vcal_{\sst I}& =  \!\!\!\!\!\!\!\sum_{(T_i, H_i)\in \Hcal} \!\!\!\!\!\!\delta_{T_i} \one{\left\{(T_i, H_i){\in} \Lcal^{(2)}_e\right\}} =  \!\!\!\!\!\!\! \sum_{(T_i, H_i)\in \Hcal} \!\!\!\!\!\! \delta_{T_i}\one{\left\{\Hcal(U^{T_i}_{H_i}){=}1\right\}}.
\label{eq:Acal_I2}
\end{align}
In the following result, we determine the intensity of the point process $\Vcal_{\sst I}$, proof of which is provided in the Appendix~(\ref{subsection:L-IVH}).
\begin{lemma}
    The intensity of the max interference point process $\Vcal_{\sst I}$ is $\la_{\sst I}=\frac{\sqrt \la}{2}$.
    \label{lem:IVH}
\end{lemma}
This enables us to define the Palm distribution with respect to the point process $\Vcal_{\sst I}$, equivalently with respect to a typical max interference (MI) event.
Let us denote the distance of the nearest interferer at typical epoch of {\em MI} as $H_{\sst I}$. In the following, we derive the Palm probability distribution of $H_{\sst I}$ with respect to $\Vcal_{\sst I}$. The proof of this is given in  Appendix~\ref{subsection:T-IVHd}.
\begin{proposition}
     The Palm probability distribution of $H_{\sst I}$ with respect to $\Vcal_{\sst I}$ is a Nakagami distribution with parameters $\left(\frac{3}{2}, \frac{3}{2\la\pi}\right)$.
    \label{prop:IVHd}
\end{proposition}
The density of $H_{\sst I}$ is given by
\[
f_{H_{\sst I}}(h)=4\pi\la^{3/2} h^2 e^{-\la\pi h^2}, \mbox{ for } h\geq 0.
\]
Let $R_{\sst I}$ be the distance of the serving BS under the Palm probability of $\Vcal_{\sst I}$, where $R_{\sst I}$ and $H_{\sst I}$ are dependent. 
In the following, we give the joint Laplace transform of $H^2_{\sst I}, R^2_{\sst I}$ characterizing their joint distribution. The proof is given in Appendix~\ref{subsection:T-JHRI}. 
\begin{proposition}
For any $\beta, \gamma\geq 0$, the joint distribution of $H^2_{\sst I}$ and $R^2_{\sst I}$ is given by the joint Laplace transform 
\[
\E^0_{\Vcal_{\sst I}}\left[e^{-\g H_{\sst I}^2-\b R_{\sst I}^2}\right] {=}\frac{2\la\pi}{\beta}\!\!\left[\!\left(1+\frac{\g}{\la\pi}\right)^{-\half}\!\!\!\!\!\!\!{-}\left(1+\frac{\g+\beta}{\la\pi}\right)^{-\half}\right].
\]
\label{proposition:JHRI}
\end{proposition}
\begin{corollary}
Conditioned on $H_{\sst I}=h$, the random variable $h^{-2}R^2_{\sst I}$ is uniformly distributed on $[0,1]$.
\label{corollary:RbyH}
\end{corollary}
\begin{corollary}
The marginal pdf of  $R_{\sst I}$ is given by
\begin{align}
    f_{R_{\sst I}}(r) &= 4\la\pi r \; \mathrm{erfc}(r\sqrt{\la\pi}),\nn
\end{align}
for $r\geq 0$, where $\mathrm{erfc}(z)=\frac{2}{\sqrt{\pi}}\int_z^\infty e^{-z^2}{\rm d}z$.
%
\end{corollary}
\subsection{Characteristics of min interference (mI)}
The tropical interference is the contribution from the second closest BS only. We are interested in analyzing the scenario when the tropical interference is minimum, and this happens at an epoch when the second closest BS is swapped. Two types of swaps can happen, namely, swaps given by pure handovers between nearest serving BS and nearest interferer, and swaps between second and third closest interferer. We call the first type a {\em signal handover} and second type a {\em tropical interference handover}. We denote the tropical interference handover point process by $\Vcal_{\sst \Ical}$. In this case, in short, we use $\sst \Ical$ in the subscript or superscript in some of our notations as it corresponds to tropical interference. 

We say that an intersection point $(s,h)$ of two radial birds is a tropical interference handover if $(s,h)\in \Lcal^{(2)}_e$. Using the characterization from Proposition~\ref{proposition:Lk}, the tropical interference handover point process is defined as
\begin{equation} 
\Vcal_{\sst \Ical}:=\sum_{(T_i,H_i)\in \Hcal}\;\;\sum_{(T_j,H_j)\in \Hcal\,\mbox{:}\, T_j <  T_i} \delta_{\hat S}\, \one_{A_1(\hat S, \hat H)},
\label{eq:VcalIHPP}
\end{equation}
where $A_1(\hat S, \hat H):= \{\Hcal(U^{\hat S}_{\hat H})=1\}$ similarly to (\ref{eq:Vcal}). The following result is about the intensity of the point process $\Vcal_{\sst \Ical}$, which is proved later in Appendix~\ref{subsection:L-VTI}.
\begin{lemma}
    The intensity of the tropical interference handover point process $\Vcal_{\sst \Ical}$ is $\la_{\sst \Ical}=\frac{6\sqrt{\la}}{\pi}$.
    \label{lem:VTI}
\end{lemma}
\begin{remark}
This result about the intensity of $\Vcal_{\sst \Ical}$, i.e., the frequency of tropical interference handovers is, $\la_{\sst \Ical}=\frac{6\sqrt{\la}}{\pi}$. It essentially derives the handover frequency in the 2-Voronoi tessellation of a Poisson point process. The $2$-Voronoi tessellation was studied in~\cite{Comp} from the perspective of Cooperative Multi-Point (CoMP) cellular networks. Our technique can be used to  generalize the results to the $k$-Voronoi tessellations for any $k>2$. 
\end{remark}
Let us denote the distance of a typical tropical interference handover as $H_{\sst \Ical}$. In the following, we give the Palm probability distribution of $H_{\sst \Ical}$ with respect to $\Vcal_{\sst \Ical}$ and the proof is given in Appendix~\ref{subsection:T-TIH}.
\begin{proposition}
     The Palm probability distribution of $H_{\sst \Ical}$ with respect to $\Vcal_{\sst \Ical}$ is a Nakagami distribution with parameters $\left(\frac{5}{2}, \frac{5}{2\la\pi}\right)$.
    \label{proposition:TIH}
\end{proposition}
The density of $H_{\sst \Ical}$ is given by
\[
f_{H_{\sst \Ical}}(h)=\frac{8}{3}\pi^2\la^{5/2} h^4 e^{-\la\pi h^2}, \mbox{ for } h\geq 0.
\] 
Let us denote the distance to the serving BS at this typical epoch by $R_{\sst \Ical}$, which is dependent on $H_{\sst \Ical}$. The following result is about the joint distribution of $H^2_{\sst \Ical}, R^2_{\sst \Ical}$ and the proof is given in Appendix~\ref{subsection:T-JHRTI}.
\begin{proposition}
For any $\beta, \gamma\geq 0$, the joint distribution of $H^2_{\sst \Ical}$ and $R^2_{\sst \Ical}$ is given by the joint Laplace transform 
\[
\E^0_{\Vcal_{\sst \Ical}}\left[e^{-\g H_{\sst \Ical}^2-\b R_{\sst \Ical}^2}\right]
{=}\frac{2\la\pi}{3\beta}\!\!\left[\!\left(1+\frac{\g}{\la\pi}\right)^{-\frac{3}{2}}\hspace{-0.17in} - \left(1+\frac{\g+\beta}{\la\pi}\right)^{-\frac{3}{2}}\right].
\]
\label{proposition:JHRTI}
\end{proposition}
\begin{corollary}
Conditioned on $H_{\sst \Ical}=h$, the random variable $h^{-2}R^2_{\sst \Ical}$ is uniformly distributed on $[0,1]$.
\label{corollary:RbyHIcal}
\end{corollary}
\begin{corollary}
The marginal pdf of  $R_{\sst \Ical}$ is given by
\begin{align}
    f_{R_{\sst \Ical}}(r) &= \frac{4\la\pi r}{3} \left(\mathrm{erfc}(r\sqrt{\la\pi})+ 2r \sqrt{\la} e^{-\la\pi r^2}\right), \mbox{ for } r\geq 0.\nn
\end{align}
\end{corollary}
%
%
\subsection{Comparison of distances: Laplace transform and MGF order}
\begin{definition}[Laplace transform order $\leq_{\textnormal{L}}$]
We say that two non-negative random variables $X,Y$ satisfy {\em Laplace order} 
\[
X\leq_{\textnormal{L}}Y, \;\;\mbox{ if }\;\; \E_X\left[e^{-\gamma X}\right] \geq  \E_Y\left[e^{-\gamma Y}\right],
\]
for any $\g\geq 0$. The inequality is flipped because of the negative exponent.    
\end{definition}
\begin{definition}[Moment generating function order $\leq_{\textnormal{mgf}}$]
We say that two non-negative random variables $X,Y$ satisfy moment generating function order ({\em MGF order}) 
\[
X\leq_{\textnormal{mgf}}Y, \;\;\mbox{ if }\;\; \E_X\left[e^{\gamma X}\right] \leq  \E_Y\left[e^{\gamma Y}\right],
\]
for any $\g\geq 0$.    
\end{definition}
Similarly to the Laplace order $H_{\sst \Vcal}\geq_{\textnormal{L}} R \geq_{\textnormal{L}}H_{\sst S}$,
form~\cite[Lemma~5.12]{FB-SKJ}, we also have the following orderings.
\begin{lemma}
The typical distances satisfy the following ordering
\begin{enumerate}
\item $H_{\sst\Ical}\geq_{\textnormal{L}} H_{\sst I}$ and  $R_{\sst\Ical}\geq_{\textnormal{L}} R_{\sst I}$,
\item $H_{\sst\Ical}\geq_{\textnormal{mgf}} H_{\sst I}$ and $R_{\sst\Ical}\geq_{\textnormal{mgf}} R_{\sst I}$,
\item  $H_{\sst \Vcal}\geq_{\textnormal{mgf}} R \geq_{\textnormal{mgf}}H_{\sst S}$.
\end{enumerate}
\label{lemma:Laplace_order}
\end{lemma}
\begin{definition}[Likelihood ratio dominance order $\leq_{\textnormal{lrd}}$]
We say that two non-negative random variables $X,Y$ satisfy  likelihood ratio dominance order ({\em lrd order}), 
\[
X\leq_{\textnormal{lrd}} Y \mbox{ if } \frac{f_Y(z)}{f_X(z)} \mbox{ is an increasing function }
\]
on their common support. Note that lrd order implies stochastic domination~\cite{ShakedShanthikumar2007}.
\end{definition}
The following lrd ordering will be useful in comparing the coverage probabilities in the tropical case, without fading.
\begin{lemma}
The typical distances satisfy the following ordering
$R_{\sst S} \stackrel{d}{=} H_{\sst \Vcal}\leq_{\textnormal{lrd}}R_1$ and $H_{\sst I}\leq_{\textnormal{lrd}} H_{\sst \Ical}$.
\label{lemma:LRD_order}
\end{lemma}
\subsection{Summary of densities of the typical distances}
For ease of reading and navigation, we provide the list (Table~\ref{tab:tdensity}) for  probability densities of the distances of interest and which will be essential in the performance evaluation later.
\begin{table}[ht!]
\centering
\begin{spacedtable}{1.7}{0.5pt}
\begin{tabular}{|c|c|c|}
\hline
Epoch & Variable & Distribution \& PDF\\  \hline
mS-MI& $H_{\sst \Vcal}$ & \textbf{$\mbox{Na}\left(\frac{3}{2}, \frac{3}{2\la\pi}\right)$},\quad  $f_{H_{\sst \Vcal}}(h)=4\pi\la^{3/2} h^2 e^{-\la\pi h^2}$  \\\hline
MS & $H_{\sst S}$ & \textbf{$\mbox{Na}\left(\frac{1}{2}, \frac{1}{2\la\pi}\right)$}, \quad $f_{H_{\sst S}}(h)=2\la^{1/2} e^{-\la\pi h^2}$  \\
& $R_{\sst S}$ &   $f_{R_{\sst S}\vert H_{\sst S}=h}(r)=  2\la\pi r  e^{-\la\pi(r^2-h^2)} \one_{\{r>h\}}
$\\
& & \textbf{$\mbox{Na}\left(\frac{3}{2}, \frac{3}{2\la\pi}\right)$},\quad  $f_{R_{\sst S}}(h)=4\pi\la^{3/2} h^2 e^{-\la\pi h^2}$ \\\hline
MI & $R^2_{\sst I}$ & \textbf{$\mbox{U}[0,H^2_{\sst I}]$},\quad  $f_{R^2_{\sst I}\vert H_{\sst I}=h}(r)= h^{-2}\one_{[0, h^2]}(r)$  \\
& $R_{\sst I}$ & $f_{R_{\sst I}}(r) = 4\la\pi r \; \mathrm{erfc}(r\sqrt{\la\pi})$\\
& $H_{\sst I}$ & \textbf{$\mbox{Na}\left(\frac{3}{2}, \frac{3}{2\la\pi}\right)$},\quad   $f_{H_{\sst I}}(h)=4\pi\la^{3/2} h^2 e^{-\la\pi h^2}$ \\ \hline
mI & $R^2_{\sst \Ical}$ & \textbf{$\mbox{U}[0,H^2_{\sst \Ical}]$},\quad   $f_{R^2_{\sst \Ical}\vert H_{\sst \Ical}=h}(r)= h^{-2}\one_{[0, h^2]}(r)$ \\
& $R_{\sst \Ical}$ & $f_{R_{\sst \Ical}}(r) = \frac{4\la\pi r}{3} \left(\mathrm{erfc}(r\sqrt{\la\pi})+ 2r \sqrt{\la} e^{-\la\pi r^2}\right)$\\
& $H_{\sst \Ical}$ & \textbf{$\mbox{Na}\left(\frac{5}{2}, \frac{5}{2\la\pi}\right)$},\quad $ f_{H_{\sst \Ical}}(h)=\frac{8}{3}\pi^2\la^{5/2} h^4 e^{-\la\pi h^2}$  \\\hline
t& $R$ & \textbf{Ray$\left(\frac{1}{\sqrt{2\la\pi}}\right)$},\quad   $f_R(h)=2\la\pi h \, e^{-\la\pi h^2}$  \\
& $R_1$ &   $f_{R_1\vert R=r}(h)=2\la\pi h \, e^{-\la\pi(h^2-r^2)}\one_{\{h\geq r\}}$   \\
&  &  $f_{R_1}(h)=2\la^2\pi^2 h^3 \, e^{-\la\pi h^2}$   \\\hline
\end{tabular}
\end{spacedtable}
\captionsetup{width=0.98\linewidth}
\caption{Densities of various typical distances. The notations (a). \textbf{$\mbox{Na}(m,\omega)$} stands for the Nakagami distributed random variable with parameters $(m,\omega)$, (b). \textbf{$\mbox{Ray}(\theta)$} stands for the Rayleigh  distributed random variable with parameter $\theta$, (c). $\mbox{\textbf{U}}[a,b]$ stands for the uniform random variable on $[a,b]$.}
\label{tab:tdensity}
\end{table}
\subsection{Distances of all other stations}
Given the distance $H \in \{H_{\sst \Vcal}, H_{\sst S}, H_{\sst I}, H_{\sst \Ical}, R\}$ at different typical epochs, let us define the distances of all other BSs can be represented as a one dimensional point process 
$\eta_H=\sum_i\delta_{D_i}$,
on $(H,\infty)$. The following result can be shown to hold using arguments similar to those in~\cite[Lemma~4.34]{FB-SKJ}:
\begin{lemma}
Conditioned on $H=h$, under the Palm probability measure of the respective typical epoch, the point process $\eta_h$ is a Poisson point process on $(h, \infty)$ with intensity measure of density $2\pi\la r\, {\rm d}r$.  
\label{lemma:etappp}
\end{lemma}
\section{A practical example in 3D}\label{section:3D}
Consider a non-terrestrial wireless communication model, where the mobile base stations are moving on the plane placed at an altitude $a$ and parallel to the ground plane. The model is the same, except for the user being on the ground and the base stations being on the other plane. This can be considered as a model for mobility of the unmanned aerial vehicles (UAVs) and  high altitude platform stations (HAPS) used for communication with the user on the ground. Note that the distance of a BS from the user at origin $O=(0,0,0)$ is given by $L=(R^2+a^2)^\half$, where $R$ denotes the distance to the BS from the origin $O_a=(0,0,a)$ on the upper plane. In other words, the distance function from the user turns out to be
\[
D_X(t):= \left(a^2+d_X(t)^2\right)^\half,
\]
for a BS located at $X$. The key observation is that the timings of the handover and other typical events with respect to $O_a$ remain the same with respect to $O$. Essentially, lifting the BSs to a plane at a higher altitude does not alter the timings of the events of interest, and so the distribution of the time point process of such events remain the same as in the previous sections. The only things that change are the typical distances, where we have to use the Pythagoras theorem. 
As a first step, one can redefine the head point process and the radial bird particle process as
\[
\Hcal^{(a)}:= \sum_i\delta_{(T_i,H_i)} \mbox{ and } \Hcal^{(a)}_c:= \sum_i\delta_{(T_i,H_i),C_i},
\]
adapting to this 3D scenario, where 
\[
(T_i, H_i)=\left(-\vert|X_i|\vert \cos\Xi, \left(a^2+\vert|X_i|\vert^2\,\sin^2\Xi\right)^\half\right)\mbox{ and }C_i\equiv C_{(T_i,H_i)}:=\{(t, D_{X_i}(t)): t\in \R\}.
\]
It can be shown that the head point process is a time homogeneous Poisson point process on $\mathbb H_a^+:=\R\times [a,\infty)$.  As we will see, this allows us to analyse a far more general model on 3D, without much change in the distribution of the distances developed so far, with respect to a change of variable. We give below a complete list of the distributions of the typical distances for this 3D model, introducing new notation for the typical distance variables, after the change of variable $L=(Z^2+a^2)^\half$, for $Z\in \{H_{\sst \Vcal}, H_{\sst S}, R_{\sst S}, R_{\sst I}, H_{\sst I}, R_{\sst \Ical}, H_{\sst \Ical}, R, R_1\}$: 
\begin{table}[ht!]
\centering
\begin{spacedtable}{1.7}{0.5pt}
\begin{tabular}{|c|c|c|}
\hline
Epoch & Variable & Distribution \& PDF\\  \hline
mS-MI& $L_{\sst \Vcal}$ &   $f_{L_{\sst \Vcal}}(\ell)=4\pi\la^{3/2} \ell (\ell^2-a^2)^\half e^{-\la\pi (\ell^2-a^2)}\one_{\{\ell\geq a\}}$  \\\hline
MS & $L_{\sst S}$ & $f_{L_{\sst S}}(\ell)=2\la^{1/2}\ell(\ell^2-a^2)^{-\half}  e^{-\la\pi (\ell^2-a^2)}\one_{\{\ell\geq a\}}$  \\
& $M_{\sst S}$ &   $f_{M_{\sst S}\vert L_{\sst S}=\ell}(m)=  2\la\pi m e^{-\la\pi(m^2-\ell^2)} \one_{\{m>\ell\}}
$\\
& &   $f_{M_{\sst S}}(m)=4\pi\la^{3/2} m (m^2-a^2)^\half e^{-\la\pi (m^2-a^2)}\one_{\{m\geq a\}}$ \\\hline
MI & $M^2_{\sst I}$ & \textbf{$\mbox{U}[a^2,L^2_{\sst I}]$},\quad  $f_{M^2_{\sst I}\vert L_{\sst I}=\ell}(m)= \frac{1}{\ell^2-a^2}\one_{[a^2, \ell^2]}(m)$  \\
& $M_{\sst I}$ & $f_{M_{\sst I}}(m) = 4\la\pi m \; \mathrm{erfc}(\sqrt{\la\pi(m^2-a^2)})\one_{\{m\geq a\}}$\\
& $L_{\sst I}$ &  $f_{L_{\sst I}}(\ell)=4\pi\la^{3/2} \ell (\ell^2-a^2)^\half e^{-\la\pi (\ell^2-a^2)}\one_{\{\ell\geq a\}}$ \\ \hline
mI & $M^2_{\sst \Ical}$ & \textbf{$\mbox{U}[a^2,L^2_{\sst \Ical}]$},\quad   $f_{M^2_{\sst \Ical}\vert L_{\sst \Ical}=\ell}(m)= \frac{1}{\ell^2-a^2} \one_{[a^2, \ell^2]}(m)$ \\
& $M_{\sst \Ical}$ & $f_{M_{\sst \Ical}}(m) = \frac{4\la\pi m}{3} \mathrm{erfc}(\sqrt{\la\pi(m^2-a^2)})\one_{\{m\geq a\}}$\\
& &  \hspace{0.5in}$+\frac{8\la\pi m}{3} \sqrt{\la\pi(m^2-a^2)} e^{-\la\pi (m^2-a^2)}\one_{\{m\geq a\}}$ \\
& $L_{\sst \Ical}$ & $ f_{L_{\sst \Ical}}(\ell)=\frac{8}{3}\pi^2\la^{5/2} \ell (\ell^2-a^2)^{3/2} e^{-\la\pi (\ell^2-a^2)}\one_{\{\ell\geq a\}}$   \\\hline
t& $M$ &  $f_M(m)=2\la\pi m \, e^{-\la\pi (m^2-a^2)}\one_{\{m\geq a\}}$  \\
& $M_1$ &   $f_{M_1\vert M=m}(m_1)=2\la\pi m_1 \, e^{-\la\pi(m_1^2-m^2)}\one_{\{m_1\geq m\}}$   \\
&  &  $f_{M_1}(m_1)=2\la^2\pi^2 m_1 (m_1^2-a^2) \, e^{-\la\pi (m_1^2-a^2)}\one_{\{m_1\geq a\}}$   \\\hline
\end{tabular}
\end{spacedtable}
\captionsetup{width=0.98\linewidth}
\caption{Densities of various typical distances. (a). Compared to the Table~\ref{tab:tdensity}, the $H$ variables are replace by $L$ and the $R$ variables are replaced by $M$, in the variable column. (b). $\mbox{\textbf{U}}[a,b]$ stands for the uniform random variable on $[a,b]$.}
\label{tab:tdensity-drone}
\end{table}
The distributions of the $L$ variables, namely $L_{\sst \Vcal}, L_{\sst S}, L_{\sst I} \mbox{ and } L_{\sst \Ical}$, and also $M_{\sst S}$ (computed from $R_{\sst S}$), can be referred to as the {\em spatially shifted Nakagami distributions} with their respective parameters. The distribution of $M$ is a {\em spatially shifted Rayleigh distribution}, corresponding to the random variable $R$, for the distance to nearest BS at typical time.
\section{Performance evaluation \& comparison}~\label{section:results}
For simplicity, let us assume that the both transmit power of each of the transmitter and the combined channel gain is $1$.

\subsection{System with fading}
\subsubsection{Performance metrics}
The main quantities of interest are SINR and its tropical version, STINR, at the typical  epochs, defined by 
\[
\mbox{SINR}:=\frac{S}{\s^2+I},\;   \mbox{STINR}:=\frac{S}{\s^2+\Tcal}, 
\]
aligning with the definition (\ref{eq:SINRt}) and (\ref{eq:STINRt}), where $\s^2$ is the thermal noise power, $S$ is the received signal power to the user at the corresponding typical epoch 
\begin{equation}
S:=\begin{cases}
    \rho H_{\sst \Vcal}^{-\a} & \text{ at typical mS-MI epoch},\\
     \rho H_{\sst S}^{-\a}& \text{ at typical MS epoch},\\
     \rho R_{\sst I}^{-\a} & \text{ at typical MI epoch },\\
     \rho R_{\sst \Ical}^{-\a} & \text{ at typical mI epoch},\\
     \rho R^{-\a} & \text{ at any typical epoch},
\end{cases} 
\label{eq:Int}
\end{equation}
using the random variable for the distance to the serving station, listed in Table~\ref{tab:abb}, and $I$ and $\Tcal$ are the additive and maxitive (or tropical) interference power.  Given $H= H_{\sst \Vcal}, H_{\sst S}, R_{\sst I}, R_{\sst \Ical}, R$, the interference $I$ and the tropical interference $\Tcal$ are defined in terms of the distances $H_{\sst I}$ or $H_{\sst \Ical}$ to the nearest interferer(s), and distances $\{D_i\}_{i\in \N}$ to all others, at the respective places in the proof of our results. The signal-to-interference-ratio (\textnormal{SIR}), signal-to-tropical interference-ratio (\textnormal{STIR}) and signal-to-noise-ratio (\textnormal{SNR}) are defined naturally. 

The factors $\rho$ used for the signal and $\{\rho_i\}_{i\in \N}$ that will be used for the interference correspond to the fading of the channels from the transmitters, which are assumed to be i.~i.~d. and distributed as Exp($\mu$), for $\mu>0$. The Table~\ref{table:CP_metrics} contains the formulas for the coverage probability metrics (SINR and STINR) in the scenario with fading.
\begin{table}[ht!]
\centering
\begin{spacedtable}{1.2}{2pt}
\begin{tabular}{|c|c|c|c|}
\hline%
Epoch & Variables & SINR (additive) & STINR (tropical)\\
 -&- & $(\bar I:= \sum \rho_i D_i^{-\a})$ &   $(\bar I:= \max \rho_i D_i^{-\a})$ \\\hline
{\em mS-MI} & $H_{\sst \Vcal}$ & $\frac{\rho H_{\sst \Vcal}^{-\a}}{\s^2+ \rho' H_{\sst \Vcal}^{-\a}+\bar I}$ & $\frac{\rho H_{\sst \Vcal}^{-\a}}{\s^2+\rho' H_{\sst \Vcal}^{-\a}\vee\bar I}$    \\
&   $H_{\sst \Vcal}, \{D_i\}_{i}$ & &  \\\hline
{\em MS} & $H_{\sst S}$ & $\frac{\rho H_{\sst S}^{-\a}}{\s^2+ \bar I}$ & $\frac{\rho H_{\sst S}^{-\a}}{\s^2+\bar I}$    \\
&   $\{D_i\}_{i}$ & &  \\\hline
{\em MI} & $R_{\sst I}$ & $\frac{\rho R_{\sst I}^{-\a}}{\s^2+\rho' H_{\sst I}^{-\a}+\bar I}$ & $\frac{\rho R_{\sst I}^{-\a}}{\s^2+ \rho' H_{\sst I}^{-\a} \vee\bar I}$    \\
&   $H_{\sst I},\{D_i\}_{i}$ &   & \\\hline
{\em mI} & $R_{\sst \Ical}$ & $\frac{\rho R_{\sst \Ical}^{-\a}}{\s^2{+}(\rho'{+}\rho'') H_{\sst \Ical}^{{-}\a}{+}\bar I}$ & $\frac{\rho R_{\sst \Ical}^{-\a}}{\s^2{+}(\rho'{\vee} \rho'') H_{\sst \Ical}^{{-}\a} {\vee} \bar I}$    \\
&   $H_{\sst \Ical},\{D_i\}_{i}$ & &   \\\hline
{\em t}  & $R$ & $\frac{\rho R^{-\a} }{\s^2+\bar I}$ & $\frac{\rho R^{-\a}}{\s^2+\bar I}$  \\
& $\{D_i\}_{i}$ &  &     \\\hline
\end{tabular}
\end{spacedtable}
\captionsetup{width=0.95\linewidth}
\caption{(a). First column represents the typical epochs. (b). In the second column, we just mention the variables corresponding to the signal and nearest interferer(s) and the distance of the rest of the interferer. (c). The third and fourth column contains the formula for the coverage probability metrics, namely SINR and STINR, in the additive and tropical interference scenario.}
\label{table:CP_metrics}
\end{table}

At all the typical epochs, the coverage probability and average Shannon rate are defined using $\Scal\in \{\mbox{SINR, STINR, SNR, SIR, STIR}\}$ as 
\[
p^c(\tau,\mu,\la,\a):= \P^0(\Scal\geq \tau),\]
and 
\[
\Rcal(\mu,\la,\a):=\E^0[B\ln(1+\Scal)],
\]
respectively, where $B=b\log_2e$, with $b$ being the bandwidth of the channel, which we assume to be a constant. Here $\P^0$ is the corresponding Palm probability measure and $\E^0$ is the Palm expectation with respect to $\P^0$. 
\begin{remark}[Shannon rate]
Note that for any metric $\Scal$, the average Shannon rate $\Rcal$ can be written in general in terms of the coverage probability as derived in~\cite{Andrews-etal}
\begin{align}
\Rcal &= B\int_0^\infty  \frac{1}{1+z} \P^0(\Scal\geq z) {\rm d}z. 
\label{eq:RcalG}
\end{align}
We treat this as the formula for the average Shannon rate, without stating explicitly in our results and hence it is enough to determine the coverage probabilities in all these cases in both cases with and without fading.
\end{remark}
In the following, we determine the coverage probability and the average Shannon rate using different performance metrics $\Scal\in \{\mbox{SNR, SIR, SINR, STIR, STINR}\}$, at typical mS-MI, MS and  typical time epoch, both in the non-tropical and tropical scenarios. We present the proof in Appendix~\ref{proofs-wf}.
\begin{theorem}
The coverage probability with respect to Palm probability distribution of the typical mS-MI, MS and typical time epoch is given by 
\begin{align}
p^c(\tau,\mu,\la,\a)&=  \E^0_H\left[e^{-\mu\tau H^{\a}\s^2} \Lcal_{I_{H}}(\mu\tau H^{\a})\right],
\label{eq:GCP}
\end{align}
where $H \in \{H_{\sst \Vcal}, H_{\sst S},R\}$.
\label{theorem:GPM-WF}
\end{theorem}
In the case of typical max-interference (MI) and typical min-interference (mI), there exists an interferer at distance $H \in \{H_{\sst I}, H_{\sst \Ical}\}$ and a serving BS is at distance $R\in \{R_{\sst I}, R_{\sst \Ical}\}$ such that $R\leq H$ a.s., and all other BSs are farther than $H$.  Subsequently the performance metrics are given by the following result, the proof of which is given in Section~\ref{proofs-wf}.
\begin{theorem}
The coverage probability with respect to Palm probability distribution of  typical mI and MI epoch is given by 
\begin{align}
p^c(\tau,\mu,\la,\a)&= \E^0_{H,R}\left[e^{-\mu\tau R^{\a}\s^2} \Lcal_{I_{H}}(\mu\tau R^{\a})\right],
\label{eq:GCP2}
\end{align}
where $R\in \{R_{\sst I}, R_{\sst \Ical}\}$, $H \in \{H_{\sst I}, H_{\sst \Ical}\}$.
\label{theorem:GPM-WF2}
\end{theorem}
\begin{remark}
In this case, the interference $I_H$ (or $\Tcal_H$ for tropical interference) is computed slightly differently with respect to the distance $H \in \{H_{\sst I}, H_{\sst \Ical}\}$ to the nearest interferer. All our computations are based on determining the Laplace transform $\Lcal_{I_{h}}(\g)$ of the interference or its tropical version $\Lcal_{\Tcal_{h}}(\g)$ for some parameter $\g$, given  $H=h$, the distribution of $H$, along with the help of Lemma~\ref{lemma:etappp}. 
\end{remark}
\subsection{System without fading}
The main quantities of interest are the SINR and its tropical version STINR, at the typical epochs, are given by 
\[
\mbox{SINR}:=\frac{S}{\s^2+I},
\;  \mbox{STINR}:=\frac{S}{\s^2+\Tcal}, 
\]
aligning with the definition (\ref{eq:SINRt}) and (\ref{eq:STINRt}), where the power of the received signal is defined as
\begin{equation}
S:=\begin{cases}
     H_{\sst \Vcal}^{-\a} & \text{ at typical mS-MI epoch},\\
     H_{\sst S}^{-\a}& \text{ at typical MS epoch},\\
      R_{\sst I}^{-\a} & \text{ at typical MI epoch },\\
      R_{\sst \Ical}^{-\a} & \text{ at typical mI epoch},\\
      R^{-\a} & \text{ at any typical epoch},
\end{cases} 
\label{eq:Int-nf}
\end{equation}
using the random variable for the distance to the serving station, listed in Table~\ref{tab:abb}, $H_{\sst \Vcal}, H_{\sst S}, R_{\sst I}, R_{\sst \Ical}, R$, for the corresponding cases, respectively. The additive interference $I$ and the maxitive or tropical interference $\Tcal$ are introduced similarly to the case with fading, at the respective places in the proof of our results. In the following, we determine the coverage probability and the average Shannon rate using different performance metrics  $\Scal\in \{\mbox{SINR, STINR, SNR, SIR, STIR}\}$, at different typical epochs of mS-MI, MS and typical time. For better readability, we present the following Table~\ref{tab:t-theorem-NF} mapping which theorem correspond to which cases of the typical epochs.  
\begin{table}[ht!]
\centering
\begin{spacedtable}{1.5}{5pt}
\begin{tabular}{|c|c|c|}
\hline
Epoch &  non-tropical & tropical \\\hline
{\em mS-MI} &  Theorem~\ref{theorem:GPM-NF} & Theorem~\ref{theorem:GPM-NF}     \\\hline 
{\em MS}  &  Theorem~\ref{theorem:GPM-NF} & Theorem~\ref{theorem:GPM-NF1S}     \\\hline 
{\em t}    & Theorem~\ref{theorem:GPM-NF} & Theorem~\ref{theorem:GPM-NF1S}     \\\hline 
{\em MI} & Theorem~\ref{theorem:GPM-NF2} & Theorem~\ref{theorem:GPM-NF2}     \\\hline 
{\em mI}  & Theorem~\ref{theorem:GPM-NF2} & Theorem~\ref{theorem:GPM-NF2}     \\\hline 
\end{tabular}
\end{spacedtable}
\captionsetup{width=0.95\linewidth}
\caption{Correspondence of the results and the typical epochs in the non-tropical and tropical cases in the environment without fading.}
\label{tab:t-theorem-NF}
\end{table}
\begin{theorem}
The coverage probability with respect to Palm probability distribution of epoch of interest, is given by 
\begin{align}
p^c(\tau,\la,\a)&= \E^0_{H}\left[F_{I_{H}}\left(0\vee (H^{-\a}/{\tau}-\s^2)\right)\right], 
\label{eq:GCP-NF}
\end{align}
where $H\in\{H_{\sst \Vcal}, H_{\sst S}, R\}$ in the non-tropical case and $H=H_{\sst \Vcal}$ in the tropical case, are the distance to the serving station.
\label{theorem:GPM-NF}
\end{theorem}
\begin{theorem}
The coverage probability with respect to Palm probability distribution of typical time and typical max-signal in the tropical case are given by 
\begin{align}
p^c(\tau,\la,\a)&= \E^0_{H, X}\left[F_{\Tcal}\left(0\vee (R^{-\a}/{\tau}-\s^2)\right)\right], 
\label{eq:GCP-NF2S}
\end{align}
for the tropical interference $\Tcal=X^{-\a}$, where $X\in \{R_1, R_{\sst S}\}$ and $H \in \{R, H_{\sst S}\}$.
\label{theorem:GPM-NF1S}
\end{theorem}
In the case of typical max-interference (MI) and typical min-interference (mI), there exists an interferer at distance $H \in \{H_{\sst I}, H_{\sst \Ical}\}$ and a serving BS at a distance $R\in \{R_{\sst I}, R_{\sst \Ical}\}$ such that $R\leq H$ almost surely. All other BSs are farther than $H$ and the locations of the latter are Poisson distributed. In this case, the interference $I_H$ (or $\Tcal_H$ for tropical interference) is computed slightly differently with respect to the distance $H$ to the nearest interferer. The performance metrics are given in:
\begin{theorem}
The coverage probability with respect to Palm probability distribution of a typical min-interference and max-interference in the non-tropical and tropical cases, are given by 
\begin{align}
p^c(\tau,\la,\a)&= \E^0_{H, R}\left[F_{I_{H}}\left(0\vee (R^{-\a}/{\tau}-\s^2)\right)\right], 
\label{eq:GCP-NF2}
\end{align}
\begin{align}
p^c(\tau,\la,\a)&= \E^0_{H, R}\left[F_{\Tcal}\left(0\vee (R^{-\a}/{\tau}-\s^2)\right)\right], 
\label{eq:GCP-NF2T}
\end{align}
where $R\in \{R_{\sst I}, R_{\sst \Ical}\}$, $H \in \{H_{\sst I}, H_{\sst \Ical}\}$ and $\Tcal=H^{-\a}$ is the tropical interference.
\label{theorem:GPM-NF2}
\end{theorem}
In all our computations, we derive the distribution $F_{I_h}$ or $F_{\Tcal_h}$ of the additive interference or tropical interference, respectively, given the distance $H=h$ in different typical epochs. Due to the absence of fading factor, the proofs are much simpler.
\begin{remark}
We derive the expression for the coverage probabilities at different typical epochs in Appendix~\ref{proofs-wf} and Appendix~\ref{section:PWFa}, including some closed form expressions in the special cases highlighted in Section~\ref{sec-closedforms}. The system performance for all other metrics, namely, \textnormal{SIR} and \textnormal{STIR}, can be derived by adapting the formulas (\ref{eq:GCP}), (\ref{eq:GCP2}), (\ref{eq:GCP-NF}) and (\ref{eq:GCP-NF2}) accordingly, by setting $\s=0$. 
On the other hand, the coverage probability in the \textnormal{SNR} regime can be derived directly from the distribution of the distance to the serving BS at the corresponding typical epochs of interest. 
\end{remark}
For ease of reading, we present a table of notation for the coverage probabilities and the average Shannon rates with respect to different time epochs in Table~\ref{tab:tnotation}, as continuation of Table~\ref{tab:abb} and Table~\ref{tab:tdensity}.
\begin{table}[ht!]
\centering
\begin{spacedtable}{1.2}{2pt}
\begin{tabular}{|c|c|c|c|c|c|}
\hline
 & & & Interference & Cov. prob. & Avg. Shannon rate \\
Epoch & Variables & Distribution & (additive) & (additive) & (additive)\\ 
& & (Table~\ref{tab:tdensity}) & (tropical) & (tropical) & (tropical) \\ \hline
{\em mS-MI} & $H_{\sst \Vcal}$ & $\mbox{\textbf{Na}}\left(\frac{3}{2}, \frac{3}{2\la\pi}\right)$ & $I_{H_{\sst \Vcal}}$ &$p_{\sst \Vcal}^c$ & $\Rcal_{\sst \Vcal}$    \\
&   $H_{\sst \Vcal}$ & & $\Tcal_{H_{\sst \Vcal}}$ & $p_{\sst{\Vcal;\Tcal}}^c$ & $\Rcal_{\sst{\Vcal;\Tcal}}$\\\hline
{\em MS}  & $H_{\sst S}$ & $\mbox{\textbf{Na}}\left(\frac{1}{2}, \frac{1}{2\la\pi}\right)$ & $I_{H_{\sst S}}$ & $p^c_{\sst S}$ & $\Rcal_{\sst S}$  \\
&  $R_{\sst S}$ & $\mbox{\textbf{Na}}\left(\frac{3}{2}, \frac{3}{2\la\pi}\right)$ & $\Tcal_{H_{\sst S}}$ & $p^c_{{\sst S};\sst{\Tcal}}$ & $\Rcal_{{\sst S};\sst{\Tcal}}$ \\\hline
{\em MI} & $R_{\sst I}$ & $\mbox{\textbf{U}}[0,H^2_{\sst I}]$ & $I_{H_{\sst I}}$ & $p^c_{\sst I}$ &  $\Rcal_{\sst I}$  \\ 
& $H_{\sst I}$ & $\mbox{\textbf{Na}}\left(\frac{3}{2}, \frac{3}{2\la\pi}\right)$ & $\Tcal_{H_{\sst I}}$ & $p^c_{{\sst I};\sst{\Tcal}}$& $\Rcal_{{\sst I};\sst{\Tcal}}$ \\ \hline
{\em mI} & $R_{\sst \Ical}$ & $\mbox{\textbf{U}}[0,H^2_{\sst \Ical}]$ & $I_{H_{\sst \Ical}}$ & $p^c_{\sst \Ical}$ & $\Rcal_{\sst \Ical}$ \\
& $H_{\sst \Ical}$, $H_{\sst \Ical}$ & $\mbox{\textbf{Na}}\left(\frac{5}{2}, \frac{5}{2\la\pi}\right)$ & $\Tcal_{H_{\sst \Ical}}$ & $p^c_{{\sst \Ical};\sst{\Tcal}}$  &  $\Rcal_{{\sst \Ical};\sst{\Tcal}}$ \\\hline
{\em t}  & $R$ & \mbox{\textbf{Ray}} $\left(\frac{1}{\sqrt{2\la\pi}}\right)$ & $I_R$ & $p^c_t$ & $\Rcal_t$  \\
& &  & $\Tcal_R$ & $p^c_{t;\sst{\Tcal}}$ & $\Rcal_{t;\sst{\Tcal}}$  \\\hline
\end{tabular}
\end{spacedtable}
\captionsetup{width=0.95\linewidth}
\caption{(a). In the $2$-nd column, we mention the variables corresponding to the signal and nearest interferer(s) and the third column reminds their distributions. (b). In the $4$-th column, $I$ stands for interference and $\Tcal$ for tropical interference. (c). The notation for coverage probability and average Shannon rate, $p^c_{*}, \Rcal_*$ denotes the non-tropical case and $p^c_{*;\sst{\Tcal}}, \Rcal_{*;\sst{\Tcal}}$ with an extra subscript $\sst{\Tcal}$ to denote the tropical case, for $*\in \{\sst{\Vcal,S,I,\Ical,t}$\hspace{-0.001em}$\}$.}
\label{tab:tnotation}
\end{table}
\begin{remark}[Alternative definition of HoPP $\Vcal$] 
In the no fading case, under the interference limited regime, observe that at a handover epoch, the \textnormal{STIR} is equal to $1$, whereas \textnormal{STIR}$(t)>1$ at any other time instant $t$, almost surely. It signifies that one can alternatively define the handover point process $\Vcal$ as 
\[
\Vcal:=\sum_{t\in \R} \delta_t\, \one{\left\{\mbox{\normalfont{STIR}}(t)=1\right\}}.
\] 
\label{remark:ADHO}
\end{remark}
\subsection{Comparison of performance metrics}~\label{subsection:comparison}

We now state another important result of this work that compares the coverage probability and the average Shannon rates among different typical epochs, the proof of which is given in Appendix~\ref{section:CPE}
\begin{theorem}
The following comparison of the coverage probabilities holds true for any QoS $\tau\geq 0$:
\begin{enumerate}
\item  \label{CP-F1} non-tropical with or without fading: 
$p_{\sst \Vcal}^c  \leq p_{t}^c \leq p_{\sst S}^c$ and $p^c_{\sst \Ical}  \leq p^c_{\sst I}$,
\item \label{CP-F2} tropical with or without fading: $p_{\sst{\Vcal;\Tcal}}^c  \leq p_{\sst{t;\Tcal}}^c \leq p_{{\sst S;\Tcal}}^c$ and $p^c_{{\sst \Ical;\Tcal}}  \leq p^c_{{\sst I;\Tcal}}$.
%
%
\end{enumerate}
\label{theorem:SINR-c}
\end{theorem}
%
%
\begin{remark}
The first results in part~\ref{CP-F1} and part~\ref{CP-F2} of Theorem~\ref{theorem:SINR-c} is a consequence of the Laplace order $H_{\sst \Vcal}\geq_{\textnormal{L}} R \geq_{\textnormal{L}}H_{\sst S}$.  
The second results in part~\ref{CP-F1} and part~\ref{CP-F2} may look a bit counter-intuitive, as in view of the Laplace transform order  $H_{\sst\Ical}\geq_{\textnormal{L}} H_{\sst I}$ shown in  Lemma~\ref{lemma:Laplace_order}, at the typical max interference epoch (MI) seem to exert lower coverage. The coverage is even lower in typical tropical interference handover (mI), because there are two equally strong interferers present at a typical tropical handover epoch. The later fact again reflects onto the comparison of rates as established in the second parts of both the results in Corollary~\ref{cor:SINR-rate}. 
\end{remark}
\begin{corollary}
The following comparisons of the average Shannon rates hold true:
\begin{enumerate}
    \item non-tropical with or without fading: $\Rcal_{\sst \Vcal}\leq \Rcal_{t} \leq \Rcal_{\sst S}$ and $\Rcal_{\sst \Ical}\leq  \Rcal_{\sst I}$,
    \item tropical with or without fading: $\Rcal_{\sst{\Vcal;\Tcal}}\leq \Rcal_{t;\sst{\Tcal}} \leq \Rcal_{{\sst S};\sst{\Tcal}}$ and $\Rcal_{{\sst \Ical};\sst{\Tcal}} \leq  \Rcal_{{\sst I;\Tcal}}$.
\end{enumerate}
\label{cor:SINR-rate} 
\end{corollary}
Let us now write the performance metrics as $\Scal_{*}$ for SINR and $\Scal_{\sst *;\Tcal}$ for STINR, where $*\in  \{\sst{\Vcal, S, I, \Ical, t}$\hspace{-0.01em}$\}$. A natural stochastic domination $\leq_{\textnormal{st}}$ among these metrics follows from Theorem~\ref{theorem:SINR-c}, when their distributions are compared under their respective Palm probability spaces.
\begin{corollary}
The following stochastic domination between \textnormal{SINR} and \textnormal{STINR} holds true for both the non-tropical and tropical cases:
\begin{enumerate}
    \item non-tropical with or without fading: $\Scal_{\sst \Vcal}\leq_{\textnormal{st}} \Scal_{t} \leq_{\textnormal{st}} \Scal_{\sst S}$ and $\Scal_{\sst \Ical}\leq_{\textnormal{st}}  \Scal_{\sst I}$,
    \item tropical with or without fading: $\Scal_{\sst{\Vcal;\Tcal}}\leq_{\textnormal{st}} \Scal_{t;\sst{\Tcal}} \leq_{\textnormal{st}} \Scal_{{\sst S};\sst{\Tcal}}$ and $\Scal_{{\sst \Ical};\sst{\Tcal}} \leq_{\textnormal{st}}  \Scal_{{\sst I};\sst{\Tcal}}$,
\end{enumerate}
compared under their respective Palm probabilities.
\label{cor:stS}
\end{corollary}
Note that for any of the  performance metrics $\Scal$, the Laplace transform of $\Scal$ can be written for any $\g\geq 0$ as
\begin{align}
\E^0_{\#}\left[e^{-\g \Scal}\right] & = \g\int e^{-\g x} F_{\Scal}(x){\rm d}x  = \g\int e^{-\g x} \left(1-p^c(x)\right){\rm d}x,
\label{eq:laplace-S}
\end{align}
where $\#\in  \{\sst{\Vcal, S, I, \Ical, t}$\hspace{-0.01em}$\}$, $p^c(x)=p^c_*(x,\mu,\la,\a)$ or $p^c_{\sst *;\Tcal}(x,\mu,\la,\a)$ for the system with fading and $p^c(x)=p^c_*(x,\la,\a)$ or $p^c_{\sst *;\Tcal}(x,\la,\a)$ for the system without fading and $*\in  \{\sst{\Vcal, S, I, \Ical, t}$\hspace{-0.01em}$\}$. Laplace order comparison of the performance metrics  $\Scal_{*}$ for SINR and $\Scal_{\sst *;\Tcal}$ for STINR follows from (\ref{eq:laplace-S}),  Theorem~\ref{theorem:SINR-c}, and also from the stochastic domination in Corollary~\ref{cor:stS}.
\begin{corollary}
The following Laplace orderings between \textnormal{SINR} and \textnormal{STINR}, compared under their respective Palm probability spaces, hold true for both the non-tropical and tropical cases:
\begin{enumerate}
    \item non-tropical with or without fading: $\Scal_{\sst \Vcal}\leq_{\textnormal{L}} \Scal_{t} \leq_{\textnormal{L}} \Scal_{\sst S}$ and $\Scal_{\sst \Ical}\leq_{\textnormal{L}}  \Scal_{\sst I}$,
    \item tropical with or without fading: $\Scal_{\sst{\Vcal;\Tcal}}\leq_{\textnormal{L}} \Scal_{t;\sst{\Tcal}} \leq_{\textnormal{L}} \Scal_{{\sst S};\sst{\Tcal}}$ and $\Scal_{{\sst \Ical};\sst{\Tcal}} \leq_{\textnormal{L}}  \Scal_{{\sst I};\sst{\Tcal}}$,
\end{enumerate}
\label{cor:Laplace-S}
\end{corollary}
\begin{remark}
The comparison of average Shannon rates in Corollary~\ref{cor:SINR-rate} can be seen as a consequence of the Laplace order in Corollary~\ref{cor:Laplace-S}, with the help of Hamdi's lemma~\cite{Hamdi} stated as
\[
\E\left[\ln(1+X)\right] = \int_0^\infty \frac{1}{z} \left(1-\E\left[e^{-zX}\right]\right) e^{-z}{\rm d}z,
\]
for any non-negative random variable $X$, using the analytical fact that $\int_0^\infty \frac{1}{z} \left(1-e^{-zx}\right) e^{-z}{\rm d}z= \ln(1+x)$.
\end{remark}
\subsection{Comparison of interference power: Stochastic domination \& Laplace transform order}
Let us denote the additive interference power as $I_{*}$ and maxitive interference power as $\Tcal_{*}$, for $*\in  \{\sst{\Vcal, S, I, \Ical, t}$\hspace{-0.01em}$\}$. Then we have the following ordering, the proof of which is given in Appendix~\ref{Int-order}.

\begin{theorem}
The interference power in the non-tropical and tropical cases satisfies the following orderings:
\begin{enumerate}
    \item \label{Int-p1}non-tropical with or without fading: $I_{\sst \Vcal}\geq_{\textnormal{L}} I_{t} \geq_{\textnormal{L}} I_{\sst S}$ and $I_{\sst \Ical}\geq_{\textnormal{L}}  I_{\sst I}$,
    \item \label{Int-p2} tropical with fading: $\Tcal_{\sst{\Vcal}}\geq_{\textnormal{st}} \Tcal_{t} \geq_{\textnormal{st}} \Tcal_{\sst S}$ and $\Tcal_{\sst \Ical} \geq_{\textnormal{st}}  \Tcal_{\sst I}$,
    \item \label{lrd} tropical without fading: $\Tcal_{\sst{\Vcal}}\stackrel{d}{=} \Tcal_{\sst S} \geq_{\textnormal{lrd}} \Tcal_{t}  $ and $\Tcal_{\sst I} \geq_{\textnormal{lrd}}  \Tcal_{\sst \Ical}$,
\end{enumerate}
compared under their respective Palm probability spaces.
\label{theorem:Int-order}
\end{theorem}
\begin{remark}
These orderings, in part~\ref{Int-p1} and part~\ref{Int-p2} of Theorem~\ref{theorem:Int-order}, are natural and remain consistent with the ordering among the coverage probabilities at different typical epochs in Theorem~\ref{theorem:SINR-c}. The lrd ordering $\Tcal_{\sst S} \geq_{\textnormal{lrd}} \Tcal_{t}$ in part~\ref{lrd} of Theorem~\ref{theorem:Int-order} can be proved as a corollary of the fact that $R_1 \geq_{\textnormal{lrd}} R_{\sst S}$ in   Lemma~\ref{lemma:LRD_order}, since $\Tcal_{\sst S}= R_{\sst S}^{-\a}$, $\Tcal_{t} = R_1^{-\a}$. The difference in the ordering in part~\ref{lrd} mainly stems on the fact that the tropical interference in this case solely depends on distance to the nearest interferer, directly inheriting the properties of the distance variables. On the other hand, in part~\ref{Int-p1} and part~\ref{Int-p2} for the non-tropical case with or without fading and tropical case with fading, respectively, the interference depends on the distance to all interferers.
\end{remark}
\subsection{Scale invariance in the interference limited regime}~\label{subsection:SI}
We have the following scale invariance property of the coverage probability and average Shannon rate with respect to the intensity $\la$ of the BSs and the fading parameter $\mu$ (in an environment with fading), in the interference limited regime, i.e., $\s=0$. In all these cases, the ratio $\frac{H^{-\a}}{I_{H}}$ is scale invariant with respect to $\la$, where $H$ is the distance to the nearest BS and $I_H$ is the additive or maxitive interference. As a result we have the scale invariance property of the coverage probabilities at different typical epochs, in the interference limited regime. The proof is given in Appendix~\ref{section:SI}. 
\begin{theorem}
For any of the typical epochs, the coverage probabilities $p^c_{*}$ and $p^c_{*;\sst{\Tcal}}$, both in the non-tropical and tropical scenario, are scale invariant with respect to the intensity parameter $\la$ (also with respect to fading parameter $\mu$ in the scenario with fading), where $*\in \{\sst{\Vcal, S, I, \Ical, t}$\hspace{-0.001em}$\}$.
\label{theorem:si}
\end{theorem}
\begin{corollary}
For any of the typical epochs, the average Shannon rates $\Rcal_{*}$ and $\Rcal_{*;\sst{\Tcal}}$, both in the non-tropical and tropical scenario, satisfy the scale invariance property in the interference limited scenario, with respect to the intensity parameter $\la$ (also with respect to fading parameter $\mu$ in the environment with fading), where $*\in \{\sst{\Vcal, S, I, \Ical, t}$\hspace{-0.001em}$\}$.
\end{corollary}
\subsection{Attenuation functions: example and counter example}~\label{subsection:example-AF}
\subsubsection{\textbf{Bounded path-loss function}}~\label{baf} All our comparison results in Theorem~\ref{theorem:SINR-c}, Corollary~\ref{cor:SINR-rate}, Corollary~\ref{cor:Laplace-S} and Theorem~\ref{theorem:Int-order} can be shown to hold true for the bounded path-loss function $\ell(r)= \left(1+r\right)^{-\a}$, as in~\cite{AlAmmouri-etal}, where $\a\geq 2$ is the path-loss exponent. Our results can be verified along the same type of analysis, see Appendix~\ref{Example} for an instance. The type of attenuation function is very useful to eliminate the question of singularity of infinite signal power from a BS infinitesimally close to the user, in case of power law attenuation. In general, the bounded path-loss model naturally lacks the scale invariance property of the performance metrics in the interference limited regime, due to the extra additive term inside the attenuation function. 
\vspace{0.1in}
\subsubsection{\textbf{Step attenuation function}}~\label{step-fun} On the other hand, the comparison results do not hold true for other attenuation functions, for example $\ell(r)= p\, \one_{\{r\leq d\}}$, where $p$ is a constant power and $d$ is a fixed cut-off distance. For example, a disordered behavior, depending on the user QoS $\tau$, among $p_{\sst \Vcal}^c, p_{t}^c,  p_{\sst S}^c$ and $p^c_{\sst \Ical}, p^c_{\sst I}$ with additive interference, and also for the tropical interference case, is shown in Appendix~\ref{Example}. As we will see therein, we also loose the scale invariance property, in this case.   
\section{Closed forms for coverage probabilities}~\label{sec-closedforms}
Much of our computations are focused on the determination of closed form or simple integral form expressions for the coverage probabilities at various typical epochs. We list out all such expressions in different regimes and typical epochs, with or without fading, and having specific values of the path-loss exponent $\a$. Nevertheless, the coverage probabilities in all these cases can be simulated numerically.
\subsection{With fading} These closed forms are mainly for SINR, SIR and SNR regimes in the presence of fading. 
\subsubsection{\textnormal{SINR} regime}
In the \textnormal{SINR} regime, for the specific case of $\a=2$ we have $p^c_*=0$, in any of the typical epochs, i.e., $*\in \{\sst{\Vcal, S, I, \Ical, t}$\hspace{-0.001em}$\}$. For $\a=4$ we have 
\begin{align}
p^c_* &=\frac{1}{(1+\tau)^{\one_{\{*=\Vcal\}}}}\E^0_{H_*}\left[e^{-\mu\tau H_*^4\s^2 -\pi\la H_*^2 \kappa(\tau,4)} \right],\nn
\end{align}
where $\kappa(\tau,4)= \tau^{\half} \arctan(\tau^{\half})$ and $*\in \{\sst{\Vcal, S, t}$\hspace{-0.001em}$\}$. In the other two typical epochs,
\begin{align}
 p^c_*&{=} \frac{\tau^{-1/4}}{4}\int_0^\tau  \!\!\!\frac{z^{-3/4}}{(1{+}z)^{1+\one_{\{*=\sst \Ical\}}}}\E^0_{H_*}\!\! \left[e^{-\mu z H_*^4 \s^2-\pi\la H_*^2  \kappa(z, 4)}\right]{\rm d}z,\nn
\end{align}
where $\kappa(z,4)= z^{\half} \arctan(z^{\half})$ and $*\in \{\sst{I, \Ical}$\hspace{-0.001em}$\}$.
\subsubsection{\textnormal{SIR} regime} For $\a{=}2$, we have  $p^c_{*}{=}0$, at any typical epoch, i.e., $*\in \{\sst{\Vcal, S, I, \Ical, t}$\hspace{-0.001em}$\}$. For $\a=4$, we have
\[
p^c_{\sst \Vcal} = (1{+}\tau)^{-1}\left(1{+}\tau^{\half} \arctan(\tau^{\half})\right)^{-3/2}\!\!, \;\;\;  
p_{t}^c =\left(1{+}\tau^{\half} \arctan(\tau^{\half})\right)^{-1}\!\!,\;\;\; p_{\sst S}^c =\left(1{+}\tau^{\half} \arctan(\tau^{\half})\right)^{-1/2}\!\!, 
\]
\[
p^c_{\sst I}=\frac{\tau^{-1/2}}{2}\!\!\!\int_0^{\tau} \!\! \frac{z^{-1/2}}{1+z} \left(1{+}z^\half  \arctan(z^\half)\right)^{-3/2}{\rm d}z, \;\;\; 
p^c_{\sst \Ical}=\frac{\tau^{-1/2}}{2}\!\! \int_0^{\tau}\!\!\! \frac{z^{-1/2}}{(1{+}z)^2} \left(1{+}z^\half \arctan(z^\half)\right)^{-5/2}{\rm d}z.
\]
The formula of $p^c_t$ naturally matches that in~\cite{Andrews-etal}. Moreover, we have the following relations among the coverage probabilities in terms of $p^c_{\sst t}(\tau)$, in the interference limited regime, as
\[
p^c_{\sst \Vcal}(\tau)= \frac{(p^c_{\sst t}(\tau))^{3/2}}{1+\tau},\;\; p^c_{\sst S}(\tau)=  (p^c_{\sst t}(\tau))^{1/2},\]
\[
p^c_{\sst I}(\tau)=\frac{\tau^{-1/2}}{2}\!\!\!\int_0^{\tau} \!\! \frac{z^{-1/2}}{1+z} (p^c_{\sst t}(z))^{3/2}{\rm d}z, \;\;\;
p^c_{\sst \Ical}(\tau)=\frac{\tau^{-1/2}}{2}\!\!\!\int_0^{\tau} \!\! \frac{z^{-1/2}}{(1+z)^2} (p^c_{\sst t}(z))^{5/2}{\rm d}z,
\]
by writing them as a function of $\tau$ only.
\subsubsection{\textnormal{SNR} regime}
In particular, for the coverage probability in the \textnormal{SNR} regime with $\a=2$, we have from (\ref{eq:GCP}) that
\[
p^c_*= \E^0_H\left[e^{-\mu\tau H^2\s^2}\right] = \left(1+\frac{\mu\tau\s^2}{\la\pi}\right)^{-\zeta},
\]
where $*\in \{\sst{\Vcal, S, t}$\hspace{-0.001em}$\}$, $\zeta=3/2, 1/2, 1$ for $H= H_{\sst \Vcal}, H_{\sst S}, R$, using the Laplace transform of $H_{\sst \Vcal}^2, H_{\sst S}^2, R^2$, under their respective Palm probability spaces. Also from (\ref{eq:GCP2}),  Proposition~\ref{proposition:JHRI} and Proposition~\ref{proposition:JHRTI}, we have
\[
p^c_* {=}\E^0_H\left[e^{-\mu\tau H^2\s^2}\right]{=} \frac{2\la\pi}{\mu\tau\s^2}\!\!\left[1{-}\left(1{+}\frac{\mu\tau\s^2}{\la\pi}\right)^{-\zeta}\right],
\]
where $*\in \{\sst{I, \Ical}$\hspace{-0.001em}$\}$, $\zeta=1/2$ and $3/2$, for $H=R_{\sst I}$ and $R_{\sst \Ical}$, respectively.
\subsection{Without fading} In no-fading scenario, the closed form or integral form expressions we have are for any value of $\a$, and these are mainly for the STINR, STIR and SNR regimes.
\subsubsection{\textnormal{STINR} regime} 
\begin{align}
p^c_{\sst \Vcal;\Tcal}&= 
\begin{cases}
     F^0_{H_{\sst \Vcal}}\left(\left(\frac{1-\tau}{\tau\s^2}\right)^{\frac{1}{\a}}\right) & \text{ if } \tau<1\\
     0 & \text{ if } \tau\geq 1.
\end{cases}\nn
\end{align}
At typical max-interference and min-interference we have,
for $*\in \{\sst{I, \Ical}$\hspace{-0.001em}$\}$
\[
p^c_{*;\sst \Tcal}= \tau^{-\frac{2}{\a}}\E^0_{H_{*}}\left[\left(1+\s^2 H_{*}^{\a}\right)^{-\frac{2}{\a}}\right].
\] 
\subsubsection{\textnormal{STIR} regime}
For any $\tau\geq 0$
\[
p^c_{\sst \Vcal;\Tcal} =\begin{cases}
    1 \mbox{ if } \tau < 1\\
    0 \mbox{ if } \tau\geq  1,
\end{cases}
p^c_{\sst t;\Tcal}=
\begin{cases}
1 &\text{ if } \tau\leq 1\\
\tau^{-2/\a} & \text{ if }\tau>1,
\end{cases}
\;\; p^c_{\sst S;\Tcal} =
\begin{cases}
1 &\text{ if } \tau\leq 1\\
\tau^{-1/\a} & \text{ if }\tau>1.
\end{cases}
\]
For any $\tau\geq 1$, $p^c_{\sst I;\Tcal}= \tau^{-2/\a}= p^c_{\sst \Ical;\Tcal}$.
\begin{remark}
Based on the alternative definition of handover using \textnormal{STIR} in Remark~\ref{remark:ADHO}, and the coverage probability $p^c_{\sst \Vcal;\Tcal}$, we see that, at handover epochs, the user is covered almost surely, unless the QoS requirement is $\tau\geq 1$ and the user is not covered almost surely in that case.  
\end{remark}
\subsubsection{\textnormal{SNR} regime} The coverage probability, in general written as $p^c_*$, is given by
\[
p^c_* = \P^0\left(H\leq (\tau \s^2)^{-\frac{1}{\a}}\right),
\]
for $H= H_{\sst \Vcal}, H_{\sst S}, R_{\sst I}, R_{\sst \Ical}, R$ and $*\in \{\sst{\Vcal, S, I, \Ical, t}$\hspace{-0.001em}$\}$, under their respective Palm probability measures $\P^0$.

\section{Numerical results and discussions}\label{sec-numerical}
In this part we present the numerical illustrations of our results on the performance evaluation in terms of coverage probabilities and their comparison at different typical epochs, in both cases with and without fading. The closed form or semi explicit integral form derived by leveraging the Poissonianity, enable us to create plots of these performance metrics quite conveniently. The coverage probabilities are plotted against the user QoS requirement variable $\tau$ in dB scale. The system parameters are the intensity of BSs $\la$, the exponent $\a$ of the path-loss attenuation function, $\mu$ the fading parameter and $\s$ the parameter for the thermal noise. 

The plots and their comparisons are made for different $\la$ and $\a$ values. Figure~\ref{fig:HO-t-MS} plots coverage probabilities in the case with fading, against the SINR threshold $\tau$ in dB scale, at typical handover, typical time and typical max-signal, when $\a=3,4$, $\mu=2$, $\s^2=10^{-6}$, and $\la=10^{-4}$ (or $5\times 10^{-4}$ in some cases) corresponding to LEO constellation scale and $\la=1$ (or $10^{-1}$ in some cases) corresponding to cellular scale, respectively. Figure~\ref{fig:mIvsMI} is similar for the cases of typical max-interference and min-interference, in the case with fading.  A similar comparison of plots is also provided in Figure~\ref{fig:HO-t-MSNF} and Figure~\ref{fig:MI-mINF} with the same values of the parameters, in the case without fading. 

In the tropical cases, we have similar comparison plots of coverage probabilities in Figure~\ref{fig:T-HotMSWF} and Figure~\ref{fig:T-mIvsMIWF} in the case with fading, Figure~\ref{fig:T-HotMSNF} and Figure~\ref{fig:T-mIvsMINF} in the case without fading. All these plots validate our comparison of coverage probabilities stated in Theorem~\ref{theorem:SINR-c}, except for the comparison of coverage probabilities $p_{\sst \Vcal;\Tcal}^c, p_{\sst t;\Tcal}^c$ and $p_{\sst S;\Tcal}^c$ at typical handover, typical time and typical max-signal, respectively, for the case of tropical SINR without fading, as shown at the end of Appendix~\ref{section:CPE} and also validated in Figure~\ref{fig:T-HotMSNF}. Our plot shows a different comparison among $p_{\sst \Vcal;\Tcal}^c, p_{\sst S;\Tcal}^c$ and $ p_{\sst t;\Tcal}^c$ holds for the threshold parameter $\tau\geq 0$ [dB]. Performance is better at the typical time epoch for high density $(\la)$ of the BSs, which is not the case in low density regimes.

For brevity we have presented the numerical comparison just for the ordering among the coverage probabilities, presented in Theorem~\ref{theorem:SINR-c}, but not for the Shannon rates (Corollary~\ref{cor:SINR-rate}) and for the interferences (Theorem~\ref{theorem:Int-order}).
\begin{figure}[ht!]
    \centering
    \subfloat[]{\includegraphics[width=0.42\textwidth]{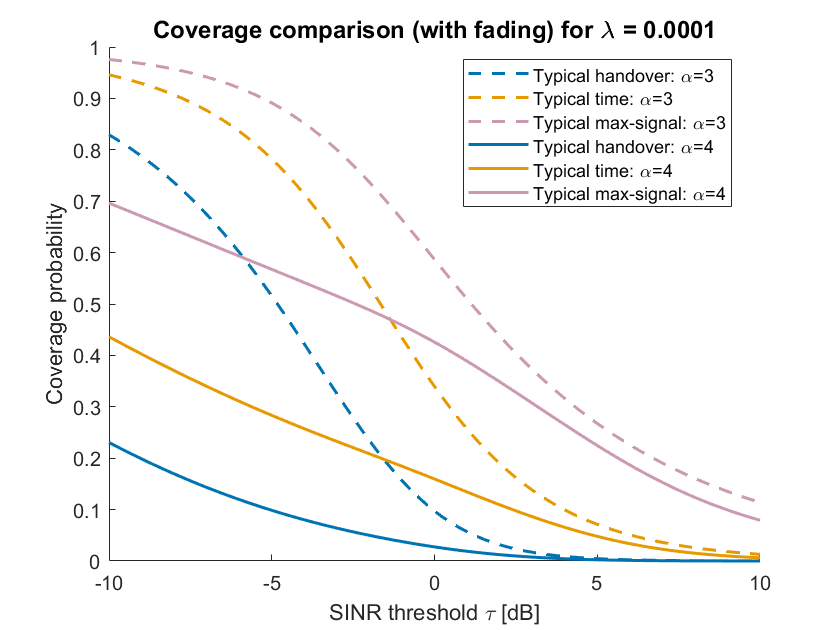}}
    \subfloat[]{
\includegraphics[width=0.42\textwidth]{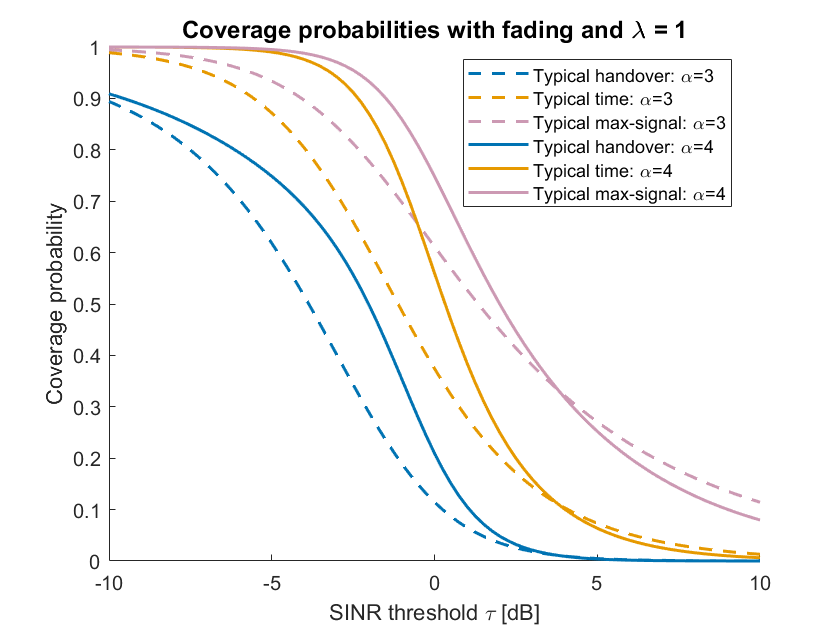}}
\caption{{\em Non-tropical case (with fading)}: Comparison of coverage probabilities at typical min-signal max-interference (handover), typical time, typical max-signal, with respect to SINR threshold or user QoS requirement $\tau$ in dB unit, in the case: (a) $\la=10^{-4}$ (LEO constellation scale), (b) $\la=1$ (cellular scale) for $\alpha=3$ and $4$ in the case with fading, between -10 to 10 dB.} 
\label{fig:HO-t-MS}
\end{figure}
\begin{figure}[ht!]
    \centering
    \subfloat[]{\includegraphics[width=0.42\textwidth]{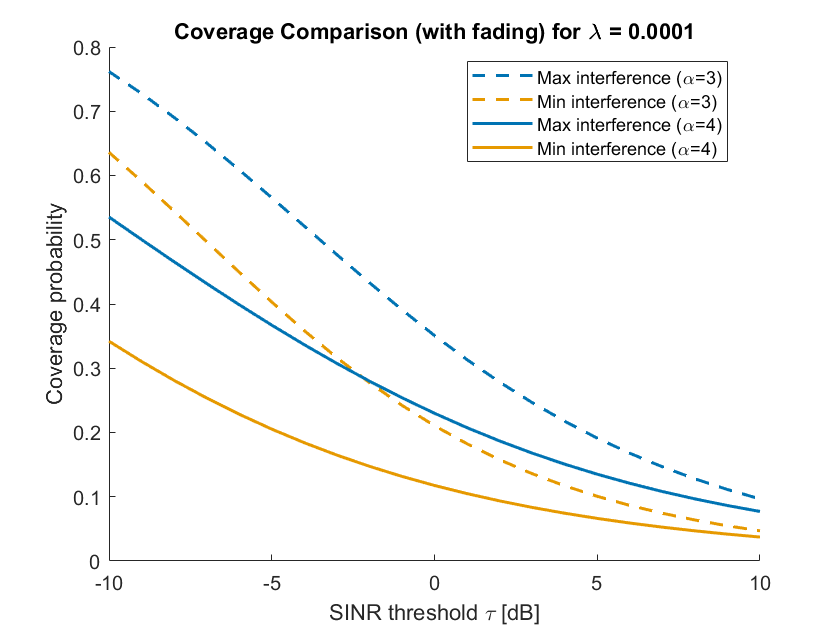}}
    \subfloat[]{
\includegraphics[width=0.42\textwidth]{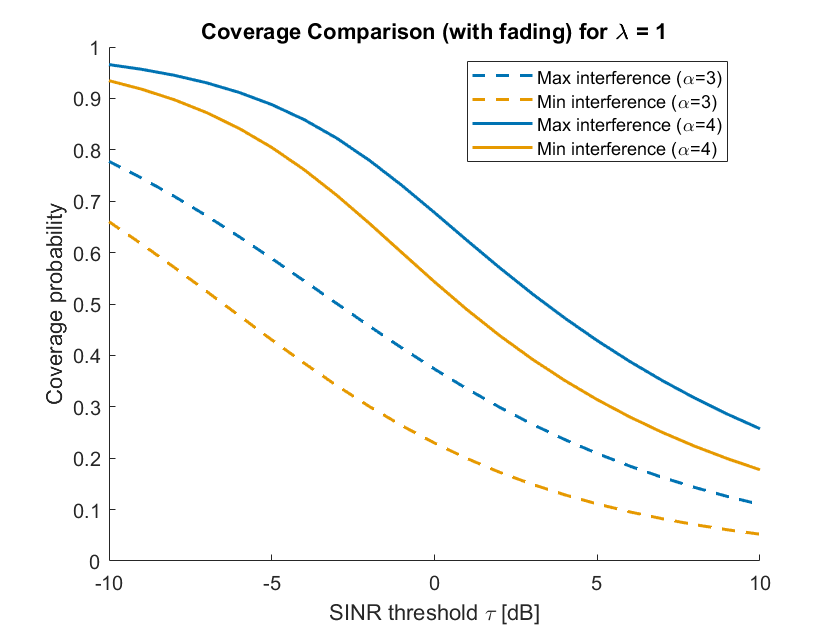}}
\caption{{\em Non-tropical case (with fading)}: Comparison of coverage probabilities at typical max-interference and typical min-interference, with respect to SINR threshold or user QoS requirement $\tau$ in dB unit, in the case: (a) $\la=10^{-4}$ (LEO constellation scale), (b) $\la=1$ (cellular scale) for $\alpha=3$ and $4$ in the case with fading, between -10 to 10 dB.} 
\label{fig:mIvsMI}
\end{figure}
\begin{figure}[ht!]
    \centering
    \subfloat[]{\includegraphics[width=0.42\textwidth]{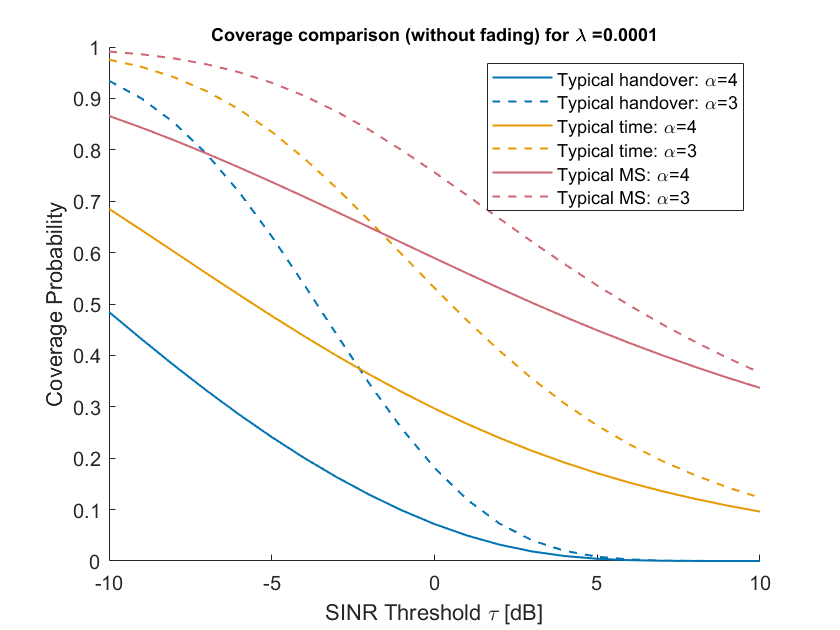}}
    \subfloat[]{
\includegraphics[width=0.42\textwidth]{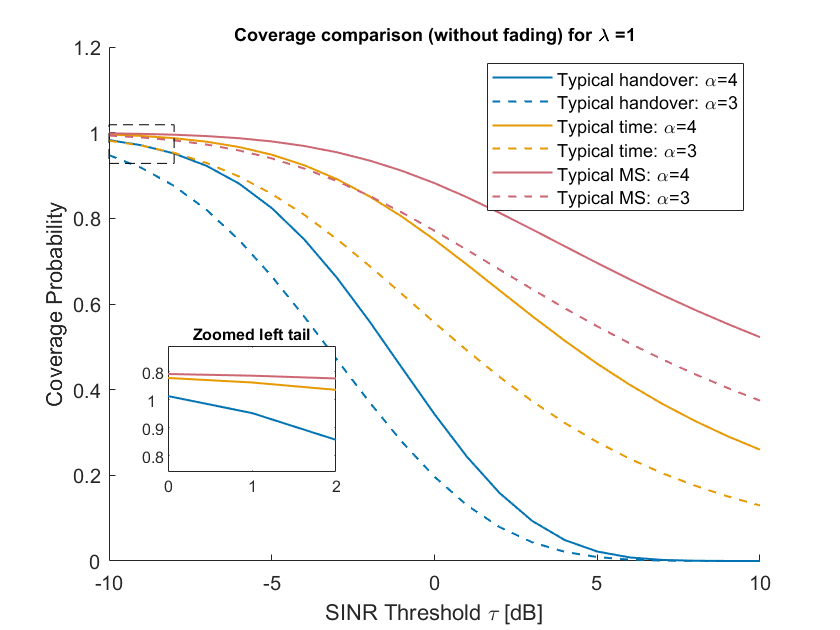}}
\caption{{\em Non-tropical case (without fading)}: Comparison of coverage probabilities at typical min-signal max-interference (handover), typical time and typical max-signal, with respect to SINR threshold or user QoS requirement $\tau$ in dB unit, in the case: (a) $\la=10^{-4}$ (LEO constellation scale), (b) $\la=1$ (cellular scale) for $\alpha=3$ and $4$ in the case without fading.} 
    \label{fig:HO-t-MSNF}
\end{figure}
\begin{figure}[ht!]
    \centering
    \subfloat[]{\includegraphics[width=0.42\textwidth]{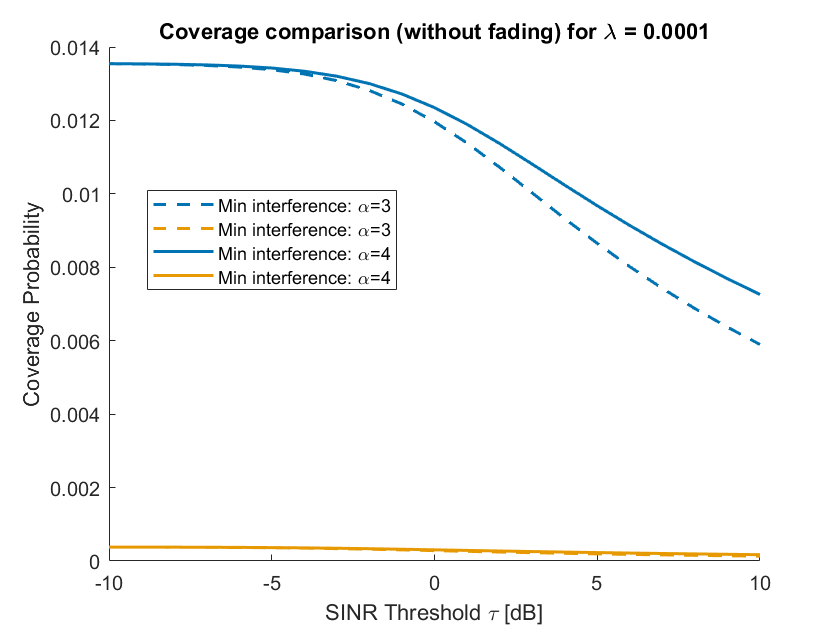}}
    \subfloat[]{
\includegraphics[width=0.42\textwidth]{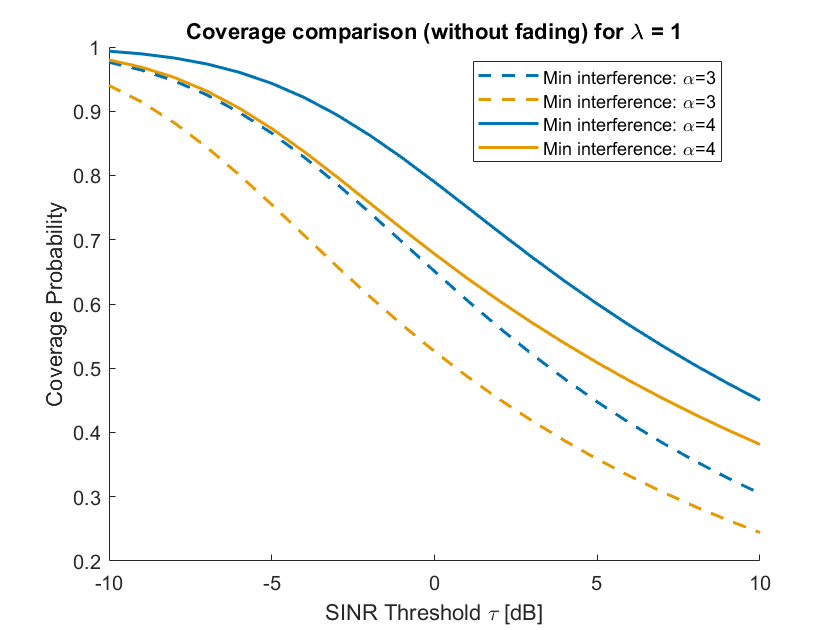}}
\caption{{\em Non-tropical case (without fading)}: Comparison of coverage probabilities at typical max-interference and typical min-signal, with respect to SINR threshold or user QoS requirement $\tau$ in dB unit, in the case: (a) $\la=10^{-4}$ (LEO scale), (b) $\la=1$ (cellular scale), for $\alpha=3$ and $4$ in the case without fading.} 
    \label{fig:MI-mINF}
\end{figure}
\begin{figure}[ht!]
    \centering
    \subfloat[]{\includegraphics[width=0.42\textwidth]{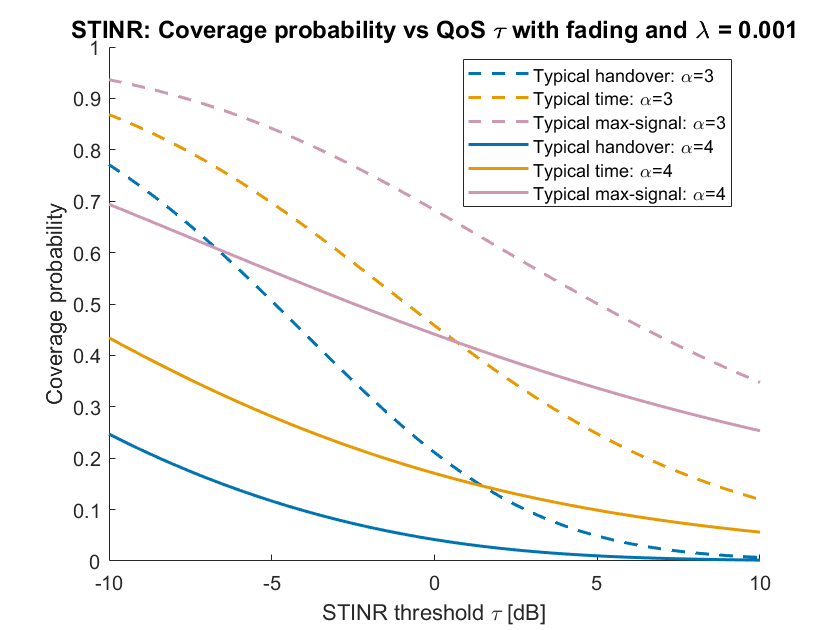}}
    \subfloat[]{
\includegraphics[width=0.42\textwidth]{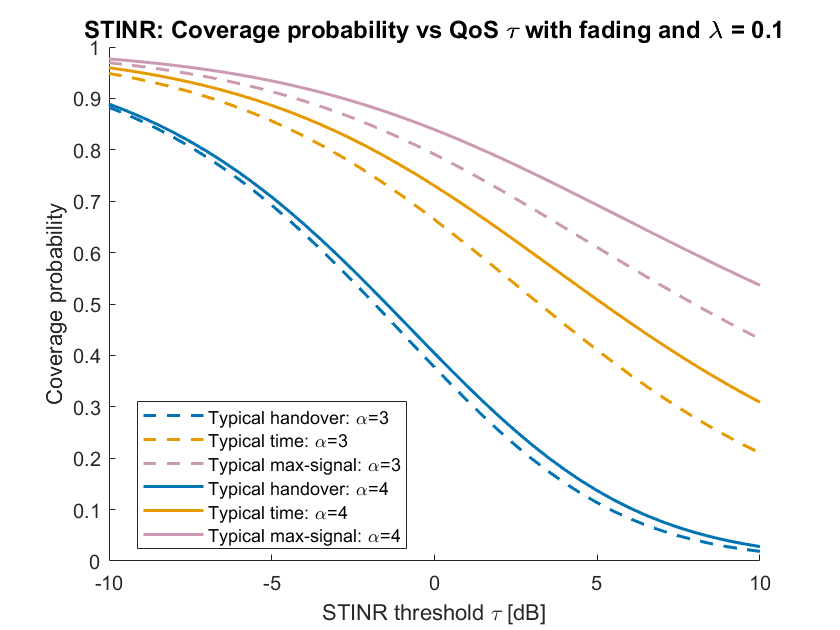}}
\caption{{\em Tropical case (with fading)}: Comparison of coverage probabilities at typical min-signal max-interference (handover), typical time and typical max-signal with respect to STINR threshold or user QoS requirement $\tau$ in dB unit, in the case: (a) $\la=10^{-3}$, (b) $\la=10^{-1}$ for $\alpha=3$ and $4$ in the case with fading.} 
\label{fig:T-HotMSWF}
\end{figure}
\begin{figure}[ht!]
    \centering
    \subfloat[]{\includegraphics[width=0.42\textwidth]{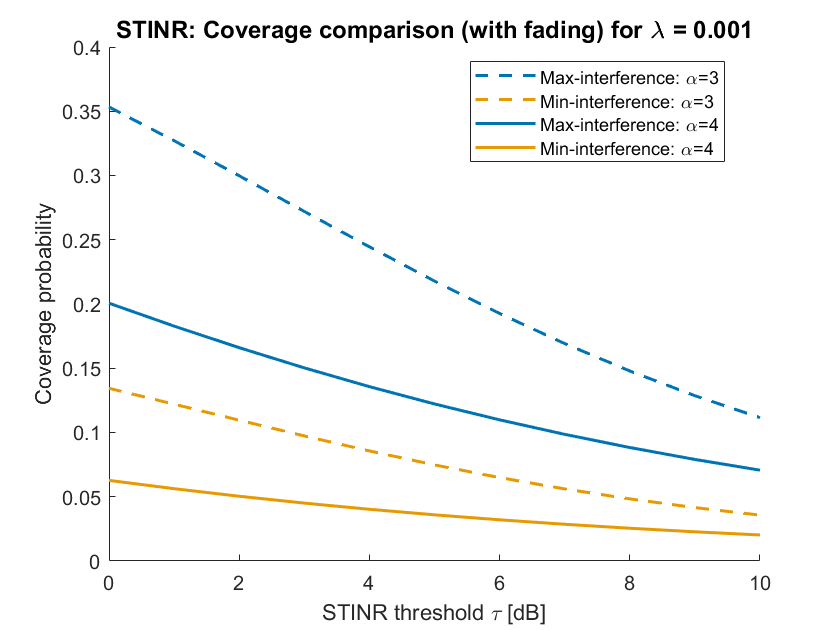}}
    \subfloat[]{
\includegraphics[width=0.42\textwidth]{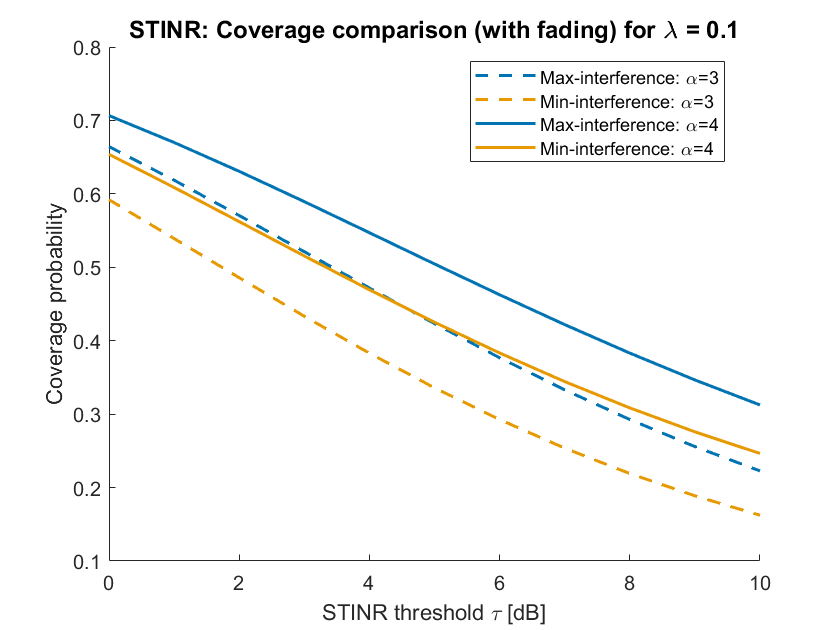}}
\caption{{\em Tropical case (with fading)}: Comparison of coverage probabilities at typical max-interference and typical min-interference, with respect to STINR threshold or user QoS requirement $\tau$ in dB unit, in the case: (a) $\la=10^{-3}$, (b) $\la=10^{-1}$ for $\alpha=3$ and $4$ in the case with fading.} 
\label{fig:T-mIvsMIWF}
\end{figure}
\begin{figure}[ht!]
    \centering
    \subfloat[]{
\includegraphics[width=0.42\textwidth]{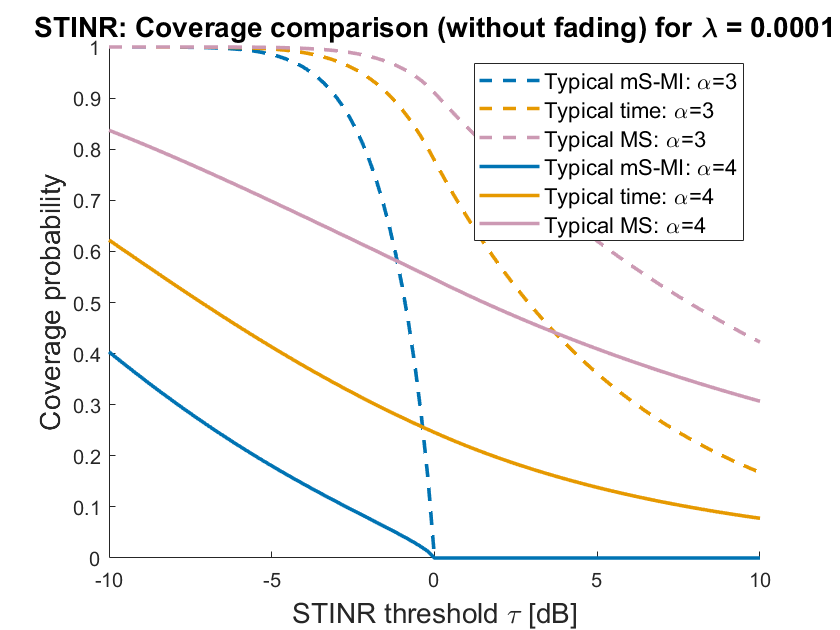}}
\subfloat[]{
\includegraphics[width=0.42\textwidth]{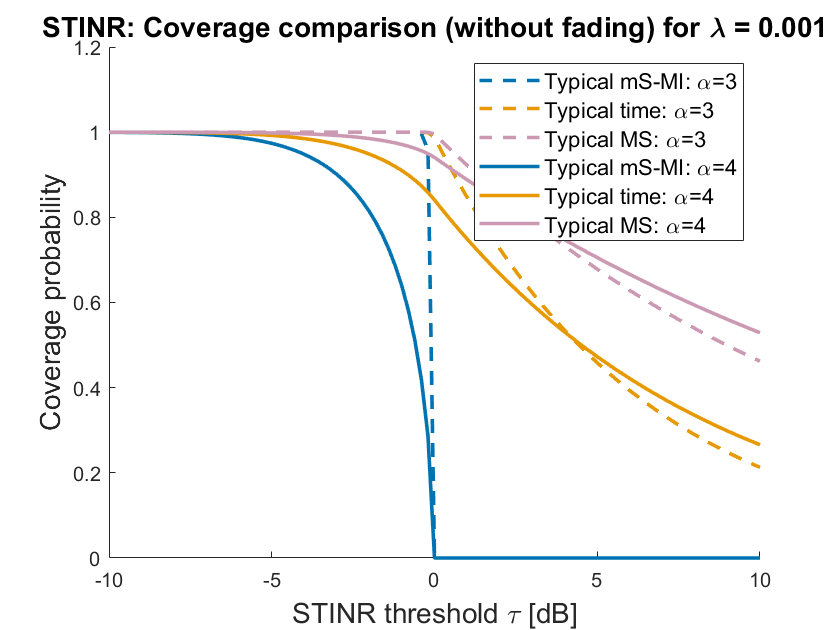}}
\caption{{\em Tropical case (without fading)}: Comparison of coverage probabilities at  typical min-signal max-interference (handover), typical time and typical max-signal, with respect to STINR threshold or user QoS requirement $\tau$ in dB unit, in the case: (a) $\la=10^{-5}$, (b) $\la=10^{-4}$, (c) $\la=10^{-3}$, (d) $\la=10^{-2}$ for $\alpha=3$ and $4$ in the case without fading.} 
\label{fig:T-HotMSNF} 
\end{figure}
\begin{figure}[ht!]
    \centering
    \subfloat[]{\includegraphics[width=0.42\textwidth]{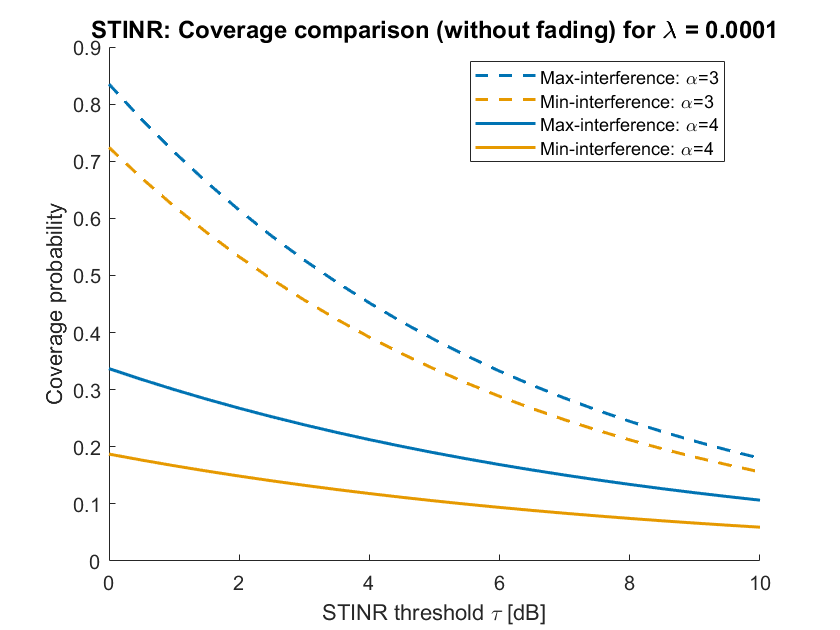}}
    \subfloat[]{
\includegraphics[width=0.42\textwidth]{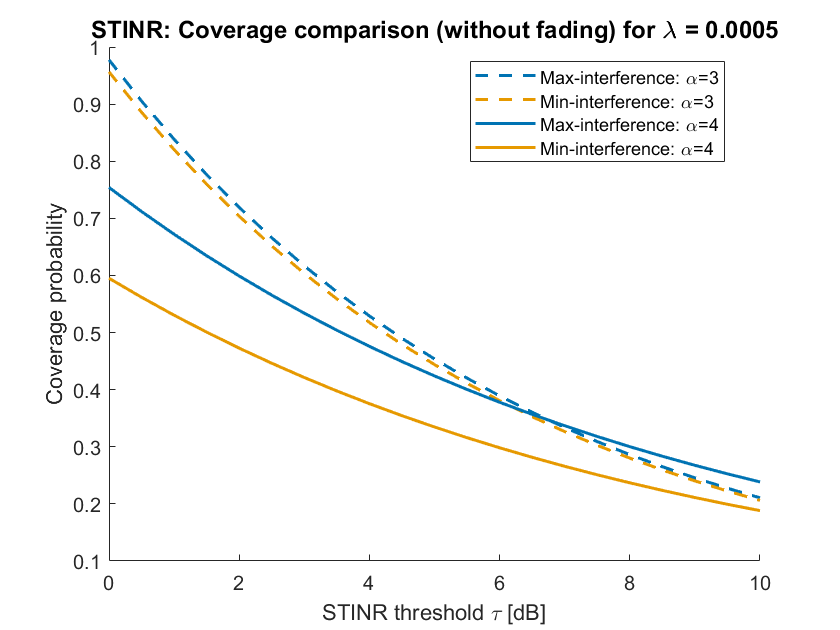}}
\caption{{\em Tropical case (without fading)}: Comparison of coverage probabilities at typical max-interference, typical min-interference, with respect to STINR threshold or user QoS requirement $\tau$ in dB unit, in the case: (a) $\la=10^{-4}$, (b) $\la=5\times 10^{-4}$ for $\alpha=3$ and $4$ in the case without fading.} 
\label{fig:T-mIvsMINF} 
\end{figure}
%
%
\section{Performance evaluation for the 3D model}~\label{section:3DCP}
\textbf{Coverage probabilities:}
In the 3D model of UAVs and HAPS, as described in Section~\ref{section:3D}, the SINR and STINR metrics can still be defined in terms of the variables used in 2D analysis. In the definitions of the SINR and STINR metrics in Table~\ref{table:CP_metrics}, we replace all the distance variables, say $X$, with its Pythagorised version $(X^2+a^2)^\half$. Our techniques can be adapted to compute coverage probabilities at various typical epochs of interest. We choose to just state their expressions in the cases with and without fading and in the non-tropical and tropical cases, in the form of following theorems:
\begin{theorem}[Scenario with fading]
In the non-tropical and tropical cases the coverage probabilities with respect to Palm probability of the epochs of interest, are as follows:
\begin{enumerate}
    \item For the typical mS-MI, MS and typical time epoch  
\begin{align}
p^c(\tau,\mu,\la,\a)&=  \E^0_H\left[e^{-\mu\tau (H^2+a^2)^{\a/2}\s^2} \Lcal_{I_{H}}(\mu\tau (H^2+a^2)^{\a/2})\right],
\label{eq:GCP3d}
\end{align}
where $H \in \{H_{\sst \Vcal}, H_{\sst S},R\}$. 
\item For the typical mI and MI epoch 
\begin{align}
p^c(\tau,\mu,\la,\a)&= \E^0_{H,R}\left[e^{-\mu\tau (R^2+a^2)^{\a/2}\s^2} \Lcal_{I_{H}}(\mu\tau (R^2+a^2)^{\a/2})\right],
\label{eq:GCP23d}
\end{align}
where $R\in \{R_{\sst I}, R_{\sst \Ical}\}$, $H \in \{H_{\sst I}, H_{\sst \Ical}\}$.
\end{enumerate}
\end{theorem}
\begin{theorem}[Scenario without fading]
In the non-tropical and tropical cases, the coverage probabilities with respect to the Palm probability of the epochs of interest, are given as follows:
\begin{enumerate}
\item For the typical mS-MI, MS and typical time epochs, in the non-tropical case and typical mS-MI in the tropical case,
\begin{align}
p^c(\tau,\la,\a)&= \E^0_{H}\left[F_{I_{H}}\left(0\vee ((H^2+a^2)^{-\a/2}/{\tau}-\s^2)\right)\right], 
\label{eq:GCP-NF3d}
\end{align}
where $H\in\{H_{\sst \Vcal}, H_{\sst S}, R\}$ in the non-tropical case and $H=H_{\sst \Vcal}$ in the tropical case.
\item For the typical time and the typical MS epoch, in the tropical case 
\begin{align}
p^c(\tau,\la,\a)&= \E^0_{H, X}\left[F_{\Tcal}\left(0\vee ((R^2+a^2)^{-\a/2}/{\tau}-\s^2)\right)\right], 
\label{eq:GCP-NF2S3d}
\end{align}
where $X\in \{R_1, R_{\sst S}\}$, $H \in \{R, H_{\sst S}\}$, and $\Tcal=(X^2+a^2)^{-\a/2}$ is the tropical interference.
\item For the typical mI and MI in the non-tropical and tropical cases 
\begin{align}
p^c(\tau,\la,\a)&= \E^0_{H, R}\left[F_{I_{H}}\left(0\vee ((R^2+a^2)^{-\a/2}/{\tau}-\s^2)\right)\right], 
\label{eq:GCP-NF23d}
\end{align}
\begin{align}
p^c(\tau,\la,\a)&= \E^0_{H, R}\left[F_{\Tcal}\left(0\vee ((R^2+a^2)^{-\a/2}/{\tau}-\s^2)\right)\right], 
\label{eq:GCP-NF2T3d}
\end{align}
where $R\in \{R_{\sst I}, R_{\sst \Ical}\}$, $H \in \{H_{\sst I}, H_{\sst \Ical}\}$ and $\Tcal=(H^2+a^2)^{-\a/2}$ is the tropical interference.
\end{enumerate}
\end{theorem}
\textbf{Comparison:} Similarly to Theorem~\ref{theorem:SINR-c}, the same ordering, as in Subsection~\ref{subsection:comparison} can be shown to hold among the coverage probabilities at various epochs:
\begin{itemize}
\item For the non-tropical case with or without fading: 
$p_{\sst \Vcal}^c  \leq p_{t}^c \leq p_{\sst S}^c$ and $p^c_{\sst \Ical}  \leq p^c_{\sst I}$.
\item For the tropical case with or without fading: $p_{\sst{\Vcal;\Tcal}}^c  \leq p_{\sst{t;\Tcal}}^c \leq p_{{\sst S;\Tcal}}^c$ and $p^c_{{\sst \Ical;\Tcal}} \leq p^c_{{\sst I;\Tcal}}$.
\end{itemize}
The ordering among the average Shannon rates at typical epochs, as in Corollary~\ref{cor:SINR-rate}, also hold true.

\textbf{Scale invariance:} The scale invariance property (Theorem~\ref{theorem:si}) of the coverage probabilities and average Shannon rates at these typical epochs in the interference limited regime does not hold with respect to the intensity parameter $\la$. This is evident from the fact that the ratios
\begin{align}
\text{SINR}= \frac{\rho (X^2+a^2)^{-\a/2}}{I} \mbox{ and } \text{STINR}= \frac{\rho (X^2+a^2)^{-\a/2}}{\Tcal},
\label{eq:ratios}
\end{align}
are not scale invariant with respect to $\la$, where $I=\sum \rho_i (D_i^2+a^2)^{-\a/2}$ and $\Tcal=\max \rho_i (D_i^2+a^2)^{-\a/2}$, are the additive and maxitive interference, respectively. Moreover, the coverage probabilities are eventually computed using a Laplace transform of the form
\[
\E^0_{*}\left[e^{-\g X^2}\right] = e^{-\g a^2}\E^0_{*}\left[e^{-\g Y^2}\right],
\]
for $*\in \{\sst{\Vcal, S, I, \Ical, t}$\hspace{-0.001em}$\}$, $X$ and $Y$ representing the typical distance variables in the 2D and 3D models, respectively, under the relation $X=(Y^2+a^2)^\half$. Since the parameter $\g$ is a function of $\la$ in our computations, we can say that the coverage probabilities depend on $\la$, and are not scale invariant with respect to $\la$. 
\section{Conclusions and future research} 
This work addresses the question of performance evaluation in a Poisson based dynamical cellular communication system at various typical epochs, at which a UE tend to suffer or benefit from the presence of interferer(s) or server close or far from it. Building on top of the mobility model described in~\cite{FB-SKJ} using stochastic geometry, we have derived the coverage probabilities at these epochs, depending on the system parameter, and made comparison between them. Our approach based on the determination of the characteristics of several typical distances that are essential for the performance evaluation. The ergodic data rate among different typical epochs correspond to different seasons of Shannon rate with respect to our seasonal analogy of the system performance. This article contains a complete study of system performance in a Poisson BSs setting in the absence or presence of fading and noise, along with simple assumption of Rayleigh fading and classical use case of power law path-loss attenuation function. The in-built tractability of the Poisson analysis allows one to analyze many other fading distribution and path-loss attenuation function. 

As a future direction of research, we plan to characterize various typical epochs of interest, investigate the performance evaluation at those typical epochs and their comparisons in other dynamical wireless system models, mentioned in~\cite{FB-SKJ}, for example, general 3D Poisson setting, 2D \& 3D Cox setting, Poisson and Cox setting under the spherical set up, with single and multiple speed and altitude scenarios, with or without visibility constraint. The Markov chain construction of the successive handover process developed in~\cite[Theorem 6.1]{FB-SKJ}, allows one to investigate the evolution of the seasons as a Markov chain, which we keep as a future work.
\appendices
\renewcommand{\thesection}{\Roman{section}}
\renewcommand{\thesectiondis}{\Roman{section}}

\section{Proofs: Preliminaries}~\label{section:A-prelim}
\subsection{Proof of Proposition~\ref{proposition:Lk}}~\label{subsection:Lk}
Let us prove this using induction. We know that the statement is true for $k=1$ from~\cite[Lemma~4.18]{FB-SKJ}. Suppose $k>1$ and assume that the statement is true for some $1, 2, \cdots, k-1$, i.e., $\left\{(s,h)\in \Lcal^{(j)}_e\right\}= \left\{\Hcal(U^s_h)=j-1\right\}$, for $j=1, 2, \cdots, k-1$. Let us now prove the statement for $j=k$. Suppose $\Hcal(U^s_h)<k-1$. Then $(s,h)\in \Lcal^{(j)}_e$, for some $j\leq k-1$. On the other hand, if $\Hcal(U^s_h)>k-1$, then the result is a corollary of Property~\ref{in-out}. Indeed, if there are more than $(k-1)$ point in $U^{s}_{h}$, then the radial bird at those points intersects the vertical line $t=s$ at a level below $h$, contradicting the fact that $(s,h)\in \Lcal^{(k)}_e$. For part~\ref{b}, if $(s,h)$ is an intersection point of two radial birds, then $\overline{U^s_h}$ will contain two extra points corresponding to the two radial birds. \hfill $\square$
\subsection{Proof of Proposition~\ref{proposition:JHRS}}~\label{subsection:JHRS}
Here we use the intensity of the point process $\Vcal_{\sst S}$, $\la_{\sst S}= \sqrt{\la}$ from~\cite[Lemma~5.8]{FB-SKJ}.  The joint Laplace transform of $H_{\sst S}^2, R_{\sst S}^2$ with parameter $\g, \b$ under the Palm probability measure of $\Vcal_{\sst S}$ is 
\begin{align}
\E^0_{\Vcal_{\sst S}}\left[e^{-\g H_{\sst S}^2-\b R_{\sst S}^2}\right]
&= \frac{1}{\la_{\sst S}}\E\left[\sum_{(T^{\sst S}_i, H^{\sst S}_i)\in \Hcal_{\sst S}: T^{\sst S}_i\in [0,1]}e^{-\gamma  \left(H^{\sst S}_i\right)^2-\b (R_i^{\sst S})^2}\right]\nn\\
&= \frac{1}{\sqrt{\la}}\E\left[\sum_{(T_i,H_i)\in \Hcal : T_i\in [0,1]} \hspace{-0.32in}e^{-\gamma  H_i^2-\b (R_i^{\sst S})^2}\one{\left\{\Hcal(U^{T_i}_{H_i})=0\right\}}\one{\left\{\exists (T_j,H_j)\in \Hcal\vert_{\left(\overline{U^{T_i}_{H_i}}\right)^c} : (T_i,H_i)\in U^{T_i}_{R_i^{\sst S}}\right\}}\right]\nn\\
&= \frac{1}{\sqrt{\la}}\E\left[\sum_{(T_i,H_i)\in \Hcal : T_i\in [0,1]} \hspace{-0.32in} e^{-\g H_i^2}\one_{\left\{\Hcal(U^{T_i}_{H_i})=0\right\}} \hspace{-0.06in} \sum_{(T_j,H_j)\in \Hcal\vert_{\left(\overline{U^{T_i}_{H_i}}\right)^c}} \hspace{-0.25in}e^{-\b R_{i,j}^2} \one{\left\{\Hcal\left(U^{T_i}_{R_{i,j}}\setminus U^{T_i}_{H_i}\right)=1\right\}}\right]\nn\\
&= \frac{1}{\sqrt{\la}} \E\left[\sum_{(T_i,H_i)\in \Hcal : T_i\in [0,1]}\hspace{-0.15in} e^{-\g H_i^2} \hspace{-0.06in} \sum_{(T_j,H_j)\in \Hcal\vert_{\left(\overline{U^{T_i}_{H_i}}\right)^c}} \hspace{-0.25in}e^{-\b R_{i,j}^2} \one_{\Hcal\left(U^{T_i}_{R_{i,j}}\right)=1}\right],\nn
\end{align}
where $(R_i^{S})^2 \equiv R_{i,j}^2= (T_i-T_j)^2+H_j^2$. In the second equality, the second indicator is for the purpose of finding the responsible head point, say $(T_j,H_j)$, from the second lower envelope $\Lcal_e^{(2)}$, which intersects the vertical line at time $T_i$ at height $R_i^{\sst S}$. The head point $(T_j,H_j)$ must be such that the only point lying inside $U^{T_i}_{R_i^{\sst S}}$ is $(T_i,H_i)$. By applying the multivariate the Campbell-Mecke formula~\cite{BBK} for the factorial power of order 2 of $\Hcal$, i.e., $\Hcal^{2,\neq}$ and and later applying the Slivnyak-Mecke theorem~\cite{BBK} similar to the Palm distribution in~\cite[Theorem~5.4]{FB-SKJ}, we get
\begin{align}
\E^0_{\Vcal_{\sst S}}\left[e^{-\g H_{\sst S}^2-\b R_{\sst S}^2}\right] 
&=\frac{4\la^2}{\sqrt{\la}} \int_0^1\!\!\int_0^\infty \!\!\!\int_{\left(\overline{U^{t_i}_{h_i}}\right)^c} 
e^{-\g h_i^2} e^{-(\b+\la\pi) ((t_i-t_j)^2+h_j^2)} {\rm d}h_j {\rm d}t_j {\rm d}h_i {\rm d}t_i\nn\\
&=4\la^{\frac 3 2} \int_0^\infty \int_{\left(\overline{U^{0}_{h_i}}\right)^c} e^{-\g h_i^2} e^{-(\b+\la\pi) (t^2+h_j^2)}  {\rm d}h_j {\rm d}t {\rm d}h_i =2\la^{3/2} \pi\int_0^\infty e^{-\g h_i^2} K_{\b+\la\pi}(h_i) {\rm d}h_i,
\label{eq:Jbeta0}
\end{align}
where
\begin{align}
K_{\b+\la\pi}(h_i)&:=\frac{2}{\pi}\int_{\left(\overline{U^{0}_{h_i}}\right)^c} 
e^{-(\b+\la\pi) (t^2+h_j^2)}  {\rm d}h_j {\rm d}t = \frac{2}{\pi}\int_0^\pi\int_{h_i}^\infty  e^{-(\b+\la\pi) r^2} r{\rm d}r{\rm d}\theta  = \frac{1}{\b+\la\pi}e^{-(\b+\la\pi) h_i^2}.
\label{eq:Jbeta1}
\end{align}
Substituting from (\ref{eq:Jbeta1}) to (\ref{eq:Jbeta0}), we get
\begin{align}
\E^0_{\Vcal_{\sst S}}\left[e^{-\g H_{\sst S}^2-\b R_{\sst S}^2}\right] &= \frac{2\la^{3/2} \pi}{\b+\la\pi} \int_0^\infty e^{-(\g+\b+\la\pi) h_i^2} {\rm d}h_i =\frac{\la^{3/2} \pi^{3/2}}{\b+\la\pi} (\g+\b+\la\pi)^{-\half} = \left(1{+}\frac{\beta}{\la\pi}\right)^{-1}\!\! \left(1{+}\frac{\g+\beta}{\la\pi}\right)^{-\half}.\nn
\end{align}
Note that for $\g=0$, we have $\E^0_{\Vcal_{\sst S}}\left[e^{-\b R_{\sst S}^2}\right]=\left(1+\frac{\beta}{\la\pi}\right)^{-3/2}$ and hence $R_{\sst S}$ has a Nakagami distribution with parameters $\left(\frac{3}{2}, \frac{3}{2\la\pi}\right)$. \hfill $\square$  
\subsection{Proof of Lemma~\ref{lem:IVH}}~\label{subsection:L-IVH}
Using the alternative definition of $\Vcal_{\sst I}$ in (\ref{eq:Acal_I2}) the intensity of the point process $\Vcal_{\sst I}$ is
\begin{align}
\la_{\sst I}&=\E\left[\sum_{T^{\sst I}_j\in \Vcal_{\sst I}\,\mbox{:}\, 0\leq T^{\sst I}_j\leq 1} \!\!\!\! 1\right] 
= \E\left[\sum_{\substack{(T_i,H_i)\in \Hcal\\0\leq T_i\leq 1}} \!\!\!\!\one{\left\{\Hcal(U^{T_i}_{H_i})=1\right\}} \right].\nn
\end{align}
Applying the Campbell-Mecke formula and the Slivnyak-Mecke theorem, we have
\begin{align}
\la_{\sst I}&=2\la\int_0^1\int_0^\infty \la\pi h^2 e^{-\la\pi h^2} {\rm d}h {\rm d}t
= 2\la\int_0^\infty \la\pi h^2 e^{-\la\pi h^2} {\rm d}h
=  \la\int_0^\infty \!\!\!\!\!\!x e^{-x} \frac{1}{\sqrt{\la\pi x}}{\rm d}x = \sqrt{\frac{\la}{\pi}}  \int_0^\infty \!\!\!x^{\frac{3}{2}-1} e^{-x}{\rm d}x 
= \frac{\sqrt \la}{2}.\nn\;\;\; \hfill \square
\end{align}
\subsection{Proof of Proposition~\ref{prop:IVHd}}~\label{subsection:T-IVHd}
The Laplace transform of $H_{\sst I}^2$ under the Palm probability measure of $\Vcal_{\sst I}$ is
\begin{align}
\hspace{-0.065in}\E^0_{\Vcal_{\sst I}}\left[e^{-\gamma H_{\sst I}^2}\right]
&= \frac{1}{\la_{\sst I}}\E\left[\sum_{(T^{\sst I}_j, H^{\sst I}_j)\in \Hcal_{\sst I}: T^{\sst I}_j\in [0,1]  }e^{-\gamma  \left(H^{\sst I}_j\right)^2}\right]
= \frac{1}{\la_{\sst I}}\E\!\!\left[\sum_{\substack{(T_i,H_i)\in \Hcal\\ T_i\in [0,1]}}  e^{-\gamma  H_i^2}\one{\left\{\Hcal(U^{T_i}_{H_i})=1\right\}}\right].
\label{eq:H2I}
\end{align}
Applying the Campbell-Mecke formula and the Slivnyak-Mecke theorem, and using $\la_{\sst I}=\half\sqrt{\la}$ from Lemma~\ref{lem:IVH}, the last term in (\ref{eq:H2I}) equals
\begin{align}
\frac{2\la}{\la_{\sst I}} \int_0^1\int_0^\infty \la\pi h^2 e^{-(\g+\la\pi) h^2} {\rm d}h {\rm d}t
=4\sqrt{\la} \int_0^\infty \la\pi h^2 e^{-(\g+\la\pi) h^2} {\rm d}h
&= 4\sqrt{\la} \frac{\la\pi}{\g+\la\pi} \int_0^\infty x e^{-x} \frac{1}{2\sqrt{(\g+\la\pi)x}}{\rm d}x\nn\\
&= 2\pi \left(\frac{\la}{\g+\la\pi}\right)^{\frac{3}{2}} \G\left(\frac{3}{2}\right)= \left(1+\frac{\g}{\la\pi}\right)^{-\frac{3}{2}}.\nn 
\end{align}
This shows that  $H_{\sst I}$ follows a Nakagami distribution with parameters $\left(\frac{3}{2}, \frac{3}{2\la\pi}\right)$. \hfill $\square$
\subsection{Proof of Proposition~\ref{proposition:JHRI}}~\label{subsection:T-JHRI}
Given $(T^{\sst I}_i,H^{\sst I}_i)\in \Hcal_{\sst I}$, if $(T_j,H_j)\in \Hcal\vert_{U^{T_i^{\sst I}}_{H_i^{\sst I}}}$, then 
$R^{\sst I}_i:=\left((T_i^{\sst I}-T_j)^2+ H_j^2\right)^\half \leq H_i^{\sst I}$.
The joint Laplace transform of $H_{\sst I}^2, R_{\sst I}^2$ is given by 
\begin{align}
\E^0_{\Vcal_{\sst I}}\left[e^{-\g H_{\sst I}^2-\b R_{\sst I}^2}\right]
&= \frac{1}{\la_{\sst I}}\E\left[\sum_{(T^{\sst I}_i, H^{\sst I}_i)\in \Hcal_{\sst I}: T^{\sst I}_i\in [0,1]}e^{-\gamma  \left(H^{\sst I}_i\right)^2-\b (R_i^{\sst I})^2}\right]\nn\\
&= \frac{1}{\la_{\sst I}}\E\left[\sum_{\substack{(T_i,H_i)\in \Hcal\\ T_i\in [0,1]}}e^{-\gamma  H_i^2-\b (R_i^{\sst I})^2}\one_{\left\{\Hcal(U^{T_i}_{H_i})=1\right\}}\one{\left\{\exists (T_j,H_j)\in \Hcal: (T_j,H_j)\in U^{T_i}_{H_i}\right\}}\right]\nn\\
&= \frac{1}{\la_{\sst I}}\E\left[\sum_{\substack{(T_i,H_i)\in \Hcal\\ T_i\in [0,1]}} \hspace{-0.15in} e^{-\g H_i^2} 
\one_{\left\{\Hcal(U^{T_i}_{H_i})=1\right\}} \hspace{-0.06in} \sum_{(T_j,H_j)\in \Hcal\vert_{U^{T_i}_{H_i}}} \hspace{-0.25in}e^{-\b R_{i,j}^2}\right],\nn
\end{align}
where, $R_{i,j}^2= (T_i-T_j)^2+H_j^2$. By applying the multivariate the Campbell-Mecke formula for the factorial power of order 2 of $\Hcal$, i.e., $\Hcal^{2,\neq}$ and the Slivnyak-Mecke theorem similar to the Palm distribution in~\cite[Theorem~5.4]{FB-SKJ}, we obtain that  
\begin{align}
\E^0_{\Vcal_{\sst I}}\left[e^{-\g H_{\sst I}^2-\b R_{\sst I}^2}\right]
&=\frac{4\la^2}{\la_{\sst I}} \int_0^1\int_0^\infty \!\!\!\int_{U^{t_i}_{h_i}} 
\hspace{-0.07in}e^{-(\g+\la\pi) h_i^2} e^{-\b ((t_i-t_j)^2+h_j^2)} {\rm d}h_j {\rm d}t_j {\rm d}h_i {\rm d}t_i\nn\\
&=8\la^{\frac 3 2} \int_0^\infty \int_{U^0_{h_i}} \hspace{-0.07in} e^{-(\g+\la\pi) h_i^2} 
e^{-\b (t^2+h_j^2)}  {\rm d}h_j {\rm d}t {\rm d}h_i
=4\la^{3/2} \pi\int_0^\infty e^{-(\g+\la\pi) h_i^2} J_{\b}(h_i) {\rm d}h_i,
\label{eq:HIR-2}
\end{align}
where
\begin{align}
J_{\b}(h_i)&:=\frac{2}{\pi}\int_{ U^0_{h_i}} 
e^{-\b (t^2+h_j^2)}  {\rm d}h_j {\rm d}t
=\frac{2}{\pi}\int_0^\pi\int_0^{h_i} e^{-\beta r^2} r{\rm d}r{\rm d}\theta = \frac{1}{\beta}\left(1-e^{-\beta h_i^2}\right).
\label{eq:Jbeta}
\end{align}
This implies that, conditioned on $H_{\sst I}= h_{\sst I}$, $R^2_{\sst I}$ is uniformly distributed on $[0,h^2_{\sst I}]$. Hence, from (\ref{eq:HIR-2}) and (\ref{eq:Jbeta}), we have
\begin{align}
\E^0_{\Vcal_{\sst I}}\left[e^{-\g H_{\sst I}^2} e^{-\b R_{\sst I}^2}\right]
&=\frac{4\la^{3/2}\pi}{\beta} \int_0^\infty e^{-(\g+\la\pi) h_{\sst I}^2} \left(1-e^{-\beta h_{\sst I}^2}\right) {\rm d}h_{\sst I}\nn\\
&=\frac{4\la^{3/2}\pi}{\beta}\left[ \int_0^\infty e^{-(\g+\la\pi) h_{\sst I}^2}  {\rm d}h_{\sst I} {-}\int_0^\infty e^{-(\g+\la\pi+\beta) h_{\sst I}^2} {\rm d}h_{\sst I} \right]\nn\\
&=\frac{2\la^{3/2}\pi^{3/2}}{\beta}\left[(\g+\la\pi)^{-\half}-(\g+\la\pi+\beta)^{-\half}\right] 
=\frac{2\la\pi}{\beta}\left[\left(1+\frac{\g}{\la\pi}\right)^{-\half}-\left(1+\frac{\g+\beta}{\la\pi}\right)^{-\half}\right],\nn
%
\end{align}
as required in this result. \hfill $\square$
\subsection{Proof of Lemma~\ref{lem:VTI}}~\label{subsection:L-VTI}
For any two points $(t_1,h_1), (t_2,h_2)$  with $t_2<t_1$, we have that the square of the distance to the intersection of two radial birds at $(t_1,h_1), (t_2,h_2)$ is given by
\begin{align}
\hat h^2 &= \frac{1}{4}\left[(t_1-t_2)^2+2(h_1^2+h_2^2)+\frac{(h_2^2-h_1^2)^2}{(t_1-t_2)^2}\right].
\label{eq:hat-h}
\end{align}
The expression for $\hat h^2$ is derived using $\hat h^2 = (t_i-\hat s)^2 + h_i^2, i = 1, 2$, where $(\hat s, \hat h)= C_{(t_1,h_1)}\cap C_{(t_2,h_2)}$. Applying the multivariate the Campbell-Mecke formula for the factorial power of order 2 of $\Hcal$, i.e., $\Hcal^{2,\neq}$, and the Slivnyak-Mecke theorem, similarly to~\cite[Theorem~4.26]{FB-SKJ}, the intensity of $\Vcal_{\sst \Ical}$ is
\begin{align}
\la_{\sst \Ical} &= 4\la^2\int_0^1\int_{0}^{\infty}\int_{0}^{\infty}\int_{-\infty}^{t_1} \la \pi \hat h^2 e^{-\la \pi \hat h^2} {\rm d} t_2\, {\rm d} h_2\, {\rm d} h_1 \, {\rm d} t_1 \nn\\
&= 4\la^2\int_{0}^{\infty}\int_{0}^{\infty}\int_{0}^{\infty} \la \pi \hat h^2 e^{-\la \pi \hat h^2} {\rm d} t\, {\rm d} h_1\, {\rm d} h_2,
\label{eq:h2}
\end{align}
where $\hat h$ is a function of $t=t_1-t_2$, $h_1$ and $h_2$. We write the inner integral as
\begin{align}
\int_{0}^{\infty} \la \pi \hat h^2 e^{-\la \pi \hat h^2} {\rm d} t
&= e^{-\frac{\la\pi}{2}(h_1^2+h_2^2)}  \left(I(a,b) + \frac{\la\pi}{2} (h_1^2+h_2^2) J(a,b) \right),
\label{eq:IJ}
\end{align}
where $a=\frac{\la\pi}{4}$ and $b=\frac{\la\pi}{4} (h_2^2-h_1^2)^2$ and 
\begin{align}
I(a,b)=\int_0^\infty \left(a t^2+\frac{b}{t^2}\right) e^{-\left(a t^2+\frac{b}{t^2}\right)}{\rm d}t,
\label{eq:Iab}
\end{align}
\begin{align}
J(a,b)=\int_0^\infty e^{-\left(a t^2+\frac{b}{t^2}\right)}{\rm d}t = \frac{\sqrt{\pi}}{2\sqrt{a}}e^{-2\sqrt{ab}}.
\label{eq:Jab}
\end{align}
The integral $J(a,b)$ in (\ref{eq:Jab}) is computed using Cauchy-Schl\"{o}milch transformation~\cite{Amdeberhan-etal}. We first compute an integral of the form (\ref{eq:Iab}). Using Leibniz's rule for differentiation under the integral sign for two parameter $a$ and $b$, we have
\begin{align}
I(a,b)&= -a\frac{\partial}{\partial a} J(a,b)- b\frac{\partial}{\partial b} J(a,b)\nn\\
& {=}{-}\frac{\sqrt{\pi}}{2} e^{-2\sqrt{ab}}\left[a\left({-}\half a^{-3/2}{-} \sqrt{b}/a\right)+b \left({-}\frac{1}{\sqrt{b}}\right) \right]\nn\\
& =\!\! \frac{\sqrt{\pi}}{2} e^{-2\sqrt{ab}}\!\!\left[\frac{1}{2\sqrt{a}}{+} 2\sqrt{b} \right] 
\!=\!\frac{\sqrt{\pi}}{4\sqrt{a}} e^{-2\sqrt{ab}}\!\!\left(1{+} 4\sqrt{ab} \right). \nn
\end{align}
Then using  $a=\frac{\la\pi}{4}$ and $b=\frac{\la\pi}{4} (h_2^2-h_1^2)^2$ in (\ref{eq:IJ}), we have
\begin{align}
\int_{0}^{\infty} \la \pi \hat h^2 e^{-\la \pi \hat h^2} {\rm d} t
&= e^{-\frac{\la\pi}{2}(h_1^2+h_2^2)} e^{-2\sqrt{ab}} \frac{\sqrt{\pi}}{4\sqrt{a}}  \left(1+ 4\sqrt{ab} + \la\pi (h_1^2+h_2^2) \right)\nn\\
&{=} e^{-\frac{\la\pi}{2}\left(h_1^2+h_2^2+|h_1^2-h_2^2|\right)} \frac{1}{2\sqrt{\la}}  \left[1+ \la\pi(|h_1^2-h_2^2|{+}h_1^2+h_2^2) \right]\nn\\
&= \frac{1}{2\sqrt{\la}} e^{-\la\pi(h_1^2\vee h_2^2)}   \left(1+ 2\la\pi(h_1^2\vee h_2^2)\right).
\label{eq:IJ1}
\end{align}
Using the value of the inner integral from (\ref{eq:IJ1}) in (\ref{eq:h2}) we have
\begin{align}
\la_{\sst \Ical}&{=} 2\la^{3/2}\!\!\int_{0}^{\infty}\!\!\!\!\int_{0}^{\infty}\!\!\!\!  e^{-\la\pi(h_1^2\vee h_2^2)}   \left(1+ 2\la\pi(h_1^2\vee h_2^2)\right)  {\rm d} h_1\, {\rm d} h_2\nn\\
&= \frac{2\sqrt{\la}}{\pi} \int_{0}^{\infty}\!\!\!\! \int_{0}^{\infty}\!\!\!\! e^{-h_1^2\vee h_2^2}   \left(1{+} 2(h_1^2\vee h_2^2)\right){\rm d} h_1 {\rm d} h_2.
\label{eq:h3}
\end{align}
Since the integrand is a symmetric function of $h_1,h_2$, it is enough to compute the integral for $h_2\in [0,\infty)$ and $h_1\in [0,h_2]$. Using this, it can be proved that
\[
\int_{(\R^+)^2}\!\!\!\!\! \!\!\! e^{-(h_1^2\vee h_2^2)} {\rm d} h_1\, {\rm d} h_2 = 1 =
\int_{(\R^+)^2} \!\!\!\!\! (h_1^2\vee h_2^2) e^{-(h_1^2\vee h_2^2)}    {\rm d} h_1\, {\rm d} h_2.
\]
This shows from (\ref{eq:h3}) that $\la_{\sst \Ical} = \frac{6\sqrt{\la}}{\pi}$. \hfill $\square$
\subsection{Proof of Proposition~\ref{proposition:TIH}}~\label{subsection:T-TIH}
The Laplace transform of $H_{\sst \Ical}^2$ is obtained under the Palm probability measure of $\Vcal_{\sst \Ical}$ similarly to the Palm distribution in~\cite[Theorem~5.4]{FB-SKJ} as follows:
\begin{align}
\E^0_{\Vcal_{\sst \Ical}}\left[e^{-\gamma H_{\sst \Ical}^2}\right]
&= \frac{1}{\la_{\sst \Ical}}\E\left[\sum_{\substack{(T_i,H_i)\in \Hcal\\ T_i\in [0,1]}} \sum_{\substack{(T_j,H_j)\in \Hcal\\ T_j<T_i}}e^{-\gamma  \hat H^2}\one_{\left\{\Hcal(U^{\hat S}_{\hat H})=1\right\}}\right]\nn\\
&{=}\frac{4\la^2}{\la_{\sst \Ical}}\int_0^1 \!\! \!\int_{0}^{\infty} \!\!\! \int_{0}^{\infty} \!\! \!\int_{-\infty}^{t_1} \!\!\!\!\! \la \pi \hat h^2 e^{-(\g+\la \pi) \hat h^2} {\rm d}  t_2\, {\rm d} h_2\, {\rm d} h_1 \, {\rm d} t_1 \nn\\
&=\frac{2\la^{3/2}\pi}{3}\int_{0}^{\infty} \!\! \int_{0}^{\infty} \!\!\int_0^{\infty} \!\! \la \pi \hat h^2 e^{-(\g+\la \pi) \hat h^2} {\rm d} t\, {\rm d} h_2\, {\rm d} h_1,
\label{eq:HTI1}
\end{align}
by applying the multivariate the Campbell-Mecke formula for the factorial power of order 2 of $\Hcal$ and the Slivnyak-Mecke theorem, similarly to that in the proof of Lemma~\ref{lem:VTI}, where $\la_{\sst \Ical} = \frac{6\sqrt{\la}}{\pi}$, from Lemma~\ref{lem:VTI}, and $\hat h$ is function of $t=t_1-t_2, h_1$ and $h_2$ as derived in (\ref{eq:hat-h}). We evaluate the inner integral in (\ref{eq:HTI1}) as
\begin{align}
\frac{\la \pi}{\g+\la\pi}\int_0^{\infty} (\g+\la \pi) \hat h^2 e^{-(\g+\la \pi) \hat h^2} {\rm d} t
&= \frac{\sqrt{\pi}}{2\sqrt{\g+\la\pi}} e^{-(\g+\la\pi)(h_1^2\vee h_2^2)}   \left(1+ 2(\g+\la\pi)(h_1^2\vee h_2^2)\right)\nn\\
&{=}\frac{1}{2\sqrt{\la}} \!\!\left(\frac{\la \pi}{\g{+}\la\pi}\right)^{1/2}    \!\!\! e^{-(\g{+}\la\pi)(h_1^2\vee h_2^2)}\! \left(1{+}2(\g{+}\la\pi)(h_1^2\vee h_2^2)\right),
\label{eq:HTI2}
\end{align}
using similar computation up to (\ref{eq:IJ1}). Substituting from (\ref{eq:HTI2}) to (\ref{eq:HTI1}), we have 
\begin{align}
\E^0_{\Vcal_{\sst \Ical}}\left[e^{-\gamma H_{\sst \Ical}^2}\right] &= \frac{\la\pi}{3} \left(\frac{\la \pi}{\g+\la\pi}\right)^{3/2}\int_{0}^{\infty}\!\!\!\int_{0}^{\infty} \!\!\! e^{-(\g+\la\pi)(h_1^2\vee h_2^2)} \left(1+ 2(\g+\la\pi)(h_1^2\vee h_2^2)\right)\,{\rm d} h_2\, {\rm d} h_1\nn\\
&{=}\frac{\la\pi}{3}\!\! \left(\!\frac{\la \pi}{\g{+}\la\pi}\!\right)^{3/2}\!\!\!\!\!\!\frac{3}{\g{+}\la\pi}{=} \left(1{+}\frac{\g}{\la \pi}\right)^{-5/2},
\label{eq:LIcal}
\end{align}
since the double integral in the first step in (\ref{eq:LIcal}) equals to $\frac{3}{\g+\la\pi}$ similarly to (\ref{eq:h3}). \hfill $\square$
\subsection{Proof of Proposition~\ref{proposition:JHRTI}}~\label{subsection:T-JHRTI}
Let $(\hat S, \hat H)$ be the intersection of radial birds located at $(T_i,H_i)$ and $(T_j,H_j)$. The joint distribution of $H_{\sst \Ical}^2, R_{\sst \Ical}^2$ is given by the joint Laplace transform obtained similarly to the Palm distribution in~\cite[Theorem~5.4]{FB-SKJ}, as follows: 
\begin{align}
\E^0_{\Vcal_{\sst \Ical}}\left[e^{-\gamma H_{\sst \Ical}^2- \beta R_{\sst \Ical}^2}\right]
&= \frac{1}{\la_{\sst \Ical}}\E\left[\sum_{\substack{(T_i,H_i)\in \Hcal\\ T_i\in [0,1]}} \sum_{\substack{(T_j,H_j)\in \Hcal\\ T_j<T_i}}e^{-\gamma  H_{\sst \Ical}^2-\beta R_{\sst \Ical}^2}\one{\left\{\Hcal(U^{\hat S}_{\hat H})=1\right\}}  \one{\left\{\exists (T_k,H_k)\in \Hcal\vert_{U^{\hat S}_{\hat H}}\right\}}\right]\nn\\
&= \frac{1}{\la_{\sst \Ical}}\E\left[\sum_{\substack{(T_i,H_i)\in \Hcal\\ T_i\in [0,1]}} \sum_{\substack{(T_j,H_j)\in \Hcal\\ T_j<T_i}}e^{-\gamma  H_{\sst \Ical}^2}\one{\left\{\Hcal(U^{\hat S}_{\hat H})=1\right\}} \sum_{(T_k,H_k)\in \Hcal\vert_{U^{\hat S}_{\hat H}}} \hspace{-0.2in}e^{-\b R_{\hat S,k}^2}  \right]\nn\\
&=\frac{4\la^2}{\la_{\sst \Ical}}\!\!\int_0^1\!\!\int_{0}^{\infty}\!\!\!\int_{0}^{\infty} \!\!\! \int_{-\infty}^{t_1} \!\!\!\! e^{-(\g+\la \pi) \hat h^2} \la\pi J_{\beta}(\hat s,\hat h) {\rm d} t_2\, {\rm d} h_2\, {\rm d} h_1 \, {\rm d} t_1,
\label{eq:HTIRI}
\end{align}
where $\hat H \equiv H_{\sst \Ical}$ and $R^2_{\sst \Ical}=R_{\hat S,k}^2:= (\hat S-T_k)^2+H_k^2$, $J_{\beta}(\hat s,\hat h)=\frac{2}{\pi}\int_{U^{\hat s}_{\hat h}}  e^{-\beta \hat r^2} {\rm d}  t_3\, {\rm d} h_3$, $\hat r^2=(\hat s-t_3)^2+h_3^2$ is a realisation of $R_{\sst \Ical}$ and $\hat s=\frac{t_1^2+h_1^2- t_2^2-h_2^2}{2(t_1-t_2)}$.
Let $\hat s-t_3= u\cos\theta$ and $h_3=u\sin \theta$. Then similarly to (\ref{eq:Jbeta}) we can compute the integral
\begin{align}
J_{\beta}(\hat s,\hat h) &= \frac{2}{\pi}\int_{U^{\hat s}_{\hat h}}  e^{-\beta \hat r^2} {\rm d}  t_3\, {\rm d} h_3
= \frac{2}{\pi} \int_0^\pi\!\!\! \int_0^{\hat h} \!\!e^{-\b u^2} u{\rm d}u\,{\rm d}\theta {=} \frac{1}{\beta}\left(1{-}e^{-\beta \hat h^2}\right).
\end{align}
Conditioned on $H_{\sst \Ical}=\hat h$, we have $R^2_{\sst \Ical}$ is uniform on $[0,\hat h^{2}]$. Then the integral in (\ref{eq:HTIRI}) becomes
\begin{align}
\E^0_{\Vcal_{\sst \Ical}}\left[e^{-\g H_{\sst \Ical}^2} e^{-\b R_{\sst \Ical}^2}\right]
&=\frac{4\la^3\pi}{\la_{\sst \Ical}\b}\int_{0}^{\infty}\int_{0}^{\infty}\int_0^{\infty}  e^{-(\g+\la \pi) \hat h^2} \left(1-e^{-\beta \hat h^2}\right) {\rm d} t\, {\rm d} h_2\, {\rm d} h_1\nn\\
&=\frac{2\la\pi}{3\b} \left[\left(1+\frac{\g}{\la\pi}\right)^{-3/2}-\left(1+\frac{\g+\beta}{\la\pi}\right)^{-3/2}\right],\nn
\end{align}
using the fact that 
$\la_{\sst \Ical}=\frac{6\sqrt{\la}}{\pi}$ from Lemma~\ref{lem:VTI} and for any $\eta\geq 0$
\[
\hspace{1.9in}\pi\la^{3/2}\!\! \int_{0}^{\infty}\!\! \int_{0}^{\infty}\!\!\int_{0}^{\infty} \!\! e^{-(\la \pi+\eta) \hat{h}^2} {\rm d} t\, {\rm d} h_2\, {\rm d} h_1= \left(1+\frac{\eta}{\la\pi}\right)^{-3/2}\!\!\!. \hspace{1.7in} \hfill \square
\]
\section{Proofs: case with fading}~\label{proofs-wf}
This section contains the proof of Theorem~\ref{theorem:GPM-WF} and the derivation of these performance metrics for individual typical epochs with SINR and STINR, with respect to the corresponding Palm probability measure.  

\subsection{Proof of Theorem~\ref{theorem:GPM-WF}}
The coverage probability with respect to the Palm probability distribution of the epoch of interest is
\begin{align}
p^c(\tau,\mu,\la,\a)&= \E^0_{H}[\P(\Scal>\tau\vert H)].
\label{eq:OP1}
\end{align}
Conditioned on $H=h$ the inner term (\ref{eq:OP1}) is
\begin{align}
\P(\Scal>\tau\vert H=h)
&= \E_{I_{h}}\left[\P(\Scal>\tau\vert I_h)\right]
=\E_{I_{h}}\left[\P\left(\rho >\tau  h^{\a}\left(\s^2+I_{h}\right)\right)\right]
=\E_{I_{h}}\left[e^{-\mu\tau h^{\a}\left(\s^2+I_{ h}\right)}\right]
= e^{-\mu\tau h^{\a}\s^2}\Lcal_{I_{h}}(\mu\tau h^{\a}).\nn
\end{align}
Then the coverage probability is given by the formula
\begin{align}
\hspace{2.3in} p^c(\tau,\mu,\la,\a)&= \E^0_{H}\left[e^{-\mu\tau H^{\a}\s^2} \Lcal_{I_{H}}(\mu\tau H^{\a})\right]. \nn \;\;\;\hspace{2in}  \hfill \square
\end{align}
\subsection{Proof of Theorem~\ref{theorem:GPM-WF2}}
The proof follows the same steps as in Theorem~\ref{theorem:GPM-WF}, together with taking into account the randomness due to the distance $H$ to the nearest interferer as well, jointly with the distance $R$ to the serving BS. Given $H=h$ and $R=r$, the coverage probability is
\begin{align}
\P(\Scal>\tau\vert H=h, R=r)
= \E_{I_{h}}\left[\P(\Scal>\tau\vert I_h)\right]
=\E_{I_{h}}\left[\P\left(\rho >\tau  r^{\a}\left(\s^2+I_{h}\right)\right)\right]
&=\E_{I_{h}}\left[e^{-\mu\tau r^{\a}\left(\s^2+I_{ h}\right)}\right]\nn\\
&= e^{-\mu\tau r^{\a}\s^2}\Lcal_{I_{h}}(\mu\tau r^{\a}).\nn
\end{align}
Jointly averaging of $H$ and $R$, we have
\begin{align}
\hspace{2.3in} p^c(\tau,\mu,\la,\a)&= \E^0_{H,R}\left[e^{-\mu\tau R^{\a}\s^2} \Lcal_{I_{H}}(\mu\tau R^{\a})\right]. \nn \hspace{2in}\hfill \square
\end{align} 
For each of the typical epoch of interest, {\em ms-MI, MS, MI, mI} and {\em t}, in each of the subsections, we just state the expression of the metrics before proving them using the formula obtained in Theorem~\ref{theorem:GPM-WF} and Theorem~\ref{theorem:GPM-WF2}, without formally stating them as a theorem or corollary. From now on, for simplicity we write $p^c_{*}\equiv p^c_{*}(\tau,\mu,\la,\a)$, $p^c_{\sst *;\Tcal}\equiv p^c_{\sst *;\Tcal}(\tau,\mu,\la,\a)$ and $\Rcal_{*}\equiv \Rcal_{*}(\mu,\la,\a)$, $\Rcal_{\sst *;\Tcal}\equiv \Rcal_{\sst *;\Tcal}(\mu,\la,\a)$, which are understood to a function of the set of parameters $\{\tau,\mu,\la,\a\}$ and $\{\mu,\la,\a\}$, respectively, where $*\in \{\sst{\Vcal, S, I, \Ical, t}$\hspace{-0.01em}$\}$.
\begin{remark}
In all our proofs, we just derive the expression for $p^c_{*}(\tau,\mu,\la,\a)$ and $p^c_{\sst *;\Tcal}(\tau,\mu,\la,\a)$, and use it in the general formula in (\ref{eq:RcalG})
for both $\Rcal_{*}(\mu,\la,\a)$ and $\Rcal_{\sst *;\Tcal}(\mu,\la,\a)$, where $*\in \{\sst{\Vcal, S, I, \Ical, t}$\hspace{-0.01em}$\}$.
\end{remark}
\subsection{Min-signal-max-interference (mS-MI): typical handover}
In this case the distance to the serving BS is given by the handover distance $H_{\sst \Vcal}$. The Palm probability distribution of a typical handover distance $H_{\sst \Vcal}$ is given by a Nakagami distribution with parameter $\left(\frac{3}{2}, \frac{3}{2\la\pi}\right)$, having density (\ref{eq:pdfhatH0}) established in~\cite[Lemma~5.6]{FB-SKJ}.

Conditioned on the typical handover distance $H_{\sst \Vcal}=h$, the distance of all other BSs is given by a Poisson point process $\eta_h=\sum_i \delta_{D_i}$ on $[h, \infty)$ with intensity measure having density $2\pi\la r {\rm d}r$, as seen in Lemma~\ref{lemma:etappp} (see also~\cite[Lemma~4.34]{FB-SKJ}). Here is a result about the probability of coverage. In the following, we use the subscript $\Vcal$ to denote the typical handover.
\subsubsection{Theorem~\ref{theorem:GPM-WF} (for \textnormal{SINR} at typical mS-MI)} 
The probability of coverage in the system with threshold $\tau$ at a typical handover is given by 
\begin{align}
p_{\sst \Vcal}^c &{=} \frac{1}{1{+}\tau}\E^0_{H_{\sst \Vcal}}\left[ e^{-\mu\tau H_{\sst \Vcal}^{\a}\s^2-\pi\la H_{\sst \Vcal}^2 \kappa(\tau,\a)}\right],
\label{eq:hocp1}
\end{align}
where $\kappa(\tau,\a):=\tau^{2/\a} \int_{\tau^{-2/\a}}^\infty\frac{1}{1+z^{\a/2}} {\rm d}z$ as in (\ref{eq:kappa}).

\begin{proof}
Using the formula (\ref{eq:GCP}) we have  
\begin{align}
p_{\sst \Vcal}^c  &=\int_{0}^\infty e^{-\mu\tau h^{\a}\s^2}\Lcal_{I_{h}}(\mu\tau h^{\a})f_{H_{\sst \Vcal}}(h){\rm d}h,
\label{eq:cp1}
\end{align}
where $\Lcal_{I_{h}}(\g)$ is the Laplace transform of the total interference 
\[
I_{h}:=\rho h^{-\a}+\sum_i \rho_i D_i^{-\a},
\]
given the handover distance $H_{\sst \Vcal}=h$. Using the Poisson point process $\eta_{h}$ of distances of all other stations, we compute the  Laplace transform of $I_{h}$ as follows:
\begin{align}
\E_{I_{h}}\left[e^{-\g I_{h}}\right] &= \E_{\eta_h, \rho, \{\rho_i\}}\left[\E\left[e^{-\g I_{h}} \vert \eta_h, \rho, \{\rho_i\}\right]\right]\nn\\
&=\E_{\rho}\left[e^{-\g\rho h^{-\a}}\right]\E_{\eta_h,  \{\rho_i\}}\left[e^{-\g\sum_{i\in \N} \rho_i D_i^{-\a}}\right]\nn\\
&=\frac{\mu}{\mu+\g h^{-\a}}\E_{\eta_h,  \{\rho_i\}}\left[\prod_{i\in\N}e^{-\g \rho_i D_i^{-\a} }\right]
= \frac{\mu}{\mu+\g h^{-\a}}
\E_{\eta_h}\left[\prod_{i\in\N}\E_{\rho_i}\left[e^{-\g \rho_i D_i^{-\a} }\right]\right].
\label{eq:cp3}
\end{align}
Computing the PGFl with respect to the point process $\eta_h$, we have
\begin{align}
\E_{\eta_h}\left[\prod_{i\in\N}\E_{\rho_i}\left[e^{-\g \rho_i D_i^{-\a} }\right]\right]
&=\exp\left(-2\pi\la\int_{h}^\infty\left(1-\E_\rho[e^{-\g \rho r^{-\a}}]\right)r{\rm d}r\right)
= \exp\left(-2\pi\la\int_{h}^\infty\frac{\g r^{-\a}}{\mu+\g r^{-\a}}r{\rm d}r\right).
\label{eq:cp3a}
\end{align}
Substituting the Laplace transform of $I_h$ with $\g=\mu\tau h^{\a}$ from (\ref{eq:cp3}) to (\ref{eq:cp1}), we get
\begin{align}
p_{\sst \Vcal}^c 
&=\int_{0}^\infty e^{-\mu\tau h^{\a}\s^2}\Lcal_{I_{h}}(\mu\tau h^{\a})f_{H_{\sst \Vcal}}(h){\rm d}h\nn\\
&= \frac{1}{1{+}\tau}\int_{0}^\infty \!\!\! e^{-\mu\tau h^{\a}\s^2 -2\pi\la \int_{h}^\infty\frac{\tau h^{\a}}{r^{\a}+\tau h^{\a}}r{\rm d}r} f_{H_{\sst \Vcal}}(h){\rm d}h
=\frac{1}{1{+}\tau}\int_{0}^\infty \!\!\!\! e^{-\mu\tau h^{\a}\s^2 -\pi\la h^2 \kappa(\tau,\a)} f_{H_{\sst \Vcal}}(h){\rm d}h,\nn
\end{align}
where we have computed the integral 
\begin{equation}
\int_{h}^\infty\frac{\tau h^{\a}}{r^{\a}+\tau h^{\a}}r{\rm d}r=\half h^2 \kappa(\tau,\a),
\label{eq:kappa-a}
\end{equation}
where $\kappa(\tau,\a):=\tau^{2/\a} \int_{\tau^{-2/\a}}^\infty\frac{1}{1+z^{\a/2}} {\rm d}z$ as in (\ref{eq:kappa}).
\end{proof}
\begin{remark}[Special case: $\a=2$ and $4$]
For the attenuation exponent $\a=2$, we have 
\[\kappa(\tau,2)=\tau \int_{\tau^{-1}}^\infty\frac{1}{1+z} {\rm d}z=\infty.\]
As a result the coverage probability $p^c_{\sst \Vcal}(\tau,\mu,\la,2)=0$. Also for $\a=4$, we have that 
\[
\kappa(\tau,4)=\tau^{\half} \int_{\tau^{-\half}}^\infty\frac{1}{1+z^2} {\rm d}z 
= \tau^{\half} \arctan(\tau^{\half}).
\]
Using this, the coverage probability in (\ref{eq:hocp1}) simplifies to
\[
p^c_{\sst \Vcal}(\tau,\mu,\la,4)= \frac{1}{1{+}\tau}\E^0_{H_{\sst \Vcal}} \left[e^{-\mu\tau H_{\sst \Vcal}^{\a}\s^2 -\pi\la H_{\sst \Vcal}^2 \tau^{\half} \arctan(\tau^{\half})} \right],
\]
in the fading environment with $\a=4$.
\end{remark}
\begin{remark}
In an interference limited regime with fading exponent $\a=4$, the coverage probability has a exact closed form 
expression as
\begin{align}
p^c_{\sst \Vcal}(\tau,\mu,\la,4) &= \frac{1}{1{+}\tau}\E^0_{H_{\sst \Vcal}} \left[ e^{-\pi\la H_{\sst \Vcal}^2 \tau^{\half} \arctan(\tau^{\half} )} \right]
=(1+\tau)^{-1} \left(1+\tau^{\half} \arctan(\tau^{\half})\right)^{-3/2},
\label{eq:CP-CF}
\end{align}
using the Laplace transform of $H_{\sst \Vcal}^2$ from~\cite[Lemma~5.6]{FB-SKJ}. 
Note from Theorem~\ref{theorem:si} that, $p^c_{\sst \Vcal}$ is scale invariant with respect to $\la,\mu$.
\end{remark}
\subsubsection{Theorem~\ref{theorem:GPM-WF} (for tropical \textnormal{SINR} at typical mS-MI)}
Under the tropical regime of SINR, the coverage probability at a threshold $\tau$ at a typical handover are
\begin{align}
p_{\sst \Vcal;\Tcal}^c  &{=} \E^0_{H_{\sst \Vcal}}\!\left[e^{-\mu\tau H_{\sst \Vcal}^{\a}\s^2} \!\!\mu\tau H_{\sst \Vcal}^{\a} \!\! \int \!\! e^{- \mu\tau H_{\sst \Vcal}^{\a} x} \left(1{-}e^{-\mu H_{\sst \Vcal}^{\a}x}\right)\right.\nn\\
& \hspace{0.5in}\times \left.e^{-2\la \pi K_\a(\mu,x, H_{\sst \Vcal})} {\rm d}x\right].\nn
\end{align}
{\em Proof.} In this case the tropical interference, given $H_{\sst \Vcal}=h$, is $\Tcal_{h}:=\rho h^{-\a}\vee\max_i \rho_i D_i^{-\a}$. Using the formula (\ref{eq:GCP}) in Theorem~\ref{theorem:GPM-WF}, we have
\begin{align}
p_{\sst{\Vcal;\Tcal}}^c 
&=\int_{0}^\infty e^{-\mu\tau h^{\a}\s^2} \Lcal_{\Tcal_{h}}(\mu\tau h^{\a})f_{H_{\sst \Vcal}}(h){\rm d}h.\nn
\end{align}
Given a handover at a distance $h$, the probability distribution of the tropical interference  is
\begin{align}
\P(\Tcal_{h}\leq x)
&=\E_{\eta_{h},\rho, \{\rho_i\}}\left[\P(I_{h}\leq x)\vert \eta_{h},\rho, \{\rho_i\}\right]\nn\\
&=\E_{\eta_{h}}\left[\E_{\rho, \{\rho_i\}}\left( \one_{\rho h^{-\a}\leq x}\prod_{i\in \N}\one_{\rho_iD_i^{-\a}\leq x}\right)\vert \eta_{h}\right]\nn\\
&=\P_{\rho}\left(\rho h^{-\a}\leq x\right)\E_{\eta_h}\!\!\left[\prod_{i\in \N}\exp\left(\log\E_{\rho_i}\left[\one_{\rho_iD_i^{-\a}\leq x}\right]\right)\vert \eta_{h}\right]\nn\\
&=\P_{\rho}\left(\rho \leq h^{\a}x\right) \E_{\eta_h}\!\!\left[\exp\left( \sum_{i\in \N}\log\E_{\rho_i}\left[\one_{\rho_iD_i^{-\a}\leq x}\right]\right)\vert \eta_{h}\right]\nn\\
&=\left(1{-}e^{-\mu h^{\a}x}\right)\exp\left(-2\la \pi\int_{h}^{\infty}\!\!\!\left(1-\P_\rho(\rho r^{-\a}\leq x)\right)r{\rm d}r\right)\nn\\
&=\left(1{-}e^{-\mu h^{\a}x}\right) \exp\left(-2\la \pi\int_{h}^{\infty}e^{-\mu r^\a x}r{\rm d}r\right) 
=\left(1{-}e^{-\mu h^{\a}x}\right) e^{-2\la \pi K_\a(\mu,x, h)},
\label{eq:STINR-h}
\end{align}
where $K_\a(\mu,x, h)$ is as defined in   (\ref{eq:kalpha}). Thus the Laplace transform of the interference $\Tcal_h$ is 
\begin{align}
\Lcal_{\Tcal_{h}}(\g)&= \g \int e^{- \g x} F_{\Tcal_{h}}(x){\rm d}x.
\label{eq:LapIh}
\end{align}
Then from (\ref{eq:LapIh}) with $\g=\mu\tau h^{\a}$ and (\ref{eq:GCP}), the coverage probability is
\begin{align}
\hspace{1.5in}p_{\sst \Vcal; \Tcal}^c 
&=\int_{0}^\infty \!\!e^{-\mu\tau h^{\a}\s^2}\!\! \mu\tau h^{\a} \int e^{- \mu\tau h^{\a} x} \left(1{-}e^{-\mu h^{\a}x}\right)  e^{-2\la \pi K_\a(\mu,x, h)} {\rm d}x f_{H_{\sst \Vcal}}(h){\rm d}h.\nn  \hspace{0.9in} \hfill \square
\end{align}
%
\begin{remark}[Special case: $\a=2$ and $4$]
For $\a=2$, we have  $K_2(\mu,x, h) = \frac{1}{2\mu x} e^{-\mu h^2 x}$. For $\a=4$ we have
$K_4(\mu,x, h) 
= \frac{1}{4\sqrt{\mu x}}  \Gamma\left(\half,\mu h^4 x\right)$.
\end{remark}
\subsection{Typical time}
At any typical time, let $R$ be the random variable for the nearest BS from the user. It is well known in literature that $R$ follows a Rayleigh distribution with parameter $1/{\sqrt{2\pi\la}}$, i.e., it has  density
\[
f_R(r):=2\la\pi r \, e^{-\la\pi r^2}.
\]
\subsubsection{Theorem~\ref{theorem:GPM-WF} (for \textnormal{SINR} at a typical time)}~\label{subsubsection:t}
In this case the coverage probability with user QoS $\tau$ is
\begin{align}
p_{t}^c  &= \E^0_R\left[e^{-\mu\tau R^{\a}\s^2-\pi\la R^2 \kappa(\tau,\a)}\right],\nn
\end{align}
where the function $\kappa$ is as defined in (\ref{eq:kappa}). 

{\em Proof.} To apply the formula for the $\tau$-coverage probability (\ref{eq:GCP}) we determine the Laplace transform $\Lcal_{I_{r}}(\g)$ with $\g=\mu\tau r^{\a}$ as 
\begin{align}
\Lcal_{I_{r}}(\mu\tau r^{\a})&=\exp\left(-2\pi\la\int_{r}^\infty\frac{\mu\tau r^{\a} v^{-\a}}{\mu+\mu\tau r^{\a} v^{-\a}}v{\rm d}v\right)
=\exp\left(-2\pi\la\int_{r}^\infty\hspace{-0.14in}\frac{\tau r^{\a}}{v^{\a}{+}\tau r^{\a}}v{\rm d}v\right)= e^{-\pi\la r^2 \kappa(\tau,\a)},\nn
\end{align}
where $\kappa$ is as in (\ref{eq:kappa}). Then, from  (\ref{eq:GCP}), the coverage probability is
\begin{align}
p_{t}^c  &= \int_{0}^\infty e^{-\mu\tau r^{\a}\s^2}  e^{-\pi\la r^2 \kappa(\tau,\a)}f_{R}(r){\rm d}r.\nn \hspace{0.7in} \hfill \square
\end{align}
\begin{remark}
Similarly to (\ref{eq:CP-CF}) we also obtain closed form in the interference limited regime for the case $\a=4$ as
$p_{t}^c =\left(1{+}\tau^{\half} \arctan(\tau^{\half})\right)^{-1}$, using the fact that $\E\left[e^{-\g R^2}\right]=\left(1+\frac{\g}{\la\pi}\right)^{-1}$.
\end{remark}
\subsubsection{Theorem~\ref{theorem:GPM-WF} (for STINR at a typical time)}~\label{subsubsection:Tt}
At a typical time, the probability of coverage with user QoS $\tau$ is given by 
\begin{align}
p_{t;\sst{\Tcal}}^c {=} \mu\tau\E_{R}^0\!\!\left[ R^\a \!\!\int \!\! e^{-\mu\tau R^\a (\s^2+x)-2\la \pi K_\a(\mu,x,R)}{\rm d}x\right].
\label{eq:}
\end{align}
\begin{proof}
We find the coverage probability using the formula (\ref{eq:GCP}) we first determine the Laplace transform of $\Tcal_r$ given $R=r$. For this, the probability distribution of $\Tcal_r$ is
\begin{align}
\P(\Tcal_{r}\leq x) &=\E_{\eta_{r}, \{\rho_i\}}\left[\P(\Tcal_{r}\leq x)\vert \eta_{r}, \{\rho_i\}\right]\nn\\
&=\E_{\eta_{r}}\left[\E_{\{\rho_i\}}\left(\prod_{i\in \N}\one_{\rho_i D_i^{-\a}\leq x}\right)\vert \eta_{r}\right]\nn\\
&=\E_{\eta_{r}}\left[\prod_{i\in \N}\exp\left(\log\E_{\rho_i}\left[\one_{\rho_iD_i^{-\a}\leq x}\right]\right)\vert \eta_{r}\right]
=\E\left[\exp\left( \sum_{i\in \N}\log\E_{\rho_i}\left[\one_{\rho_iD_i^{-\a}\leq x}\right]\right)\vert \eta_{r}\right].\nn
\end{align}
Using the PGFl of the Poisson point process $\eta_r$, we have 
\begin{align}
\P(\Tcal_{r}\leq x) &=\exp\left(-2\la \pi\int_{r}^{\infty}\left(1-\P_\rho(\rho v^{-\a}\leq x)\right)v{\rm d}v\right)\nn\\
&=\exp\left(-2\la \pi\int_{r}^{\infty}\P_\rho(\rho> v^\a x)v{\rm d}v\right)
=e^{-2\la \pi\int_{r}^{\infty}e^{-\mu v^\a x}v{\rm d}v} =e^{-2\la \pi K_\a(\mu,x, h)},
\label{eq:STINR-R}
\end{align}
where the function $K_\a(\mu,x, h)$ is defined in (\ref{eq:kalpha}). The Laplace transform of the tropical interference is 
\begin{align}
\E_{\Tcal_{r}}\left[e^{-\g \Tcal_{r}}\right]
&= \g\int e^{-\g x} \P(\Tcal_r\leq x){\rm d}x 
= \g\int e^{-\g x} e^{-2\la \pi K_\a(\mu,x, r)}{\rm d}x,
\label{eq:LT-t-T}
\end{align}
using $\g=\mu\tau r^\a$, the coverage probability turns out to be
\begin{align}
p_{t;\sst{\Tcal}}^c &{=}\mu\tau\!\!\int_{0}^\infty \!\! r^\a e^{-\mu\tau r^\a \s^2}\int \!\! e^{-\mu\tau r^\a x-2\la \pi K_\a(\mu,x,r)}{\rm d}x f_{R}(r){\rm d}r.\nn
\end{align}
\end{proof}
\subsection{Typical max-signal (MS)}
Let $H_{\sst S}$ be the distance to the typical visible head point. It follows from~\cite[Lemma~5.11]{FB-SKJ} that the distance $H_{\sst S}$ to the typical visible head follows a Nakagami distribution with parameter $\left(\half, \frac{1}{2\la\pi}\right)$, i.e.,
$f_{H_{\sst S}}(h)= 2\la^\half e^{-\la\pi h^2}$.
\subsubsection{Theorem~\ref{theorem:GPM-WF} (for \textnormal{SINR} at a typical MS)}~\label{subsubsection:MS}
The $\tau$-coverage probability is
\begin{align}
p_{\sst S}^c  &= \E^0_{H_{\sst S}}\left[e^{-\mu\tau H_{\sst S}^{\a}\s^2-\pi\la H_{\sst S}^2\kappa(\tau,\a)}\right].
\label{eq:cp-V}
\end{align}
\begin{proof}
To find the $\tau$-coverage probability using (\ref{eq:GCP}) we first compute the Laplace transform of the interference $I_h$ given $H_{\sst S}=h$ as 
\begin{align}
\E_{I_{h}}[e^{-\g I_{h}}]
&=\E_{\eta_h,\{\rho_i\}}\left[e^{-\g\sum_{i\in \N} \rho_i D_i^{-\a}} \vert \eta_{h},  \{\rho_i\}\right]
=  \E_{\eta_{h}}\left[\prod_{i\in\N}\E_{\rho_i}[e^{-\g \rho_i D_i^{-\a} }]\vert \eta_{h} \right].
\label{eq:pgfl}
\end{align}
Applying the PGFl with respect to the point process $\eta_h$ in (\ref{eq:pgfl}), we have 
\begin{align}
\E_{I_{h}}[e^{-\g I_{h}}] &=\exp\left(-2\pi\la\int_{h}^\infty\left(1-\E_\rho[e^{-\g \rho r^{-\a}}]\right)r{\rm d}r\right)\nn\\
&=\exp\left(-2\pi\la\int_{h}^\infty\frac{\g r^{-\a}}{\mu+\g r^{-\a}}r{\rm d}r\right) \label{eq:V2a}
\\
&\stackrel{(\g=\mu\tau h^\a)}{=}\exp\left(-2\pi\la\int_{h}^\infty \frac{\tau h^{\a} r^{-\a}}{1+\tau h^{\a} r^{-\a}}r{\rm d}r\right)
=e^{-2\pi\la\int_{h}^\infty\frac{\tau h^{\a}}{r^{\a}+\tau h^{\a}}r{\rm d}r} = e^{-\pi\la h^2 \kappa(\tau,\a)},
\label{eq:V2}
\end{align}
where $\kappa$ as in (\ref{eq:kappa}). Then using (\ref{eq:GCP}), the $\tau$-coverage probability is 
\begin{align}
p_{\sst S}^c  &=\int_{0}^\infty e^{-\mu\tau h^{\a}\s^2-\pi\la h^2\kappa(\tau,\a)}f_{H_{\sst S}}(h){\rm d}h.\nn
\end{align}
\end{proof}
\begin{remark}
Similarly to (\ref{eq:CP-CF}) we also obtain a closed form in the interference limited regime for the case $\a=4$ as
\[
p_{\sst S}^c =\left(1{+}\tau^{\half} \arctan(\tau^{\half})\right)^{-1/2}, 
\]
using the fact from~\cite[Lemma~5.11]{FB-SKJ} that 
$\E\left[e^{-\g H_{\sst S}^2}\right]=\left(1+\frac{\g}{\la\pi}\right)^{-\half}$.
\end{remark}
\subsubsection{Theorem~\ref{theorem:GPM-WF} (for STINR at a typical MS)}~\label{subsubsection:MST}
In this case the $\tau$-coverage probability is
\begin{align}
p_{{\sst S};\sst{\Tcal}}^c &{=}  \E^0_{H_{\sst S}}\!\!\left[e^{-\mu\tau H_{\sst S}^{\a}\s^2} \!\! \mu\tau H_{\sst S}^{\a}  \int e^{-\mu\tau H_{\sst S}^{\a} x-2\la \pi K_\a(\mu,x,H_{\sst S})}{\rm d}x\right].\nn
\end{align}
\begin{proof}
Given $H_{\sst S}=h$, the tropical interference is $\Tcal_h= \max_{i} \rho_i D_i^{-\a}$ and its distribution is
\begin{align}
\P(\Tcal_h\leq x) &=\E_{\eta_{h}, \{\rho_i\}}\left[\P(\Tcal_h \leq x)\vert \eta_{h}, \{\rho_i\}\right]\nn\\
&=\E_{\eta_{h}}\left[\E_{\{\rho_i\}}\left(\prod_{i\in \N}\one_{ \rho_i D_i^{-\a} \leq x}\right)\vert \eta_{h_{\sst I}}\right]\nn\\
&=   \E_{\eta_{h}}\left[\prod_{i\in \N}\E_{\rho_i}\left[\one_{ \rho_i D_i^{-\a} \leq x}\right]\vert \eta_{h}\right]\nn\\
&=   \exp\left(-2\la\pi\int_{h}^\infty \P_{\rho}(\rho>v^\a x) v{\rm d}v\right)\nn\\
&=  \exp\left(-2\la\pi\int_{h}^\infty e^{-\mu v^\a x} v{\rm d}v\right) =  e^{-2\la\pi K_\a(\mu,x, h)}.
\label{eq:dnst-trop-sinr-V}
\end{align}
where $K_\a(\mu,x, h)$ is defined in (\ref{eq:kalpha}). Using the distribution of $\Tcal_h$ from (\ref{eq:dnst-trop-sinr-V}), we obtain the Laplace transform of $\Tcal_h$ as
\begin{align}
\Lcal_{I_h}(\mu\tau h^{\a}) &=  \mu\tau h^{\a}  \int e^{-\mu\tau h^{\a} x-2\la \pi K_\a(\mu,x, h)}{\rm d}x.\nn
\end{align}
We obtain the coverage probability by using (\ref{eq:GCP}).
\end{proof}
\subsection{Typical max-interference (MI)}
Recall that the distance to the nearest interferer and the serving BS is denoted by $H_{\sst I}$ and $R_{\sst I}$, where $H_{\sst I}\stackrel{d}{\sim} \mbox{Na}\left(\frac{3}{2}, \frac{3}{2\la\pi}\right)$ from Proposition~\ref{prop:IVHd} and, from Proposition~\ref{proposition:JHRI}, $R^2_{\sst I}\stackrel{d}{\sim} \mbox{U}[0,h^2]$, given $H_{\sst I}=h$.
\subsubsection{Theorem~\ref{theorem:GPM-WF2} (for \textnormal{SINR} at a typical MI)}~\label{subsubsection:MI}
The coverage probability  for user QoS $\tau$ is
\begin{align}
p^c_{\sst I} &{=} \frac{2\tau^{-2/\a}}{\a}\int_0^\tau \!\frac{z^{2/\a-1}}{1{+}z}\E^0_{H_{\sst I}}\!\! \left[e^{-\mu z H_{\sst I}^\a \s^2-\pi\la H_{\sst I}^2  \kappa(z, \a)}\right]{\rm d}z.\nn
\end{align}
\begin{proof}
In this case, the total interference is 
\[
I_{H_{\sst I}}:=\rho H_{\sst I}^{-\a}+ \sum_{i\in \N} \rho_i D_i^{-\a},
\]
where $H_{\sst I}$ is the distance to the nearest interferer. Given $H_{\sst I}=h$ and distance to the serving BS $R_{\sst I}=r$, the Laplace transform of the interference is 
\begin{align}
\E_{I_h}\left[e^{-\g I_h}\right]
&=\E_{\eta_{h}, \rho, \{\rho_i\}}\left[e^{-\g(\rho h^{-\a}+\sum_{i\in \N} \rho_i D_i^{-\a})} \vert \eta_{h}, \rho, \{\rho_i\}\right]\nn\\
&=\E_{\rho}\left[e^{-\g\rho h^{-\a}}\right]
\E_{\eta_{h}, \{\rho_i\}}\left[\prod_{i\in\N}e^{-\g \rho_i D_i^{-\a} }\vert \eta_{h}, \{\rho_i\}\right]\nn\\
&= \frac{\mu}{\mu+\g h^{-\a}} \E_{\eta_{h}}\left[\prod_{i\in\N}\E_{\rho_i}[e^{-\g \rho_i D_i^{-\a} }]\vert \eta_{h} \right]\nn\\
&= \frac{\mu}{\mu+\g h^{-\a}} e^{-2\pi\la\int_{h}^\infty\frac{\g u^{-\a}}{\mu+\g u^{-\a}}u{\rm d}u}\label{eq:gamma-int}\\
&\stackrel{\g=\mu\tau r^{\a}}{=} \frac{1}{1+\tau r^{\a} h^{-\a}} e^{-2\pi\la\int_{h}^\infty\frac{\tau r^{\a} u^{-\a}}{1+\tau r^{\a} u^{-\a}}u{\rm d}u} = \frac{h^{\a}}{h^{\a}+\tau r^{\a}} e^{-2\pi\la\int_{h}^\infty\frac{\tau r^{\a}}{u^{\a}+\tau r^{\a}}u{\rm d}u}.\label{eq:OPI-2a}
\end{align}
Hence the coverage probability is
\begin{align}
\!\! p_{\sst I}^c &{=} E^0_{H_{\sst I}, R_{\sst I}}\!\! \left[ \!\frac{H_{\sst I}^\a}{H_{\sst I}^\a{+}\tau R_{\sst I}^\a} e^{-\mu\tau R_{\sst I}^{\a}\s^2-2\pi\la\int_{H_{\sst I}}^\infty\frac{\tau R_{\sst I}^{\a}}{u^\a{+} \tau R_{\sst I}^{\a}}u{\rm d}u}\right].
\end{align}
Considering the random variable $R_{\sst I}^2$, the integral in the exponent in (\ref{eq:OPI-2a}) equals 
\begin{align}
\int_h^\infty\!\!\! \frac{\tau r^{\a/2}}{u^{\a}+\tau r^{\a/2}}u{\rm d}u &= \frac{h^2}{2}\kappa\left(\tau (r/h^2)^{\a/2}, \a\right),
\label{eq:hralpha}
\end{align}
using (\ref{eq:kappar-h}). Hence the coverage probability turns out to be 
\begin{align}
p^c_{\sst I} & {=}\int_0^\infty\!\!\! \int_0^{h^2} \!\!\!  \frac{e^{-\mu\tau r^{\a/2}\s^2}}{1+\tau (r/{h^2})^{\a/2}} e^{-\pi\la h^2  \kappa(\tau (r/h^2)^{\a/2}, \a)} \frac{1}{h^2}{\rm d}r f_{H_{\sst I}}(h){\rm d}h\nn\\
&\stackrel{\frac{r}{h^2}=w}{=}  \!\!\!\int_0^1\!\!\! \int_0^\infty\!\!   \frac{e^{-\mu\tau h^\a w^{\a/2}\s^2}}{1+\tau w^{\a/2}} e^{-\pi\la h^2  \kappa(\tau w^{\a/2}, \a)} f_{H_{\sst I}}(h){\rm d}h{\rm d}w \nn\\
&{=} \int_0^1 \!\!\!\frac{1}{1{+}\tau w^{\a/2}}\E^0_{H_{\sst I}}\!\! \left[e^{-\mu\tau w^{\a/2} H_{\sst I}^\a \s^2-\pi\la H_{\sst I}^2  \kappa(\tau w^{\a/2}, \a)}\right]{\rm d}w\nn\\
&{=} \frac{2\tau^{-2/\a}}{\a} \!\!\! \int_0^\tau \!\!\frac{z^{2/\a-1}}{1{+}z}\E^0_{H_{\sst I}}\!\! \left[e^{-\mu z H_{\sst I}^\a \s^2-\pi\la H_{\sst I}^2  \kappa(z, \a)}\right]{\rm d}z,\label{eq:OPI-4}
\end{align}
by taking $\tau w^{\a/2}=z$. 
\end{proof}
\subsubsection{Theorem~\ref{theorem:GPM-WF2} (for STINR at a typical MI)}~\label{subsubsection:MIT}
In this case the $\tau$-coverage probability is
\begin{align}
p^c_{\sst I;\Tcal} &{=} \E^0_{H_{\sst I}, R_{\sst I}}\left[ e^{-\mu\tau R_{\sst I}^{\a}\s^2} 
\mu\tau R_{\sst I}^{\a} \int e^{- \mu\tau R_{\sst I}^{\a} x} \left(1-e^{-\mu H_{\sst I}^\a x}\right) e^{-2\la\pi K_\a(\mu,x,H_{\sst I})}{\rm d}x \right].\nn
\end{align}
\begin{proof}
Given the distance to the nearest interferer $H_{\sst I}=h$, the tropical interference is  $\Tcal_h= \rho h^{-\a}\vee \max_{i} \rho_i D_i^{-\a}$. The probability distribution of $\Tcal_h$ is
\begin{align}
\P(\Tcal_h\leq x) &=\E_{\eta_{h},\rho, \{\rho_i\}}\left[\P(\Tcal_h\leq x)\vert \eta_{h}, \{\rho_i\}\right]\nn\\
&=\E_{\eta_{h}}\left[\E_{\rho, \{\rho_i\}}\left(\rho h^{-\a}\prod_{i\in \N}\one_{ \rho_i D_i^{-\a} \leq x}\right)\vert \eta_{h}\right]\nn\\
&= \P_{\rho}(\rho h^{-\a}\leq x) \E_{\eta_{h}}\left[\prod_{i\in \N}\E_{\rho_i}\left[\one_{ \rho_i D_i^{-\a} \leq x}\right]\vert \eta_{h}\right]\nn\\
&=\left(1-e^{-\mu h^\a x}\right) e^{-2\la\pi\int_{h}^\infty \P_{\rho}(\rho>v^\a x) v{\rm d}v}\nn\\
&=\left(1-e^{-\mu h^\a x}\right) e^{-2\la\pi\int_{h}^\infty e^{-\mu v^\a x} v{\rm d}v}
=\left(1-e^{-\mu h^\a x}\right) e^{-2\la\pi K_\a(\mu,x, h)},
\label{eq:dnst-trop-sinr-T}
\end{align}
where $K_\a$ is as defined in (\ref{eq:kalpha}). We compute the Laplace transform $\Lcal_{\Tcal_h}(\g)$, for $\g=\mu\tau r^{\a}$ given $R_{\sst I}=r$, similarly to (\ref{eq:LapIh}) as  
\begin{align}
\E\left[e^{-\g \Tcal_{h}}\right]&{=} \mu\tau r^{\a}\!\! \int\!\! e^{- \mu\tau r^{\a} x} \left(1{-}e^{-\mu h^\a x}\right) e^{-2\la\pi K_\a(\mu,x, h)}{\rm d}x.
\label{eq:LT-TI}
\end{align}
We obtain the coverage probability using the Laplace transform (\ref{eq:LT-TI}) in the formula (\ref{eq:GCP2}) as 
\begin{align}
p^c_{\sst I;\Tcal} &{=} \E^0_{H_{\sst I}, R_{\sst I}}\left[ e^{-\mu\tau R_{\sst I}^{\a}\s^2} 
\mu\tau R_{\sst I}^{\a} \int e^{- \mu\tau R_{\sst I}^{\a} x} \left(1-e^{-\mu H_{\sst I}^\a x}\right)  e^{-2\la\pi K_\a(\mu,x,H_{\sst I})}{\rm d}x \right]\nn\\
&= \int_0^1 \E_{H_{\sst I}}\left[ e^{-\mu\tau H_{\sst I}^{\a} w^{\a/2}\s^2} 
\mu\tau H_{\sst I}^{\a} w^{\a/2} \int e^{- \mu\tau H_{\sst I}^{\a} w^{\a/2} x} \left(1-e^{-\mu H_{\sst I}^\a x}\right)  e^{-2\la\pi K_\a(\mu,x,H_{\sst I})}{\rm d}x \right] {\rm d}w,\nn
\end{align}
by taking $R^2_{\sst I}/h^2$ as a uniform random variable on $[0,1]$, given $H_{\sst I}=h$. 
\end{proof}
\subsection{Typical min-interference (mI): typical tropical interference handover}
Recall that the distance to the nearest interferer and the serving BS is denoted by $H_{\sst \Ical}$ and $R_{\sst \Ical}$, where $H_{\sst \Ical}\stackrel{d}{\sim} \mbox{Na}\left(\frac{3}{2}, \frac{3}{2\la\pi}\right)$ from Proposition~\ref{proposition:TIH} and, from Proposition~\ref{proposition:JHRTI}, $R^2_{\sst \Ical}\stackrel{d}{\sim} \mbox{U}[0,h^2]$, given $H_{\sst \Ical}=h$.
\subsubsection{Theorem~\ref{theorem:GPM-WF2} (for \textnormal{SINR} at a typical mI)}~\label{subsubsection:mI}
In this case, the $\tau$-coverage probability is
\begin{align}
\!\! p_{\sst \Ical}^c &{=} \frac{2\tau^{-2/\a}}{\a}\int_0^\tau \!\frac{z^{2/\a-1}}{(1{+}z)^2}\E^0_{H_{\sst \Ical}}\!\! \left[e^{-\mu z H_{\sst \Ical}^\a \s^2-\pi\la H_{\sst \Ical}^2  \kappa(z,\a)}\right]{\rm d}z.\nn
\end{align}
\begin{proof}
Suppose the distance to the typical handover is $H_{\sst \Ical}$. Given $H_{\sst \Ical}=h$, let $R_{\sst \Ical}$ be the distance to serving station. In this case the total interference is
\[
I_{h}:=(\rho^++\rho^-) h^{-\a}  + \sum_{i\in \N} \rho_i D_i^{-\a}.
\]
Given $H_{\sst \Ical}=h$ and $R_{\sst \Ical}=r$, the Laplace transform of the interference is
\begin{align}
\E\left[e^{-\g I_{h}}\right] &= \left(\frac{\mu}{\mu+\g h^{-\a}}\right)^2 e^{-2\pi\la\int_{h}^\infty\frac{\g v^{-\a}}{\mu+\g v^{-\a}}v{\rm d}v}\label{eq:LTmI1}\\
&=\left(\frac{h^\a}{h^\a+\tau r^\a}\right)^2 e^{-2\pi\la\int_{h}^\infty\frac{\tau r^{\a}}{v^\a+ \tau r^{\a}}v{\rm d}v},\nn
\end{align}
by substituting $\g=\mu\tau r^\a$. Using the formula (\ref{eq:GCP2}), the coverage probability is 
\begin{align}
p_{\sst \Ical}^c 
&{=}E^0_{H_{\sst \Ical}, R_{\sst \Ical}}\!\! \left[ \!\left(\frac{H_{\sst \Ical}^\a}{H_{\sst \Ical}^\a{+}\tau R_{\sst \Ical}^\a}\right)^2 \!\! e^{-\mu\tau R_{\sst \Ical}^{\a}\s^2-2\pi\la\int_{H_{\sst \Ical}}^\infty\frac{\tau R_{\sst \Ical}^{\a}}{v^\a{+} \tau R_{\sst \Ical}^{\a}}v{\rm d}v}\right].
\label{eq:cpTTIH2}
\end{align}
By considering the random variable $R^2_{\sst \Ical}$ and performing an analysis similarly to (\ref{eq:OPI-4}), we have
\begin{align}
p_{\sst \Ical}^c &{=} \int_0^1 \!\!\!\frac{1}{(1{+}\tau w^{\a/2})^2}\E^0_{H_{\sst \Ical}}\!\! \left[e^{-\mu\tau w^{\a/2} H_{\sst \Ical}^\a \s^2-\pi\la H_{\sst \Ical}^2  \kappa(\tau w^{\a/2}, \a)}\right]{\rm d}w \nn\\
&= \frac{2\tau^{-2/\a}}{\a}\int_0^\tau \!\frac{z^{2/\a-1}}{(1{+}z)^2}\E^0_{H_{\sst \Ical}}\!\! \left[e^{-\mu z H_{\sst \Ical}^\a \s^2-\pi\la H_{\sst \Ical}^2  \kappa(z,\a)}\right]{\rm d}z,\nn
\end{align}
where $\tau w^{\a/2}=z$ and $\kappa(z,\a)$ is defined in using (\ref{eq:kappar-h}).
\end{proof}
\subsubsection{Theorem~\ref{theorem:GPM-WF2} (for STINR at a typical mI)}~\label{subsubsection:mIT}
The $\tau$-coverage probability is
\begin{align}
p^c_{\sst \Ical;\Tcal} &{=} \E^0_{H_{\sst \Ical},R_{\sst \Ical}}\left[e^{-\mu\tau R_{\sst \Ical}^{\a}\s^2}
\mu\tau R_{\sst \Ical}^{\a} \int e^{- \mu\tau R_{\sst \Ical}^{\a} x} \left(1{-}e^{-\mu H_{\sst \Ical}^{\a} x}\right)^2 e^{-2\la\pi K_\a(\mu,x,H_{\sst \Ical})}{\rm d}x\right].\nn 
\end{align}
\begin{proof}
Given $H_{\sst \Ical}=h$, let $R_{\sst \Ical}$ be the distance to serving station. In this case the tropical interference is 
\[
\Tcal_h:= \rho^+ h^{-\a}\vee  \rho^- h^{-\a}  \vee \max_{i\in \N} \rho_i D_i^{-\a}.
\]
We compute the distribution  of $\Tcal_h$, given $H_{\sst \Ical}=h$ similarly to (\ref{eq:dnst-trop-sinr-T}), as
\begin{align}
\P(\Tcal_h\leq x)&=\left(1-e^{-\mu h^\a x}\right)^2 e^{-2\la\pi K_\a(\mu,x, h)}.
\nn
\end{align}
Hence, similarly to (\ref{eq:LT-TI}), the Laplace transform of $\Tcal_h$ is
\begin{align}
\E\left[e^{-\g \Tcal_{h}}\right] &{=} \mu\tau r^{\a} \!\!\int\!\! e^{- \mu\tau r^{\a} x} \left(1{-}e^{-\mu h^\a x}\right)^2 e^{-2\la\pi K_\a(\mu,x, h)}{\rm d}x.\nn
\end{align}
We obtain the coverage probability using the Laplace transform of the interference $\Tcal_h$ in the formula (\ref{eq:GCP2}), with $\g=\mu\tau r^\a$, given $(H_{\sst \Ical}, R_{\sst \Ical})= (h,r)$, as
\begin{align}
p^c_{\sst \Ical;\Tcal} &{=} \E^0_{H_{\sst \Ical},R_{\sst \Ical}}\left[e^{-\mu\tau R_{\sst \Ical}^{\a}\s^2}
\mu\tau R_{\sst \Ical}^{\a} \int e^{- \mu\tau R_{\sst \Ical}^{\a} x} \left(1{-}e^{-\mu H_{\sst \Ical}^{\a} x}\right)^2  e^{-2\la\pi K_\a(\mu,x,H_{\sst \Ical})}{\rm d}x\right].\nn
\end{align}
\end{proof}
\section{Proofs: case without fading}~\label{section:PWFa}
As an opening remark for this section, we use the same set of notations for the distance variables, coverage probabilities and average Shannon rates, as in the case with fading. Recall from Lemma~\ref{lemma:etappp} that at any of the typical epochs of interest, the point process $\eta_h$ of the distance of all other BS is Poisson on $(h, \infty)$ with intensity measure $\hat{\mu}$ with density ${\rm d}\hat{\mu}:=2\pi\la r\, {\rm d}r$, given $H=h$, for $H \in \{H_{\sst \Vcal}, H_{\sst S}, R_{\sst I}, R_{\sst \Ical}, R\}$. 

\subsection{Proof of Theorem~\ref{theorem:GPM-NF}}
The coverage probability is given by
\begin{align}
p^c(\tau,\la,\a)&= \E^0_{H}[\P(\Scal>\tau\vert H)],
\label{eq:WFP1}
\end{align}
where, given $H=h$, the inner term in (\ref{eq:WFP1}) is
\begin{align}
\P(\Scal>\tau\vert H=h)
&= \E_{I_{h}}\left[\P(\Scal>\tau\vert I_h)\right] 
=\P_{I_{h}}\left(\tau  \left(\s^2+I_{h}\right)< h^{-\a}\right)
=\P_{I_{h}}\left(I_{h}<h^{-\a}/{\tau}-\s^2\right)
=F_{I_{h}}\left(0\vee (h^{-\a}/{\tau}-\s^2)\right).
\nn
\end{align}
Then the baseline formula for the coverage probability is 
\begin{align}
\hspace{2in}p^c(\tau,\la,\a)&= \E^0_{H}\left[F_{I_{H}}\left(0\vee (H^{-\a}/{\tau}-\s^2)\right)\right].\nn \hspace{2.3in} \hfill \square
\end{align}
\subsection{Proof of Theorem~\ref{theorem:GPM-NF2}}
The proof follows the same steps as those in Theorem~\ref{theorem:GPM-NF} by considering the distance $R$ to the serving BS and the distance $H$ to the nearest interferer as well. Given $H=h$ and $R=r$, the coverage probability is
\begin{align}
\P(\Scal>\tau\vert H=h, R=r)
= \E_{I_{h}}\left[\P(\Scal>\tau\vert h, I_h)\right]
=\P_{I_{h}}\left(\tau  \left(\s^2+I_{h}\right)< r^{-\a}\right)
&=\P_{I_{h}}\left(I_{h}<r^{-\a}/{\tau}-\s^2\right)\nn\\
&=F_{I_{h}}\left(0\vee (r^{-\a}/{\tau}-\s^2)\right).
\label{eq:NFCP}
\end{align}
By jointly averaging over the distances $R$ and $H$, from (\ref{eq:NFCP}), we have
\begin{align}
\hspace{2.15in} p^c(\tau,\la,\a)&= \E^0_{H,R}\left[F_{I_{H}}\left(0\vee (R^{-\a}/{\tau}-\s^2)\right)\right].\nn  \hspace{2.15in} \hfill\square
\end{align}
\subsection{Proof of Theorem~\ref{theorem:GPM-NF1S}} 
The result can be proved similarly to the last one. \hfill $\square$
\begin{remark}
In the \textnormal{SNR} regime, the coverage probability turns out to be
\[
p^c(\tau,\la,\a) = \P^0\left(H\leq (\tau \s^2)^{-1/\a}\right),
\]
for $H= H_{\sst \Vcal}, H_{\sst S}, R_{\sst I}, R_{\sst \Ical}, R$, under their respective Palm probability measures $\P^0$.
\end{remark}

For simplicity we write $p^c_{*}\equiv p^c_{*}(\tau,\la,\a)$, $p^c_{\sst *;\Tcal}\equiv p^c_{\sst *;\Tcal}(\tau,\la,\a)$ and $\Rcal_{*}\equiv \Rcal_{*}(\la,\a)$, $\Rcal_{\sst *;\Tcal}\equiv \Rcal_{\sst *;\Tcal}(\la,\a)$, which are understood to a function of the set of parameters $\{\tau,\la,\a\}$ and $\{\la,\a\}$, respectively, where $*\in \{\sst{\Vcal, S, I, \Ical, t}$\hspace{-0.01em}$\}$.
\begin{remark}
In all our proofs, we just derive the expression for $p^c_{*}(\tau,\la,\a)$ and $p^c_{\sst *;\Tcal}(\tau,\la,\a)$. The average Shannon rate can be derived using the general formula in (\ref{eq:RcalG}) 
for both $\Rcal_{*}(\la,\a)$ and $\Rcal_{\sst *;\Tcal}(\la,\a)$, where $*\in \{\sst{\Vcal, S, I, \Ical, t}$\hspace{-0.001em}$\}$.
\end{remark}
\begin{remark}
Suppose the contribution of the total interference created by the BSs beyond distance $h$, for some $h\geq 0$, is $I'_h:=\sum_{i\in \N}D_i^{-\a}$. The distribution of $I'_h$ is the same at all typical epochs. Indeed, the Laplace transform of which is given by 
\begin{align}
\E\left[e^{-\nu I'_h}\right]
&{=} \E_{\eta_{h}}\!\!\left[\prod_{i\in\N}e^{-\nu D_i^{-\a} }\vert \eta_h\right] {=} e^{-2\pi\la\int_{h}^\infty\left(1-e^{-\nu r^{-\a}}\right) r{\rm d}r}.
\label{eq:LI}
\end{align}
Using the integral inside the exponent as in (\ref{eq:NFI}) we have 
\begin{align}
\Lcal_{I'_h}(\nu)
& = e^{-\pi\la L_\nu(h, \a)},
\label{eq:LI1}
\end{align}
where $ L_\nu(h, \a):=-  h^2 (1{-}e^{-\nu h^{-\a}}) {+}  \nu^{2/\a}\gamma\left(1{-}\frac{2}{\a}, \nu h^{-\a}\right)$.
\end{remark}
\begin{remark}
Note from (\ref{eq:NFI}) that for $\a=2$ and $4$ we have 
\begin{align}
L_\nu(h, 2)= -  h^2\left(1-e^{-\nu h^{-2}}\right)+ \nu\gamma\left(0, \nu h^{-2}\right),\nn
\end{align}
\begin{align}
L_\nu(h, 4)= - h^2\left(1-e^{-\nu h^{-4}}\right)+  \nu^{1/2}\gamma\left(1/2, \nu h^{-4}\right),\nn
\end{align}
in terms of the incomplete gamma function defined as $\gamma(a,b):=\int_0^b z^{a-1} e^{-z}{\rm d}z$.
\end{remark}
\subsection{Min-signal-max-interference (mS-MI): typical handover}
The typical handover distance $H_{\sst \Vcal}$ is given by a Nakagami distribution with parameter $\left(\frac{3}{2}, \frac{3}{2\la\pi}\right)$ and with density (\ref{eq:pdfhatH0}).
\subsubsection{Theorem~\ref{theorem:GPM-NF} (for \textnormal{SINR} at typical mS-MI)} 
In this case the interference is $I_{h}:=h^{-\a}+\sum_i D_i^{-\a}$, given $H_{\sst \Vcal}=h$. Suppose $F_{I_h}$ is the distribution of the interference $I_h$. By applying the formula (\ref{eq:GCP-NF}) the coverage probability is
\begin{align}
p^c_{\sst \Vcal} &=  \E^0_{H_{\sst \Vcal}}\left[F_{I_{H_{\sst \Vcal}}}\left(0\vee (H_{\sst \Vcal}^{-\a}/{\tau}-\s^2)\right)\right],  
\label{eq:cp-hp-nf}
\end{align}
where the distribution of $I_h=h^{-\a}+I'_h$, given $H_{\sst \Vcal}=h$, can be found using the Laplace transform
\begin{align}
\Lcal_{I_h}(\nu)
& = e^{-\nu h^{-\a}}\Lcal_{I'_h}(\nu) \stackrel{(\ref{eq:LI1})}{=} e^{-\nu h^{-\a} -\pi\la L_\nu(h, \a)}.
\end{align}
\subsubsection{Theorem~\ref{theorem:GPM-NF} (for STINR at typical mS-MI)} 
Given a handover at a distance $H_{\sst \Vcal}=h$, the tropical interference in this case is given by the other BS at a distance $h$, i.e., $\Tcal_h=h^{-\a}$. The coverage probability is 
\begin{align}
p^c_{\sst \Vcal;\Tcal} &{=} \int_0^\infty \one_{h^{-\a}>\tau(\s^2+h^{-\a})} f_{H_{\sst \Vcal}}(h){\rm d}h
= \int_0^\infty \one_{h<\left(\frac{1-\tau}{\tau\s^2}\right)^\frac{1}{\a}} f_{H_{\sst \Vcal}}(h){\rm d}h
= F_{H_{\sst \Vcal}}\left(\left(\frac{1{-}\tau}{\tau\s^2}\right)^\frac{1}{\a}\right).
\label{eq:WFTP2}
\end{align}
which is computed using the density of $H_{\sst \Vcal}$ from (\ref{eq:pdfhatH0}).
\subsection{Typical time}
The distance to the nearest BS follows a Rayleigh density
\[
f_R(r):=2\la\pi r \, e^{-\la\pi r^2}.
\]
\subsubsection{Theorem~\ref{theorem:GPM-NF} (for \textnormal{SINR} at typical time)} 
In this case the distance $R$ to the nearest point is Rayleigh distributed and the interference is 
\[
I_{r}:=\sum_i D_i^{-\a},
\] 
given $R=r$. The coverage probability is given by the formula (\ref{eq:GCP-NF}) as 
\begin{align}
p^c_{t} &= \int_{0}^\infty F_{I_{r}}\left(\left(\frac{ r^{-\a}}{\tau}-\s^2\right)\vee 0\right) f_{R}(r){\rm d}r. 
\label{eq:WFTtime1} 
\end{align}
where $F_{I_r}$ is the distribution of $I_r$ when $R=r$, is given by the Laplace transform as in (\ref{eq:LI1}) 
\[
\Lcal_{I_r}(\nu)
 = e^{-\pi\la L_\nu(r,\a)}.
\] 
%
\subsubsection{Theorem~\ref{theorem:GPM-NF1S} (for STINR at typical time)} 
In the tropical case the interference is given by the second closest station. Let $R_1$ be the distance to the second closest station. In this case the tropical interference is given by $\Tcal= R_1^{-\a}$. From (\ref{eq:GCP-NF2S}) the coverage probability turns out to be
\begin{align}
p^c_{t;\sst{\Tcal}}
&= \int_{0}^\infty F_{\sst \Tcal}\left(\left(r^{-\a}/{\tau}-\s^2\right)\vee 0\right) f_{R}(r){\rm d}r~\label{eq:WFTropV}\\
&= \int_{0}^\infty \!\!\P_{R_1\vert r}\left(R_1{>}\left( \left(r^{-\a}/{\tau}{-}\s^2\right)\vee 0\right)^{-\frac{1}{\a}}\vee r\right) f_{R}(r){\rm d}r\nn\\
&= \int_{0}^{(\tau \s^2)^{-1/\a}} \hspace{-0.4in}\P_{R_1\vert r}\left(R_1>\left(r^{-\a}/{\tau}-\s^2\right)^{-1/\a}\vee r\right) f_{R}(r){\rm d}r
= \int_{0}^{(\tau \s^2)^{-1/\a}} \!\!\!\! \!\!  \int_{\left(h^{-\a}/{\tau}-\s^2\right)^{-1/\a}\vee h}^\infty \!\!\! f_{R_1\vert h}(r){\rm d}r f_{R}(h){\rm d}h.\label{eq:TNFt}
\end{align}
Note that $h^{-\a}/{\tau}-\s^2< h^{-\a}$ for all $\tau\geq 1$ and under this case we have
\begin{align}
p^c_{{\sst t;\Tcal}} &= \int_{0}^{(\tau \s^2)^{-1/\a}} \!\!\!\! \!  \int_{\left(h^{-\a}/{\tau}-\s^2\right)^{-1/\a}}^\infty f_{R_1\vert h}(r){\rm d}r f_{R}(h){\rm d}h\nn\\
&= \int_{0}^{(\tau \s^2)^{-1/\a}} \!\!\!\! \!  \int_{\left(h^{-\a}/{\tau}-\s^2\right)^{-1/\a}}^\infty 4\la^2\pi^2 r h\, e^{-\la\pi r^2} {\rm d}r {\rm d}h\nn\\
&= 2\la\pi \int_{0}^{(\tau \s^2)^{-1/\a}} h   \int_{\la\pi\left(h^{-\a}/{\tau}-\s^2\right)^{-2/\a}}^\infty e^{-z} {\rm d}z {\rm d}h\nn\\
&= 2\la\pi \int_{0}^{(\tau \s^2)^{-1/\a}} h   e^{-\la\pi\left(h^{-\a}/{\tau}-\s^2\right)^{-2/\a}} {\rm d}h = 2\la\pi \tau^{-\frac{2}{\a}} \int_{0}^{\infty}   e^{-\la\pi z^2}  z  (\s^2 z^\a+1)^{-1-\frac{2}{\a}} {\rm d}z,\nn
\end{align}
by making change the variable $h^{-\a}/{\tau}-\s^2$ to $z^{-\a}$ at the last step. For $\tau< 1$, $h^{-\a}/{\tau}-\s^2\geq h^{-\a}$ if and only if $h\leq \left(\frac{\s^2 \tau}{1-\tau}\right)^{-1/\a}$ and hence
\begin{align}
p^c_{{\sst t;\Tcal}} &= \int_{0}^{\left(\frac{\s^2 \tau}{1-\tau}\right)^{-1/\a}} \!\!\!\! \!  \int_{h}^\infty f_{R_1\vert h}(r){\rm d}r f_{R}(h){\rm d}h + \int_{\left(\frac{\s^2 \tau}{1-\tau}\right)^{-1/\a}}^{(\tau \s^2)^{-1/\a}}   \int_{\left(h^{-\a}/{\tau}-\s^2\right)^{-1/\a}}^\infty f_{R_1\vert h}(r){\rm d}r f_{R}(h){\rm d}h\nn\\
&= \P\left(R\leq (1-\tau)^{1/\a}\left(\s^2 \tau\right)^{-1/\a}\right) + \int_{\left(\frac{\s^2 \tau}{1-\tau}\right)^{-1/\a}}^{(\tau \s^2)^{-1/\a}}   \int_{\left(h^{-\a}/{\tau}-\s^2\right)^{-1/\a}}^\infty f_{R_1\vert h}(r){\rm d}r f_{R}(h){\rm d}h,
\label{eq:int1}
\end{align}
since $\left(\frac{\s^2 \tau}{1-\tau}\right)^{-1/\a}\leq (\s^2 \tau)^{-1/\a}$ for $\tau<1$. Using the same change of variable $h^{-\a}/{\tau}-\s^2$ to $z^{-\a}$, the second term in (\ref{eq:int1}) becomes
\begin{align}
\int_{\left(\frac{\s^2 \tau}{1-\tau}\right)^{-1/\a}}^{(\tau \s^2)^{-1/\a}}   \int_{\left(h^{-\a}/{\tau}-\s^2\right)^{-1/\a}}^\infty f_{R_1\vert h}(r){\rm d}r f_{R}(h){\rm d}h &= 2\la\pi \tau^{-\frac{2}{\a}} \int_{\left(\frac{\s^2 \tau}{1-\tau}\right)^{-1/\a}}^{\infty}   e^{-\la\pi z^2}  z  (\s^2 z^\a+1)^{-1-\frac{2}{\a}} {\rm d}z. \nn
\end{align}
In total we have
\begin{align}
p^c_{{\sst t;\Tcal}} &= \P\left(R\leq (1-\tau)^{1/\a}\left(\s^2 \tau\right)^{-1/\a}\right) + 2\la\pi \tau^{-\frac{2}{\a}} \int_{\left(\frac{\s^2 \tau}{1-\tau}\right)^{-1/\a}}^{\infty}   e^{-\la\pi z^2}  z  (\s^2 z^\a+1)^{-1-\frac{2}{\a}} {\rm d}z.
\label{eq:int2}
\end{align}
\subsection{Typical max-signal (MS)}
Let $H_{\sst S}$ be the distance to the typical visible head point. It follows from~\cite[Lemma~5.11]{FB-SKJ} that the distance $H_{\sst S}$ to the typical visible head follows a Nakagami distribution with parameter $\left(\half, \frac{1}{2\la\pi}\right)$, i.e.,
\[
f_{H_{\sst S}}(h)= 2\la^\half e^{-\la\pi h^2}.
\]
\subsubsection{Theorem~\ref{theorem:GPM-NF} (for \textnormal{SINR} at typical MS)} 
Given the distance to the serving BS $H_{\sst S}=h$, the total interference is $I_h=\sum_i D_i^{-\a}$, and suppose its distribution is $F_{I_h}$ given by the Laplace transform
\[
\Lcal_{I_h}(\nu)
= e^{-\pi\la L_\nu(h,\a)}.
\] 
We determine the coverage probability using (\ref{eq:GCP-NF}) as
\begin{align}
p^c_{\sst S} &= \int_{0}^\infty F_{I_{h}}\left(0\vee \left(\frac{ h^{-\a}}{\tau}-\s^2\right)\right) f_{H_{\sst S}}(h){\rm d}h. 
\label{eq:WFV1}
\end{align}
\subsubsection{Theorem~\ref{theorem:GPM-NF1S} (for STINR at typical MS)} 
The tropical interference is given by the second closest station, distance of which is $R_{\sst S}$. Given $H_{\sst S}=h$, we have $\Tcal_h= R_{\sst S}^{-\a}$, where $R_{\sst S}\geq h$ is the distance to the nearest interferer. The distribution of interference is
\begin{align}
\P(\Tcal_h\leq x)=\P(R_{\sst S}^{-\a}\leq x)=\P(R_{\sst S}>x^{-1/\a}).\nn
\end{align}
From (\ref{eq:GCP-NF2S}) the coverage probability $p^c_{{\sst S;\Tcal}}$ turns out to be
\begin{align}
\P\left(R_{\sst S}>\left((H_{\sst S}^{-\a}/{\tau}-\s^2)\vee 0\right)^{-1/\a} \vee H_{\sst S}\right)
&= \int_{0}^\infty \!\!\!\! \!\!\P_{R_{\sst S}\vert h}\left(R_{\sst S}{>}\left((h^{-\a}/{\tau}{-}\s^2)\vee 0\right)^{-1/\a} \vee h\right)\!\! f_{H_{\sst S}}(h){\rm d}h\nn\\
&= \int_{0}^{(\tau \s^2)^{-1/\a}} \!\!\!\! \!\!\!\!\!\!\!\! \!\!\!\! \P_{R_{\sst S}\vert h}\left(R_{\sst S}>\left(h^{-\a}/{\tau}{-}\s^2\right)^{-1/\a}\vee h\right) \!\!f_{H_{\sst S}}(h){\rm d}h\nn\\
&= \int_{0}^{(\tau \s^2)^{-1/\a}} \!\!\!\! \! \! \int_{\left(h^{-\a}/{\tau}-\s^2\right)^{-1/\a}\vee h}^\infty \!\!\! f_{R_{\sst S}\vert h}(r){\rm d}r f_{H_{\sst S}}(h){\rm d}h.
\label{eq:TNFS}
\end{align}
Note that $h^{-\a}/{\tau}-\s^2< h^{-\a}$ for all $\tau\geq 1$ and in this case, we have
\begin{align}
p^c_{{\sst S;\Tcal}} &= \int_{0}^{(\tau \s^2)^{-1/\a}} \!\!\!\! \!  \int_{\left(h^{-\a}/{\tau}-\s^2\right)^{-1/\a}}^\infty f_{R_{\sst S}\vert h}(r){\rm d}r f_{H_{\sst S}}(h){\rm d}h\nn\\
&= \int_{0}^{(\tau \s^2)^{-1/\a}} \!\!\!\! \!  \int_{\left(h^{-\a}/{\tau}-\s^2\right)^{-1/\a}}^\infty 4\la^{3/2}\pi r  e^{-\la\pi r^2}{\rm d}r{\rm d}h \nn\\
&= 2\la^\half\int_{0}^{(\tau \s^2)^{-1/\a}} \!\!\!\! \!  \int_{\la\pi\left(h^{-\a}/{\tau}-\s^2\right)^{-2/\a}}^\infty  e^{-z}{\rm d}z {\rm d}h \nn\\
&= 2\la^\half\int_{0}^{(\tau \s^2)^{-1/\a}} \!\!\!\! \!  e^{-\la\pi\left(h^{-\a}/{\tau}-\s^2\right)^{-2/\a}} {\rm d}h\nn\\
&= 2\la^{1/2} \int_0^{\infty} e^{-\la\pi z^2} \tau^{-\frac{1}{\a}} (\s^2 z^\a+1)^{-1-\frac{1}{\a}} {\rm d}z = 2\la^{1/2}\tau^{-\frac{1}{\a}} \int_0^{\infty} e^{-\la\pi z^2}  (\s^2 z^\a+1)^{-1-\frac{1}{\a}} {\rm d}z.\nn
\end{align}
by changing the variable $h^{-\a}/{\tau}-\s^2$ to $z^{-\a}$. For $\tau< 1$, $h^{-\a}/{\tau}-\s^2\geq h^{-\a}$ if $h\leq \left(\frac{\s^2 \tau}{1-\tau}\right)^{-1/\a}$ and hence
\begin{align}
p^c_{{\sst S;\Tcal}} &= \int_{0}^{\left(\frac{\s^2 \tau}{1-\tau}\right)^{-1/\a}} \!\!\!\! \!  \int_{h}^\infty f_{R_{\sst S}\vert h}(r){\rm d}r f_{H_{\sst S}}(h){\rm d}h + \int_{\left(\frac{\s^2 \tau}{1-\tau}\right)^{-1/\a}}^{(\tau \s^2)^{-1/\a}}   \int_{\left(h^{-\a}/{\tau}-\s^2\right)^{-1/\a}}^\infty f_{R_{\sst S}\vert h}(r){\rm d}r f_{H_{\sst S}}(h){\rm d}h\nn\\
&= \P\left( H_{\sst S}\leq (1-\tau)^{1/\a}\left(\s^2 \tau\right)^{-1/\a}\right) + \int_{\left(\frac{\s^2 \tau}{1-\tau}\right)^{-1/\a}}^{(\tau \s^2)^{-1/\a}}   \int_{\left(h^{-\a}/{\tau}-\s^2\right)^{-1/\a}}^\infty f_{R_{\sst S}\vert h}(r){\rm d}r f_{H_{\sst S}}(h){\rm d}h,
\label{eq:inrs1}
\end{align}
since $\left(\frac{\s^2 \tau}{1-\tau}\right)^{-1/\a}\leq (\s^2 \tau)^{-1/\a}$ for $\tau<1$. Using the conditional density of $R_{\sst S}$, given $H_{\sst S}=h$ we have the inner integral in the second term as
\begin{align}
\int_{\left(h^{-\a}/{\tau}-\s^2\right)^{-1/\a}}^\infty 2\pi\la r e^{-\la\pi(r^2-h^2)}{\rm d}r = e^{\la\pi h^2}  \int_{\la\pi \left(h^{-\a}/{\tau}-\s^2\right)^{-2/\a}}^\infty e^{-z}{\rm d}z = e^{\la\pi h^2 - \la\pi \left(h^{-\a}/{\tau}-\s^2\right)^{-2/\a}}.
\label{eq:inrs2}
\end{align}
Substituting from (\ref{eq:inrs2}) to the second term in (\ref{eq:inrs1}) and using the density of $H_{\sst S}$, we get 
\begin{align}
\int_{\left(\frac{\s^2 \tau}{1-\tau}\right)^{-1/\a}}^{(\tau \s^2)^{-1/\a}}   \int_{\left(h^{-\a}/{\tau}-\s^2\right)^{-1/\a}}^\infty f_{R_{\sst S}\vert h}(r){\rm d}r f_{H_{\sst S}}(h){\rm d}h &= \int_{\left(\frac{\s^2 \tau}{1-\tau}\right)^{-1/\a}}^{(\tau \s^2)^{-1/\a}} e^{- \la\pi \left(h^{-\a}/{\tau}-\s^2\right)^{-2/\a}} 2\la^{1/2} {\rm d}h\nn\\
&= 2\la^{1/2} \int_{\left(\frac{\s^2 \tau}{1-\tau}\right)^{-1/\a}}^{\infty} e^{-\la\pi z^2} \tau^{-\frac{1}{\a}} (\s^2 z^\a+1)^{-1-\frac{1}{\a}} {\rm d}z\nn\\
&= 2\la^{1/2}\tau^{-\frac{1}{\a}} \int_{\left(\frac{\s^2 \tau}{1-\tau}\right)^{-1/\a}}^{\infty} e^{-\la\pi z^2}  (\s^2 z^\a+1)^{-1-\frac{1}{\a}} {\rm d}z.\nn
\end{align}
In total we have 
\begin{align}
p^c_{{\sst S;\Tcal}} &= \P\left( H_{\sst S}\leq (1-\tau)^{1/\a}\left(\s^2 \tau\right)^{-1/\a}\right) + 2\la^{1/2}\tau^{-\frac{1}{\a}} \int_{\left(\frac{\s^2 \tau}{1-\tau}\right)^{-1/\a}}^{\infty} e^{-\la\pi z^2}  (\s^2 z^\a+1)^{-1-\frac{1}{\a}} {\rm d}z.\nn
\end{align}
\subsection{Typical max-interference (MI)}
Suppose the nearest BS and the nearest interferer are at distance $R_{\sst I}$ and $H_{\sst I}$, respectively. 
\subsubsection{Theorem~\ref{theorem:GPM-NF2} (for \textnormal{SINR} at typical MI)} 
In this case given $H_{\sst I}$, the distance to the closest interferer at its nearest position, the total interference is 
\[
I_{H_{\sst I}}:= H_{\sst I}^{-\a}+ \sum_{i\in \N}  D_i^{-\a},
\]
and given $R_{\sst I} =r$, we have $D_i\geq H_{\sst I}\geq r$, for all $D_i\in \eta_{h}$. Given the typical invisible head at distance $H_{\sst I}=h$ the Laplace transform of $I_h$ is
\begin{align}
\Lcal_{I_h}(\nu)
&\stackrel{(\ref{eq:LI1})}{=}e^{-\nu h^{-\a}-\pi\la L_\nu(h, \a)}.
\label{eq:Iint-sinr}
\end{align}
Let $F_{I_{h}}$ be the distribution of the interference $I_{h}$. Using that and the formula (\ref{eq:GCP-NF2}) we find the coverage probability as 
\begin{align}
p^c_{\sst I} &= \E^0_{H_{\sst I},R_{\sst I}}\left[F_{I_{H_{\sst I}}}\left(\left(\frac{ R_{\sst I}^{-\a}}{\tau}-\s^2\right)\vee 0\right) \right].
\label{eq:cpTIH}
\end{align}
\subsubsection{Theorem~\ref{theorem:GPM-NF2} (for STINR at typical MI)} 
In this case, we can only consider the case where $\tau\geq 1$. The tropical interference is given by the distance to the second closest BS $H_{\sst I}$,  $\Tcal= H_{\sst I}^{-\a}$. The distribution of interference is
\begin{align}
F_{\Tcal}(x)&= \P\left(H_{\sst I}^{-\a}\leq x\right)=\P\left(H_{\sst I}\geq x^{-1/\a}\right).\nn
\end{align}
From formula (\ref{eq:GCP-NF2T}), we have
\begin{align}
p^c_{{\sst I ;\Tcal}} &=  \E^0_{H_{\sst I},R_{\sst I}}\left[F_{\Tcal}\left(\left(\frac{ R_{\sst I}^{-\a}}{\tau}-\s^2\right)\vee 0\right) \right]\nn\\
&= \P^0\left(H_{\sst I}\geq \left(\left(\frac{ R_{\sst I}^{-\a}}{\tau}-\s^2\right)\vee 0\right)^{-1/\a}\right)\nn\\
&= \P^0\left(H_{\sst I}\geq \left(\left(\frac{ R_{\sst I}^{-\a}}{\tau}-\s^2\right)\vee 0\right)^{-1/\a}\!\!\! ; R^{-\a}_{\sst I}\leq \s^2\tau \right) +\P^0\left(H_{\sst I}\geq \left(\left(\frac{ R_{\sst I}^{-\a}}{\tau}-\s^2\right)\vee 0\right)^{-1/\a}\!\!\! ; R_{\sst I}^{-\a}\geq\s^2\tau \right)\nn\\
&=\P^0(H_{\sst I}=\infty; R^{-\a}_{\sst I}\leq \s^2\tau) +\P^0\left(R_{\sst I}\leq \left(\tau(H_{\sst I}^{-\a}+\s^2)\right)^{-1/\a}\!\!\! ; R_{\sst I}^{-\a}\geq\s^2\tau \right)\nn\\
&=\P^0\left(R_{\sst I}\leq \left(\tau(H_{\sst I}^{-\a}+\s^2)\right)^{-1/\a} \right).
\label{eq:tr-Ical-nf}
\end{align}
Note that
\begin{align}
\P^0\left( R_{\sst I}\leq \left(\tau(H_{\sst I}^{-\a}+\s^2)\right)^{-1/\a}\right) 
&= \tau^{-2/\a}\E^0_{H_{\sst I}}\left[H^{-2}_{\sst I}\left(H_{\sst I}^{-\a}+\s^2\right)^{-2/\a}\right]
= \tau^{-2/\a}\E^0_{H_{\sst I}}\left[\left(1+\s^2 H_{\sst I}^{\a}\right)^{-2/\a}\right].
\label{eq:WFtropI}
\end{align}
\subsection{Typical min-interference (mI): tropical interference handover}
Suppose the nearest BS and nearest interferer are at distance $R_{\sst \Ical}$ and $H_{\sst \Ical}$. 
\subsubsection{Theorem~\ref{theorem:GPM-NF2} (for \textnormal{SINR} at typical mI)} 

In this case, given $H_{\sst \Ical}$, the total interference is 
\[
I_{H_{\sst \Ical}}:=2H_{\sst \Ical}^{-\a} + \sum_{i\in \N} D_i^{-\a}. 
\]
Then using (\ref{eq:GCP-NF2}), the coverage probability is given by 
\begin{align}
p^c_{\sst \Ical} &= \E^0_{H_{\sst \Ical},R_{\sst \Ical}}\left[F_{I_{H_{\sst \Ical}}}\left(\frac{ R_{\sst \Ical}^{-\a}}{\tau}-\s^2\right) \right],
\label{eq:TI-cpT}
\end{align}
where the distribution of the interference $I_h$ given $H_{\sst \Ical}= h$, is computed similarly to (\ref{eq:Iint-sinr}) using the Laplace transform as
\begin{align}
\Lcal_{I_h}(\nu)
&\stackrel{(\ref{eq:LI1})}{=}e^{-2\nu h^{-\a}-\pi\la L_\nu(h, \a)}.
\nn
\end{align}
\subsubsection{Theorem~\ref{theorem:GPM-NF2} (for STINR at typical mI)} 
This is only under the case where $\tau\geq 1$.
In this case, the tropical interference is $\Tcal:= H_{\sst \Ical}^{-\a}$. The distribution of the interference is
\begin{align}
F_{\Tcal}(x)&= \P\left(H_{\sst \Ical}^{-\a}\leq x\right)=\P\left(H_{\sst \Ical}\geq x^{-1/\a}\right).
\nn
\end{align}
The coverage probability is computed similarly to (\ref{eq:tr-Ical-nf}) and (\ref{eq:WFtropI}) and using formula (\ref{eq:GCP-NF2T}) as
\begin{align}
p^c_{{\sst \Ical;\Tcal}} &=  \E^0_{H_{\sst \Ical},R_{\sst \Ical}}\left[F_{\Tcal}\left(\left(\frac{ R_{\sst \Ical}^{-\a}}{\tau}-\s^2\right) \vee 0 \right)\right]\nn\\
&=\P^0\left(H_{\sst \Ical}\geq \left(\left(\frac{ R_{\sst \Ical}^{-\a}}{\tau}-\s^2\right)\vee 0\right)^{-\frac{1}{\a}}\right)\nn\\
&=  \P^0\left(R_{\sst \Ical}\leq \left(\tau( H_{\sst \Ical}^{-\a}+\s^2)\right)^{-1/\a}\right)
= \tau^{-2/\a}\E^0_{H_{\sst \Ical}}\left[H^{-2}_{\sst \Ical}\left(H_{\sst \Ical}^{-\a}+\s^2\right)^{-2/\a}\right]
= \tau^{-2/\a}\E^0_{H_{\sst \Ical}}\left[\left(1+\s^2 H_{\sst \Ical}^{\a}\right)^{-2/\a}\right].
\label{eq:cp-nftrop}
\end{align}
\begin{remark}
One can consider the interference from both the closest interferer at typical min interference epoch. This lead to discontinuity in the SINR as well as the Shannon rate.
\end{remark}
\section{Proofs: Comparison of performance metrics}~\label{section:CPE}
\subsection{Theorem~\ref{theorem:SINR-c} (with fading)}~\label{subsection:CF}

\textbf{Part~\ref{CP-F1}.}~\label{subsection:CPE-WF}
All these inequalities can be obtained by comparing the coverage probabilities  
\[
p_{\sst \Vcal}^c = \frac{1}{1+\tau}\int_{0}^\infty e^{-\mu\tau h^{\a}\s^2-\pi\la h^2 \kappa(\tau,\a)} f_{H_{\sst \Vcal}}(h){\rm d}h,
\]
\[
p_{t}^c = \int_{0}^\infty e^{-\mu\tau r^{\a}\s^2-\pi\la r^2 \kappa(\tau,\a)} f_{R}(r){\rm d}r,\;\;\;
p_{\sst S}^c  =\int_{0}^\infty e^{-\mu\tau h^{\a}\s^2-\pi\la h^2\kappa(\tau,\a)}f_{H_{\sst S}}(h){\rm d}h,
\]
and comparing the densities of $H_{\sst \Vcal}, R, H_{\sst S}$. Note that the integrand in all the integrals are same and we define
\begin{align}
G(h):=e^{-\mu\tau h^{\a}\s^2-\pi\la h^2\kappa(\tau,\a)},
\label{eq:Gh}
\end{align}
writing it in same variable $h$. Recall the probability densities of the random variables  $H_{\sst \Vcal}$, $R$ and $H_{\sst S}$, 
$f_{H_{\sst \Vcal}}(h)=4\pi\la^{3/2} h^2 e^{-\la\pi h^2}$, $f_R(h)=2\la\pi h \, e^{-\la\pi h^2}$ and $f_{H_{\sst S}}(h)=2\la^{1/2} e^{-\la\pi h^2}$, for $h\geq 0$, respectively. 
Observe that $f_{H_{\sst \Vcal}}(h)\leq f_R(h)$ on $[0, 1/{2\sqrt{\la}}]$ and $f_{H_{\sst \Vcal}}(h)>f_R(h)$ on $(1/{2\sqrt{\la}}, \infty)$. Hence 
\begin{align}
p_t^c -(1+\tau) p_{\sst \Vcal}^c 
&= \int_0^{1/2\sqrt{\la}} G(h) \left[f_R(h)-f_{H_{\sst \Vcal}}(h)\right]{\rm d}h - \int_{1/2\sqrt{\la}}^\infty G(h) \left[f_{H_{\sst \Vcal}}(h)- f_R(h)\right]{\rm d}h\nn\\
&\geq G(1/2\sqrt{\la}) \int_0^\infty  \!\!\!\!\left[f_{H_{\sst \Vcal}}(h)- f_R(h)\right]{\rm d}h =0,\nn 
\end{align}
which implies that $p_t^c \geq (1+\tau) p_{\sst \Vcal}^c \geq  p_{\sst \Vcal}^c$, since $G(1/2\sqrt{\la})\geq 0$ and $\tau\geq 0$. 

For the other inequality, note that $f_R(h)\leq f_{H_{\sst S}}(h)$ on $[0, 1/{\pi\sqrt{\la}}]$ and $f_R(h)>f_{H_{\sst S}}(h)$ on $(1/{\pi\sqrt{\la}}, \infty)$. Then by similar comparison we have that $p_{\sst S}^c \geq p_{t}^c $.

For the second part, we compare the coverage probabilities 
\begin{align}
p^c_{\sst I}&{=} \E^0_{H_{\sst I}, R_{\sst I}}\left[ \frac{H_{\sst I}^{\a}}{H_{\sst I}^{\a}{+}\tau R_{\sst I}^{\a}} e^{-\mu\tau R_{\sst I}^{\a}\s^2-2\pi\la\int_{H_{\sst I}}^\infty\frac{\tau R_{\sst I}^{\a}}{r^{\a}{+}\tau R_{\sst I}^{\a}}r{\rm d}r} \right], \nn
\end{align}
\begin{align}
p_{\sst \Ical}^c 
&{=}E^0_{H_{\sst \Ical}, R_{\sst \Ical}}\!\!\left[\left(\frac{H_{\sst \Ical}^\a}{H_{\sst \Ical}^\a{+}\tau R_{\sst \Ical}^\a}\right)^2\!\! e^{-\mu\tau R_{\sst \Ical}^{\a}\s^2-2\pi\la\int_{H_{\sst \Ical}}^\infty\frac{\tau R_{\sst \Ical}^{\a}}{v^\a{+} \tau R_{\sst \Ical}^{\a}}v{\rm d}v}\right].\nn
\end{align}
The densities of $H_{\sst I}$ and $H_{\sst \Ical}$ are
\[
f_{H_{\sst I}}(h)=4\pi\la^{3/2} h^2 e^{-\la\pi h^2}, f_{H_{\sst \Ical}}(h)=\frac{8}{3}\pi^2\la^{5/2} h^4 e^{-\la\pi h^2},
\]
respectively for all $h\geq 0$. Note that $f_{H_{\sst \Ical}}(h)\leq f_{H_{\sst I}}(h)$ on $[0, {\sqrt{3/{2\la\pi}}}]$ and $f_{H_{\sst \Ical}}(h)>f_{H_{\sst I}}(h)$ on $(\sqrt{3/{2\la\pi}}, \infty)$. Define
\[
G_1(h,r):= \frac{h^\a}{h^\a+\tau r^\a} e^{-\mu\tau r^{\a}\s^2- 2\pi\la\int_{h}^\infty\frac{\tau r^{\a}}{v^\a+ \tau r^{\a}}v{\rm d}v},
\]
\[
G_2(h,r):=\left(\frac{h^\a}{h^\a+\tau r^\a}\right)^2e^{-\mu\tau r^{\a}\s^2- 2\pi\la\int_{h}^\infty\frac{\tau r^{\a}}{v^\a+ \tau r^{\a}}v{\rm d}v}.
\]
Observe that $G_2(h,r)\leq G_1(h,r)$ for all pair $(h,r)$, because of the extra factor of the form $\frac{h^\a}{h^\a+\tau r^\a}\leq 1$, in $G_2(h,r)$, for all $h,r,\tau\in \R^+$. Let us define
\begin{align}
G_1(h):=\E_{R_{\sst I}\vert h}[G_1(h,R_{\sst I}\vert h)], 
G_2(h):= \E_{R_{\sst \Ical}\vert h}[G_2(h,R_{\sst \Ical}\vert h)],
\label{eq:G1G2-def}
\end{align}
where  $G_2(h)\leq G_1(h)$ for all $h\geq 0$. Then
\begin{align}
p_{\sst I}^c - p_{\sst \Ical}^c &= \int_0^\infty \int_0^h G_1(h,r) f_{R_{\sst I}\vert h}(r){\rm d}r f_{H_{\sst I}}(h) {\rm d}h  -\int_0^\infty \int_0^h G_2(h,r)f_{R_{\sst \Ical}\vert h}(r){\rm d}r f_{H_{\sst \Ical}}(h) {\rm d}h\nn\\
&= \int_0^\infty \E_{R_{\sst I}\vert h}[G_1(h,R_{\sst I}\vert h)]f_{H_{\sst I}}(h) {\rm d}h  -\int_0^\infty \E_{R_{\sst \Ical}\vert h}[G_2(h,R_{\sst \Ical}\vert h)] f_{H_{\sst \Ical}}(h) {\rm d}h\nn\\
&\stackrel{(\ref{eq:G1G2-def})}{=} \int_0^\infty G_1(h) f_{H_{\sst I}}(h) {\rm d}h  -\int_0^\infty G_2(h) f_{H_{\sst \Ical}}(h) {\rm d}h\nn\\
&\geq \int_0^\infty G_2(h) f_{H_{\sst I}}(h) {\rm d}h  -\int_0^\infty G_2(h) f_{H_{\sst \Ical}}(h) {\rm d}h\nn\\
&= \int_0^{\sqrt{3/{2\la\pi}}} G_2(h) \left[f_{H_{\sst I}}(h)- f_{H_{\sst \Ical}}(h)\right] {\rm d}h   -\int_{\sqrt{3/{2\la\pi}}}^\infty G_2(h) \left[f_{H_{\sst \Ical}}(h)- f_{H_{\sst I}}(h)\right] {\rm d}h\nn\\
&\geq G_2(\sqrt{3/{2\la\pi}}) \int_0^\infty\!\!  \left[f_{H_{\sst I}}(h){-} f_{H_{\sst \Ical}}(h)\right] {\rm d}h =0,\nn
\end{align}
since $G_2(\sqrt{3/{2\la\pi}})\geq 0$. This proves the result for the non-tropical case. \hfill$\square$

\textbf{Part~\ref{CP-F2}.}
The results $p_{\sst{\Vcal;\Tcal}}^c  \leq p_{t;\sst{\Tcal}}^c \leq p_{{\sst S};\sst{\Tcal}}^c \mbox{ and } p^c_{{\sst \Ical ; \Tcal}}  \leq p^c_{{\sst I;\Tcal}}$, in the tropical case with fading can be proved similarly. \hfill$\square$
\subsection{Theorem~\ref{theorem:SINR-c} (without fading)}~\label{subsection:CPE-NF}

\textbf{Part~\ref{CP-F1}.}
The proof essentially follows by comparing the random variables $H_{\sst \Vcal}$, $R$ and $H_{\sst S}$.
Note that $F_{I_{H_{\sst \Vcal}}}(x)\leq F_{I_R}(x)=F_{I_{\sst S}}(x)$ for any $x\geq 0$, since 
\[
I_{H_{\sst \Vcal}=h}= h^{-\a} + \sum_i D_i^{-\a}, I_{R=h}= \sum_i D_i^{-\a} \stackrel{d}{=} I_{H_{\sst S}=h}.
\]
We prove the result by comparing the coverage probabilities 
\[
p^c_{\sst \Vcal} =  \E^0_{H_{\sst \Vcal}}\left[F_{I_{H_{\sst \Vcal}}}\left(0\vee (H_{\sst \Vcal}^{-\a}/{\tau}-\s^2)\right)\right], 
\;\;
p^c_t =  \E^0_R\left[F_{I_R}\left(0\vee (R^{-\a}/{\tau}-\s^2)\right)\right], 
\;\;
p^c_{\sst S} =  \E^0_{H_{\sst S}}\left[F_{I_{H_{\sst S}}}\left(0\vee (H_{\sst S}^{-\a}/{\tau}-\s^2)\right)\right],
\]
and using the technique similar to the first part of Theorem~\ref{theorem:SINR-c}, with
$G(h)= F_{I_{H=h}}\left(0\vee (h^{-\a}/{\tau}-\s^2)\right)$, for any $H\in \{H_{\sst \Vcal}, R, H_{\sst S}\}$.

For this part, we use a comparison argument similar to second part of Theorem~\ref{theorem:SINR-c} for 
\[
p^c_{\sst I} = \E^0_{H_{\sst I},R_{\sst I}}\left[F_{I_{H_{\sst I}}}\left(0\vee\left( R_{\sst I}^{-\a}/{\tau}-\s^2\right)\right) \right],
\;\;
p^c_{\sst \Ical} = \E^0_{H_{\sst \Ical},R_{\sst \Ical}}\left[F_{I_{H_{\sst \Ical}}}\left(0\vee\left( R_{\sst \Ical}^{-\a}/{\tau}-\s^2\right)\right) \right],
\]
from (\ref{eq:cpTIH}) and (\ref{eq:TI-cpT}), respectively. Define
\[
G_1(h,r):= F_{I_{H_{\sst I}=h}}\left(0\vee\left(r^{-\a}/{\tau}-\s^2\right)\right),
\]
\[
G_2(h,r)= F_{I_{H_{\sst \Ical}=h}}\left(0\vee\left(r^{-\a}/{\tau}-\s^2\right)\right).
\]
Observe that by definition of the interference $I_{H_{\sst I}=h}$ and $I_{H_{\sst \Ical}=h}$ at a distance $h$ and the expressions in (\ref{eq:cpTIH}) and (\ref{eq:TI-cpT}) we have $G_2(h,r)\leq G_1(h,r)$ for all $h\geq 0$. Also
\begin{align}
G_2(h) &:= \E_{R_{\sst \Ical}\vert h}[G_2(h,R_{\sst \Ical})] 
\leq \E_{R_{\sst I}\vert h}[G_1(h,R_{\sst I})]=: G_1(h).
\label{eq:G1G2-I}
\end{align}
Then by comparing the two densities 
\[
f_{H_{\sst I}}(h)=4\pi\la^{3/2} h^2 e^{-\la\pi h^2}, f_{H_{\sst \Ical}}(h)=\frac{8}{3}\pi^2\la^{5/2} h^4 e^{-\la\pi h^2},
\]
we obtain that
\begin{align}
p_{\sst I}^c - p_{\sst \Ical}^c 
%
%
%
&= \int_0^\infty \E_{R_{\sst I}\vert h}[G_1(h,R_{\sst I})]f_{H_{\sst I}}(h) {\rm d}h  -\int_0^\infty \E_{R_{\sst \Ical}\vert h}[G_2(h,R_{\sst \Ical})] f_{H_{\sst \Ical}}(h) {\rm d}h.
\label{eq:G1G21-I}
\end{align}
Using the definition (\ref{eq:G1G2-I}),
the last expression in (\ref{eq:G1G21-I}) is equal to
\begin{align}
\int_0^\infty G_1(h) f_{H_{\sst I}}(h) {\rm d}h  -\int_0^\infty G_2(h) f_{H_{\sst \Ical}}(h) {\rm d}h
&\geq \int_0^\infty G_2(h) f_{H_{\sst I}}(h) {\rm d}h  -\int_0^\infty G_2(h) f_{H_{\sst \Ical}}(h) {\rm d}h\nn\\
&= \int_0^{\sqrt{3/{2\la\pi}}} G_2(h) \left[f_{H_{\sst I}}(h)- f_{H_{\sst \Ical}}(h)\right] {\rm d}h   -\int_{\sqrt{3/{2\la\pi}}}^\infty G_2(h) \left[f_{H_{\sst \Ical}}(h)- f_{H_{\sst I}}(h)\right] {\rm d}h\nn\\
&\geq G_2(\sqrt{3/{2\la\pi}}) \!\!\int_0^\infty \!\!\!\!  (f_{H_{\sst I}}(h)- f_{H_{\sst \Ical}}(h)) {\rm d}h =0,\nn
\end{align}
where by using the definition of $G_2$, it can be shown that $G_2(\sqrt{3/{2\la\pi}})\geq 0$. Indeed, using the fact that for $h=\sqrt{3/{2\la\pi}}$, $R_{\sst \Ical}^2\stackrel{d}{\sim} \mbox{U}\left[0, \sqrt{3/{2\la\pi}}\right]$, from (\ref{eq:TI-cpT}) we have
\begin{align}
G_2(\sqrt{3/{2\la\pi}})
& = \sqrt{2\la\pi/3} \int_0^{\sqrt{3/{2\la\pi}}} \hspace{-0.2in} F_{I_{H_{\sst \Ical}=\sqrt{3/{2\la\pi}}}}\left(0\vee\left(\frac{r^{-\a/2}}{\tau}-\s^2\right)\right){\rm d}r\nn\\
& = \sqrt{2\la\pi/3} \int_0^{(\tau \s^2)^{-1/\a} \wedge \sqrt{3/{2\la\pi}}} 
\hspace{-0.2in}
F_{I_{H_{\sst \Ical}=\sqrt{3/{2\la\pi}}}}\left(\frac{r^{-\a/2}}{\tau}-\s^2\right){\rm d}r,\nn
\end{align}
which is non-negative, and hence we have the result. 
\hfill $\square$

\textbf{Part~\ref{CP-F2}.}
We establish the comparison of the coverage probabilities
\[
p^c_{\sst \Vcal; \Tcal} =  \E^0_{H_{\sst \Vcal}}\left[\one_{H^{-\a}_{\sst \Vcal}<H_{\sst \Vcal}^{-\a}/{\tau}-\s^2}\right], 
\;\;
p^c_{t;\sst \Tcal} =  \E^0_{R, R_1}\left[\one_{R_1^{-\a}< R^{-\a}/{\tau}-\s^2}\right], \;\; 
p^c_{\sst S;\Tcal} =  \E^0_{H_{\sst S}, R_{\sst S}}\left[\one_{R_{\sst S}^{-\a}< H_{\sst S}^{-\a}/{\tau}-\s^2}\right].
\]
Let us define
\[
G(h)= \one_{h^{-\a}<h^{-\a}/\tau-\s^2}, \; G_1(h,R_1):= \one_{R_1^{-\a}<h^{-\a}/{\tau}-\s^2},\; 
G_2(h,R_{\sst S})= \one_{R_{\sst S}^{-\a}< h^{-\a}/{\tau}-\s^2}.
\]
Let us first prove that $p^c_{\sst S;\Tcal} \geq p^c_{t;\sst\Tcal}$. Note that 
\begin{align}
G_1(h):= \E_{R_1\vert h}\left[ G_1(h,R_1)\right] = \E_{R_{\sst S}\vert h}\left[G_2(h,R_{\sst S})\right] =: G_2(h),\nn
\end{align}
since $R_1\vert h$ and $R_{\sst S}\vert h$ have the same pdf. Using the densities of $R$ and $H_{\sst S}$, we have 
\begin{align}
p^c_{\sst S;\Tcal}-p^c_{t;\sst\Tcal} &= \int_{0}^{\infty} G_2(h) \left(f_{H_{\sst S}}(h) - f_{R}(h)\right){\rm d}h\nn\\
&= \int_{0}^{1/{\la^\half \pi}} G_2(h) \left(f_{H_{\sst S}}(h) - f_{R}(h)\right){\rm d}h - \int_{1/{\la^\half \pi}}^{\infty} G_2(h) \left( f_{R}(h)-f_{H_{\sst S}}(h)\right){\rm d}h\nn\\
&\geq G_2\left(1/{\la^\half \pi}\right) \int_{0}^{\infty} \left(f_{H_{\sst S}}(h) - f_{R}(h)\right){\rm d}h =0,\nn
\end{align}
since $G_2\left(1/{\la^\half \pi}\right)>0$. Hence $p^c_{\sst S;\Tcal} \geq p^c_{t;\sst\Tcal}$ for any $\tau\geq 0$. 
To prove $p^c_{t;\sst \Tcal} \geq p^c_{\sst\Vcal;\Tcal}$, note that given $H_{\Vcal}=h$ and $R=h$, with a common variable $h$, we have that $R_1>h$. This implies that $R_1^{-\a}<h^{-\a}$ and  
\begin{align}
G(h)= \one_{h^{-\a}<h^{-\a}/\tau-\s^2} \leq \one_{R_1^{-\a}< h^{-\a}/{\tau}-\s^2}= G_1(h,R_1).\nn 
\end{align}
Then
\begin{align}
p^c_{t;\sst\Tcal}- p^c_{\sst \Vcal;\Tcal} &= \int_{0}^{\infty} \int_{h}^{\infty} G_1(h,R_1) f_{R_1|h}(r){\rm d}r  f_{R}(h){\rm d}h - \int_{0}^{\infty} G(h) f_{H_{\sst \Vcal}}(h){\rm d}h\nn\\
&\geq  \int_{0}^{\infty} \int_{h}^{\infty} G(h) f_{R_1|h}(r){\rm d}r  f_{R}(h){\rm d}h - \int_{0}^{\infty} G(h) f_{H_{\sst \Vcal}}(h){\rm d}h  =\int_{0}^{\infty} G(h)  \left(f_{R}(h) - f_{H_{\sst \Vcal}}(h)\right){\rm d}h.\nn
\end{align}
Comparing the densities $f_{H_{\sst \Vcal}}(h)=4\pi\la^{3/2} h^2 e^{-\la\pi h^2}$ and $f_{R}(h)=2\la\pi h e^{-\la\pi h^2}$, we have 
\begin{align}
p^c_{t;\sst\Tcal}- p^c_{\sst \Vcal;\Tcal}
&=\int_{0}^{1/{2\sqrt{\la}}} G(h)  \left(f_{R}(h) {-} f_{H_{\sst \Vcal}}(h)\right){\rm d}h {-} \int_{1/{2\sqrt{\la}}}^{\infty} G(h)  \left(f_{H_{\sst \Vcal}}(h) {-} f_{R}(h)\right){\rm d}h\nn\\
&\geq  G(1/{2\sqrt{\la}}) \int_{0}^{\infty}  \left(f_{R}(h) {-} f_{H_{\sst \Vcal}}(h)\right){\rm d}h=0,\nn
\end{align}
since $G(1/{2\sqrt{\la}})\geq 0$. Hence we have the ordering that $p^c_{t;\sst\Tcal} \geq p^c_{\sst \Vcal;\Tcal}$, for any $\tau\geq 0$.

The result $p^c_{{\sst \Ical ; \Tcal}}  \leq p^c_{{\sst I;\Tcal}}$ in the tropical case can be proved similarly. For this part, we use a comparison argument similar to second part of Theorem~\ref{theorem:SINR-c} for 
\[
p^c_{\sst I} = \E^0_{H_{\sst I},R_{\sst I}}\left[F_{I_{H_{\sst I}}}\left(0\vee\left( R_{\sst I}^{-\a}/{\tau}-\s^2\right)\right) \right],
\;\;
p^c_{\sst \Ical} = \E^0_{H_{\sst \Ical},R_{\sst \Ical}}\left[F_{I_{H_{\sst \Ical}}}\left(0\vee\left( R_{\sst \Ical}^{-\a}/{\tau}-\s^2\right)\right) \right].
\]
Define
\[
G_1(h,R_{\sst I}):= F_{I_{H_{\sst I}=h}}\left(0\vee\left(R_{\sst I}^{-\a}/{\tau}-\s^2\right)\right),
\]
\[
G_2(h,R_{\sst \Ical})= F_{I_{H_{\sst \Ical}=h}}\left(0\vee\left(R_{\sst \Ical}^{-\a}/{\tau}-\s^2\right)\right).
\]
Observe that by definition of $I_{H_{\sst I}=h}$ and $I_{H_{\sst \Ical}=h}$ and the expressions in (\ref{eq:WFtropI}) and (\ref{eq:cp-nftrop}) we have $G_2(h,R_{\sst I})\leq G_1(h,R_{\sst \Ical})$ for all $h$. Moreover
\begin{align}
G_2(h) &:= \E_{R_{\sst \Ical}\vert h}[G_2(h,R_{\sst \Ical})]
= \tau^{-2/\a} \left(1+\s^2 h^{\a}\right)^{-2/\a} = \E_{R_{\sst I}\vert h}[G_1(h,R_{\sst I})]=: G_1(h).
\label{eq:G1G2}
\end{align}
Then by comparing the two densities
\[
f_{H_{\sst I}}(h)=4\pi\la^{3/2} h^2 e^{-\la\pi h^2}, f_{H_{\sst \Ical}}(h)=\frac{8}{3}\pi^2\la^{5/2} h^4 e^{-\la\pi h^2},
\]
we have
\begin{align}
p_{\sst I}^c - p_{\sst \Ical}^c 
%
%
%
&= \int_0^\infty \E_{R_{\sst I}\vert h}[G_1(h,R_{\sst I})]f_{H_{\sst I}}(h) {\rm d}h  -\int_0^\infty \E_{R_{\sst \Ical}\vert h}[G_2(h,R_{\sst \Ical})] f_{H_{\sst \Ical}}(h) {\rm d}h.
\label{eq:G1G21}
\end{align}
Using the definition (\ref{eq:G1G2}),
the last term in (\ref{eq:G1G21}) is equal to
\begin{align}
\int_0^\infty G_1(h) f_{H_{\sst I}}(h) {\rm d}h  -\int_0^\infty G_2(h) f_{H_{\sst \Ical}}(h) {\rm d}h
&= \int_0^\infty G_2(h) f_{H_{\sst I}}(h) {\rm d}h  -\int_0^\infty G_2(h) f_{H_{\sst \Ical}}(h) {\rm d}h\nn\\
&= \int_0^{\sqrt{3/{2\la\pi}}}\!\!\!\! G_2(h) \left[f_{H_{\sst I}}(h){-} f_{H_{\sst \Ical}}(h)\right] {\rm d}h   {-} \int_{\sqrt{3/{2\la\pi}}}^\infty \! G_2(h) \left[f_{H_{\sst \Ical}}(h) {-} f_{H_{\sst I}}(h)\right] {\rm d}h\nn\\
&\geq G_2(\sqrt{3/{2\la\pi}}) \!\!\int_0^\infty \!\!\!\!  (f_{H_{\sst I}}(h)- f_{H_{\sst \Ical}}(h)) {\rm d}h =0,\nn
\end{align}
where, by definition of $G_2$, it can be shown that $G_2(\sqrt{3/{2\la\pi}})\geq 0$. Indeed for $h=\sqrt{3/{2\la\pi}}$, from (\ref{eq:cp-nftrop}) we have
\begin{align}
G_2(\sqrt{3/{2\la\pi}})
&= \tau^{-2/\a}\left[\left(1+\s^2 h^{\a}\right)^{-2/\a}\right]_{h=\sqrt{3/{2\la\pi}}} = \tau^{-2/\a}\left[\left(1+\s^2\left({3}/{2\la\pi}\right)^{\a/2}\right)^{-2/\a}\right],\nn
\end{align}
which is non-negative for any $\la, \tau, \s>0$ and $\a$. Hence we have the result. 
\hfill $\square$
\section{Comparison of interference}~\label{Int-order}
\subsection{Theorem~\ref{theorem:Int-order} (with fading)}
For  $I_{\sst \Vcal}\geq_{\textnormal{L}} I_{t} \geq_{\textnormal{L}} I_{\sst S}$ we compare the Laplace transforms
\begin{align}
\E^0\left[e^{-\g I_{\sst \Vcal}}\right] &= \E^0_{H_{\sst \Vcal}}\!\!\left[ \frac{\mu}{\mu{+}\g H_{\sst \Vcal}^{-\a}}e^{-2\pi\la\int_{H_{\sst \Vcal}}^\infty\frac{\g r^{-\a}}{\mu+\g r^{-\a}}r{\rm d}r}\right],
\label{eq:int-1}
\end{align}
\[
\E^0\left[e^{-\g I_t}\right] = \E^0_{R}\left[e^{-2\pi\la\int_{R}^\infty\frac{\g r^{-\a}}{\mu+\g r^{-\a}}r{\rm d}r}\right],\;\;\;
\E^0\left[e^{-\g I_{\sst S}}\right] = \E^0_{H_{\sst S}}\left[ e^{-2\pi\la\int_{H_{\sst S}}^\infty\frac{\g r^{-\a}}{\mu+\g r^{-\a}}r{\rm d}r}\right],\]
similarly to the proof of Theorem~\ref{theorem:SINR-c} in Subsection~\ref{subsection:CF}, in the case of an environment with fading, by comparing the densities of $H_{\sst \Vcal}, H_{\sst S}$ and $R$, using the function $G$ analogous to (\ref{eq:Gh}) as 
\begin{align}
G(h)&:= e^{-\la\pi h^2 \kappa(\g h^{-\a}/\mu, \a)},
\end{align}
using the function $\kappa$ from~(\ref{eq:Int2}), along with the fact that $\frac{\g}{\mu+\g h^{-\a}}\leq 1$, for all $h\geq 0$.

To prove $I_{\sst \Ical}\geq_{\textnormal{L}}  I_{\sst I}$, we compare 
\begin{align}
\E^0\left[e^{-\g I_{\sst I}}\right] &= \E^0_{H_{\sst I}} \left[\frac{\mu}{\mu+\g H_{\sst I}^{-\a}} e^{-2\pi\la\int_{H_{\sst I}}^\infty\frac{\g u^{-\a}}{\mu+\g u^{-\a}}u{\rm d}u}\right],\nn
\end{align}
\begin{align}
\E^0\left[e^{-\g I_{\sst \Ical}}\right] &= \E^0_{H_{\sst \Ical}} \left[\left(\frac{\mu}{\mu+\g H_{\sst \Ical}^{-\a}}\right)^2 e^{-2\pi\la\int_{H_{\sst \Ical}}^\infty\frac{\g u^{-\a}}{\mu+\g u^{-\a}}u{\rm d}u}\right],
\nn
\end{align}
from (\ref{eq:gamma-int}) and (\ref{eq:LTmI1}), by comparing the densities of $H_{\sst I}$ and $H_{\sst \Ical}$, and using the $G$ function as
\begin{align}
G(h) &= \frac{\mu}{\mu+\g h^{-\a}} e^{-2\pi\la\int_h^\infty\frac{\g u^{-\a}}{\mu+\g u^{-\a}}u{\rm d}u},
\label{eq:gamma-intG}
\end{align}
along with the fact that $\frac{\g}{\mu+\g h^{-\a}}\leq 1$, for all $h\geq 0$.

For proving $\Tcal_{\sst{\Vcal}}\geq_{\textnormal{st}} \Tcal_{t} \geq_{\textnormal{st}} \Tcal_{\sst S}$, we compare
\begin{align}
\P^0(\Tcal_{\sst \Vcal} \leq x)&=\E^0_{H_{\sst \Vcal}}\left[\left(1{-}e^{-\mu H_{\sst \Vcal}^{\a}x}\right) e^{-2\la \pi K_\a(\mu,x, H_{\sst \Vcal})}\right],\nn
\end{align}
\[
\P^0(\Tcal_t \leq x)=\E^0_{R}\left[e^{-2\la \pi K_\a(\mu,x, R)}\right],\;\;\;
\P^0(\Tcal_{\sst S} \leq x)=\E^0_{H_{\sst S}}\left[ e^{-2\la \pi K_\a(\mu,x, H_{\sst S})}\right],
\]
by again comparing the densities of $H_{\sst \Vcal}, H_{\sst S}$ and $R$. In the proof we use the function $G$ similarly to (\ref{eq:Gh})  as $G(h):=e^{-2\la\pi K_\a(\mu,x, h)}$ and the fact that $(1-e^{-\mu h^\a x})\leq 1$ for any $h\geq 0$.

For $\Tcal_{\sst \Ical} \geq_{\textnormal{st}}  \Tcal_{\sst I}$, we compare
\begin{align}
\P^0(\Tcal_{\sst I} \leq x)&=\E^0_{H_{\sst I}}\left[\left(1{-}e^{-\mu H_{\sst I}^{\a}x}\right)\!\ e^{-2\la \pi K_\a(\mu,x,H_{\sst I})}\right],\nn
\end{align}
\begin{align}
\P^0(\Tcal_{\sst \Ical} \leq x)&{=}\E^0_{H_{\sst \Ical}}\left[\left(1{-}e^{-\mu H_{\sst \Ical}^{\a}x}\right)^2 \!\!\! e^{-2\la \pi K_\a(\mu,x,H_{\sst \Ical})}\right],\nn
\end{align}
by comparing the densities of $H_{\sst I}$ and $H_{\sst \Ical}$, and using the $G$ function as $G(h):= \left(1{-}e^{-\mu h^{\a}x}\right)e^{-2\la\pi K_\a(\mu,x, h)}$.\hfill $\square$
\subsection{Theorem~\ref{theorem:Int-order} (without fading)}
The comparison of the interference at different typical epochs in the no-fading scenario can be proved using same type of comparisons as in the case with fading and using the proof of Theorem~\ref{theorem:SINR-c} in no-fading case. For the first part, $I_{\sst \Vcal}\geq_{\textnormal{L}} I_{t} \geq_{\textnormal{L}} I_{\sst S}$, we compare
\begin{align}
\E^0\left[e^{-\nu I_{\sst \Vcal}}\right] &= \E^0_{H_{\sst \Vcal}}\left[e^{-\nu H_{\sst \Vcal}^{-\a} -\pi\la L_\nu(H_{\sst \Vcal}, \a)}\right],\nn
\end{align}
\[
\E^0\left[e^{-\nu I_{t}}\right] = \E^0_{R}\left[e^{ -\pi\la L_\nu(R, \a)}\right],\;\;\;
\E^0\left[e^{-\nu I_{\sst S}}\right] = \E^0_{H_{\sst S}}\left[e^{-\pi\la L_\nu(H_{\sst S}, \a)}\right],\]
by a similar comparison of the densities of $H_{\sst \Vcal}, R$ and $H_{\sst S}$, using the function $G(h)= e^{-\pi\la L_{\nu}(h,\a)}$, where $L_{\nu}(r,\a)$ is as defined in~(\ref{eq:NFI}). Also for proving $I_{\sst \Ical}\geq_{\textnormal{L}} I_{\sst I}$, we compare
\begin{align}
\E^0\left[e^{-\nu I_{\sst I}}\right] &= \E^0_{H_{\sst I}}\left[e^{-\nu H_{\sst I}^{-\a} -\pi\la L_\nu(H_{\sst I}, \a)}\right],\nn
\end{align}
\begin{align}
\E^0\left[e^{-\nu I_{\sst \Ical}}\right] &= \E^0_{H_{\sst \Ical}}\left[e^{-2\nu H_{\sst \Ical}^{-\a} -\pi\la L_\nu(H_{\sst \Ical}, \a)}\right], \nn
\end{align}
using the densities of $H_{\sst I}$ and $H_{\sst \Ical}$ and the function $G(h)= e^{-\nu h^{-\a}-\pi\la L_{\nu}(h,\a)}$. 

Part~\ref{lrd}: In the tropical case without fading, we have that $\Tcal_{\sst{\Vcal}} = H_{\sst \Vcal}^{-\a}\stackrel{d}{=} R_{\sst S}^{-\a}= \Tcal_{\sst S}$ since $H_{\sst \Vcal}\stackrel{d}{=} R_{\sst S}$, as in Lemma~\ref{lemma:LRD_order}. To prove the ordering,  
$\Tcal_{\sst S}\geq_{\textnormal{lrd}} \Tcal_{t}  $, note that
\begin{align}
\Tcal_{\sst S}= R_{\sst S}^{-\a} \geq_{\textnormal{lrd}} R_1^{-\a} = \Tcal_{t},
\end{align}
since $R_{\sst S}\leq_{\textnormal{lrd}} R_1$, from Lemma~\ref{lemma:LRD_order}. Similarly, 
$\Tcal_{\sst I} \geq_{\textnormal{lrd}}  \Tcal_{\sst \Ical}$ follows from the fact that $H_{\sst I}\leq_{\textnormal{lrd}} H_{\sst \Ical}$ from Lemma~\ref{lemma:LRD_order}.
\hfill $\square$
\section{Proofs: Scale invariance}~\label{section:SI}
\subsection{Theorem~\ref{theorem:si} (with fading)}

\textbf{Proof}: Recall the coverage probabilities for the typical handover case in (\ref{eq:hocp1}) with $\s=0$ as
\begin{align}
p_{\sst \Vcal}^c &= \frac{1}{1{+}\tau}\E^0_{H_{\sst \Vcal}}\left[ e^{-\pi\la H_{\sst \Vcal}^2 \kappa(\tau,\a)}\right] = \frac{1}{1{+}\tau} (1+\kappa(\tau,\a))^{-3/2},\nn
\end{align}
using the Laplace transform of $H_{\sst I}^2$ from Proposition~\ref{prop:IVHd}. The coverage probability $p_{\sst \Vcal}^c$ in this case depends only on $\tau, \a$. By a similar computation we can show that 
\[
p_{t}^c = (1+\kappa(\tau,\a))^{-1} \mbox{ and }p_{\sst S}^c  = (1+\kappa(\tau,\a))^{-\half},
\]
both of which depend only on $\tau, \a$. This shows the scale invariance property of the coverage probabilities with respect to the parameter $\la,\mu$.

For the case of typical max interference epoch, the coverage probability with $\s=0$ is
\begin{align}
p^c_{\sst I} 
&{=} \frac{2\tau^{-2/\a}}{\a}\int_0^\tau \!\frac{z^{2/\a-1}}{1{+}z}\E^0_{H_{\sst I}}\!\! \left[e^{-\pi\la H_{\sst I}^2  \kappa(z, \a)}\right]{\rm d}z. 
\label{eq:si2}
\end{align}
from (\ref{eq:OPI-4}).
Using the density of $H_{\sst I}$, after exchanging the integrals in (\ref{eq:si2}), we have
\begin{align}
p^c_{\sst I} 
&{=} \frac{2\tau^{-2/\a}}{\a}\int_0^\tau \!\frac{z^{2/\a-1}}{1{+}z} (1+\kappa(z, \a))^{-3/2}  {\rm d}z,
\label{eq:si4}
\end{align}
using the Laplace transform of $H_{\sst I}^2$ from Proposition~\ref{prop:IVHd}. The last expression for $p^c_{\sst I}$ in (\ref{eq:si4}) is independent of $\la, \mu$. Similarly we can show from (\ref{eq:cpTTIH2}) that 
\begin{align}
p_{\sst \Ical}^c =\frac{2\tau^{-2/\a}}{\a}\int_0^\tau \!\frac{z^{2/\a-1}}{(1{+}z)^2} \left(1{+}\kappa(z, \a)\right)^{-5/2}{\rm d}z,
\label{eq:si3}
\end{align}
which is independent of $\la,\mu$ and hence we have the result. The same scale invariance can be shown to hold for the tropical case as well. 
\hfill $\square$
\begin{remark}
Similarly to (\ref{eq:CP-CF}) we also obtain closed form for the coverage probability in the interference limited regime from (\ref{eq:si4}), $p_{\sst I}^c=0$ with $\a=2$, since $\kappa(z,2)=\infty$. For $\a=4$ we have $\kappa(z,4)= z^\half  \arctan(z^\half)$ and hence
\begin{align}
p_{\sst I}^c &=\frac{\tau^{-1/2}}{2}\int_0^\tau \!\frac{z^{-1/2}}{1{+}z} \left(1{+}z^\half  \arctan(z^\half)\right)^{-3/2} {\rm d}z.
\label{eq:CPI-CF}
\end{align}
\end{remark}
\begin{remark}
Similarly to (\ref{eq:CPI-CF}) we also obtain closed form for $p_{\sst \Ical}^c$ in the interference limited regime from (\ref{eq:si3}), $p_{\sst \Ical}^c=0$ with $\a=2$. For $\a=4$
\begin{align}
p_{\sst \Ical}^c &=\frac{\tau^{-1/2}}{2}\!\! \int_0^{\tau}\!\!\! \frac{z^{-1/2}}{(1{+}z)^2} \left(1{+}z^\half \arctan(z^\half)\right)^{-5/2}{\rm d}z,
\nn
\end{align}
\end{remark}
\subsection{Theorem~\ref{theorem:si} (without fading)}
\textbf{Proof}: At the typical mS-MI in the non-tropical case we have
\begin{align}
p^c_{\sst \Vcal} &=\P^0\left(H_{\sst \Vcal}^{-\a}/{I_{H_{\sst \Vcal}}}\geq \tau \right).\nn
\end{align}
Since the ratio $\frac{H_{\sst \Vcal}^{-\a}}{I_{H_{\sst \Vcal}}}$ is independent of $\la$, so is the coverage probability. The same reasoning applies for the scale invariance of the coverage probabilities in the four other cases:
\begin{itemize}
\item typical time:  $p^c_{t} =\P^0\left(R^{-\a}/{I_R}\geq \tau \right)$, 
\item typical MS: $p^c_{\sst S} =\P^0\left(H_{\sst S}^{-\a}/{I_{H_{\sst S}}}\geq \tau \right)$, 
\item typical MI: $p^c_{\sst I} =\P^0\left(H_{\sst I}^{-\a}/{I_{H_{\sst I}}}\geq \tau \right)$, 
\item typical mI: $p^c_{\sst \Ical} =\P^0\left(H_{\sst \Ical}^{-\a}/{I_{H_{\sst \Ical}}}\geq \tau \right)$,
\end{itemize}
from (\ref{eq:cp-hp-nf}), (\ref{eq:WFTtime1}), (\ref{eq:WFV1}) and (\ref{eq:cpTIH}), respectively. For typical mS-MI in the tropical case, from (\ref{eq:WFTP2}), we have
\begin{align}
p^c_{\sst \Vcal;\Tcal} 
&{=} \int_0^\infty \one_{h^{-\a}>\tau h^{-\a}} f_{H_{\sst \Vcal}}(h){\rm d}h =
\begin{cases}
    1 \mbox{ if } \tau<1\\
    0 \mbox{ if } \tau= 1,
\end{cases}\nn
\end{align}
since $\tau\leq 1$. At a typical time from (\ref{eq:WFTropV}), for any $\tau\geq 0$ we have
\begin{align}
p^c_{\sst{t;\Tcal}} &= \int_{0}^\infty
\P\left(R_1>\tau^{1/\a} r\right)
f_{R}(r){\rm d}r 
=\int_{0}^\infty \int_{\tau^{1/\a}r\vee r}^\infty 2\la\pi h \ e^{-\lambda\pi (h^2-r^2)} {\rm d}h f_{R}(r){\rm d}r.\nn
\end{align}
For $\tau\leq 1$, we have $\tau^{1/\a}\leq 1$ and hence $p^c_{t;\sst{\Tcal}} =1$ in this case. For $\tau>1$, we have $\tau^{1/\a} > 1$ and  
\begin{align}
p^c_{\sst{t;\Tcal}}  &=\int_{0}^\infty \int_{\tau^{1/\a}r}^\infty 2\la\pi h \ e^{-\lambda\pi (h^2-r^2)} {\rm d}h f_{R}(r){\rm d}r\nn\\
&=\int_{0}^\infty \int_{\la\pi\tau^{2/\a}r^2}^\infty  e^{-z} e^{\lambda\pi r^2} {\rm d}z f_{R}(r){\rm d}r\nn\\
&=\int_{0}^\infty e^{-\lambda\pi r^2(\tau^{2/\a}-1)} f_R(r) {\rm d}r 
=\int_{0}^\infty e^{-\lambda\pi r^2\tau^{2/\a}} 2\la\pi r {\rm d}r = \tau^{-2/\a}.\nn
\end{align}
We can similarly prove from (\ref{eq:WFTropV}) that $p^c_{\sst S;\Tcal}$ is also independent of $\la$, as follows:
\begin{align}
p^c_{\sst S;\Tcal} &= \int_{0}^\infty
\P\left(R_{\sst S}>\tau^{1/\a} h\right) f_{H_{\sst S}}(h){\rm d}h 
=\int_{0}^\infty
\int_{\tau^{1/\a} h \vee h}^\infty  f_{R_{\sst S}\vert H_{\sst S}=h}(r){\rm d}r
f_{H_{\sst S}}(h){\rm d}h.\nn
\end{align}
For $\tau\leq 1$, we have $p^c_{\sst S;\Tcal}=1$. For $\tau>1$ we have
\begin{align}
p^c_{\sst S;\Tcal} &= \int_{0}^\infty
\int_{\tau^{1/\a} h}^\infty  f_{R_{\sst S}\vert H_{\sst S}=h}(r){\rm d}r
f_{H_{\sst S}}(h){\rm d}h\nn\\
&= \int_{0}^\infty
\int_{\tau^{1/\a} h}^\infty 2\la\pi r e^{-\la\pi(r^2-h^2)}{\rm d}r
f_{H_{\sst S}}(h){\rm d}h\nn\\
&= \int_{0}^\infty \int_{\la\pi\tau^{2/\a} h^2}^\infty   e^{-z} {\rm d}z e^{\la\pi h^2} f_{H_{\sst S}}(h){\rm d}h \nn\\
&= \int_{0}^\infty  e^{-\la\pi(\tau^{2/\a}-1) h^2} f_{H_{\sst S}}(h){\rm d}h = \E^0_{H_{\sst \Vcal}}\left[e^{-\la\pi(\tau^{2/\a}-1) H_{\sst \Vcal}^2}\right]= \left(1+(\tau^{2/\a}-1)\right)^{-1/2}= \tau^{-1/\a},
\end{align}
using the Laplace transform of $H_{\sst \Vcal}^2$.

In case of $\tau\leq 1$, at the typical max-interference, we have from (\ref{eq:WFtropI}) that 
\begin{align}
p^c_{{\sst I ;\Tcal}} 
&=\P^0\left(H_{\sst I}\geq  R_{\sst I}\tau^{1/\a}\right)
=\int_0^\infty \P^0(h^{-2}R^2_{\sst I}\leq \tau^{-2/\a}) f_{H_{\sst I}}(h){\rm d}h
=\int_0^\infty \tau^{-2/\a} f_{H_{\sst I}}(h){\rm d}h = \tau^{-2/\a},\nn
\end{align}
using Corollary~\ref{corollary:RbyH}. The final expression is independent of $\la$. For the typical min interference, from (\ref{eq:TI-cpT}) we have
\begin{align}
p^c_{\sst \Ical;\Tcal} = \E^0_{H_{\sst \Ical},R_{\sst \Ical}}\left[F_{I_{H_{\sst \Ical}}}\left( R_{\sst \Ical}^{-\a}/\tau\right) \right]
=\P^0\left(H_{\sst \Ical}^{-\a}\leq R_{\sst \Ical}^{-\a}/\tau\right)
&=\int_0^\infty \P^0\left(h^{-2}R^2_{\sst \Ical}\leq \tau^{-2/\a}\right) f_{H_{\sst \Ical}}(h){\rm d}h\nn\\
&=\int_0^\infty \tau^{-2/\a} f_{H_{\sst \Ical}}(h){\rm d}h = \tau^{-2/\a},\nn
\end{align}
using Corollary~\ref{corollary:RbyHIcal}, where the last term is independent of $\la$. \hfill $\square$
\section{Other attenuation functions}~\label{Example}
\subsection{Bounded  path-loss function} In the case of bounded attenuation function $\ell(r)=(1+r)^{-\a}$ as in \ref{baf}, given $H=h$ we have the coverage probability as
\begin{align}
\P(\Scal>\tau\vert H=h)
= \E_{I_{h}}\left[\P(\Scal>\tau\vert I_h)\right]
&=\E_{I_{h}}\left[\P\left(\rho >\tau  (1+h)^{\a}\left(\s^2+I_{h}\right)\right)\right]\nn\\
&=\E_{I_{h}}\left[e^{-\mu\tau (1+h)^{\a}\left(\s^2+I_{ h}\right)}\right]\nn\\
&= e^{-\mu\tau (1+h)^{\a}\s^2}\Lcal_{I_{h}}(\mu\tau (1+h)^{\a}).\nn
\end{align}
For example, in the case of signal handover $I_h:=\rho (1+h)^{-\a}+\sum_{i} \rho_i (1+D_i)^{-\a}$. Computing the PGFl with respect to the point process $\eta_h$ we have  the Laplace transform of $I_h$ for any $\gamma\geq 0$ as
\begin{align}
\E_{I_{h}}\left[e^{-\g I_{h}}\right]
&= \frac{\mu}{\mu+\g (1+h)^{-\a}} e^{-2\pi\la\int_{h}^\infty\frac{\g (1+r)^{-\a}}{\mu+\g (1+r)^{-\a}}r{\rm d}r}.
\label{eq:BA1}
\end{align}
Hence for $\g=\mu\tau (1+h)^{\a}$, we re-write the integral in the exponent in (\ref{eq:BA1}) as 
\[
\int_{h}^\infty\frac{\tau}{\tau+ (1+r)^\a(1+h)^{-\a}}  r{\rm d}r := M_\a(h,\tau)
\]
and we have the required Laplace transform as  
\begin{align}
\E_{I_{h}}\left[e^{-\mu\tau (1+h)^{\a} I_{h}}\right] &= \frac{1}{1+\tau} e^{-2\la\pi M_\a(h,\tau)}.\nn
\end{align}
The coverage probability is 
\[
p^c_{\sst \Vcal}= \frac{1}{1+\tau} \E^0_{H_{\sst \Vcal}}\left[e^{-\mu\tau (1+H_{\sst \Vcal})^{\a}\s^2} e^{-2\la\pi M_\a(H,\tau)}\right].
\]
Similarly for the typical time and typical MS, we have
\[
p^c_{\sst t}= \E^0_{R}\left[e^{-\mu\tau (1+R)^{\a}\s^2} e^{-2\la\pi M_\a(R,\tau)}\right],
\]
\[
p^c_{\sst S}= \E^0_{H_{\sst S}}\left[e^{-\mu\tau (1+H_{\sst S})^{\a}\s^2} e^{-2\la\pi M_\a(H_{\sst S},\tau)}\right].
\]
A comparison result similar to Theorem~\ref{theorem:SINR-c} can be shown to hold true, which is $p^c_{\sst \Vcal}\leq p^c_{\sst t}\leq p^c_{\sst S}$. Similarly, for the other two typical epochs 
\begin{align}
p_{\sst I}^c 
&{=}E^0_{H_{\sst I}, R_{\sst I}}\!\! \left[Q_\a(H_{\sst I}, R_{\sst I},\tau) e^{-\mu\tau (1+R_{\sst I})^{\a}\s^2} e^{-2\pi\la\int_{H_{\sst I}}^\infty Q_\a(R_{\sst I},u,\tau)u{\rm d}u}\right],\nn
\end{align}
\begin{align}
p_{\sst \Ical}^c 
&{=}E^0_{H_{\sst \Ical}, R_{\sst \Ical}}\!\! \left[\left(Q_\a(H_{\sst \Ical}, R_{\sst \Ical},\tau)\right)^2 e^{-\mu\tau (1+R_{\sst \Ical})^{\a}\s^2}  e^{-2\pi\la\int_{H_{\sst \Ical}}^\infty Q_\a(R_{\sst \Ical},u,\tau)u{\rm d}u}\right],\nn
\end{align}
where $Q_\a(x,y,\tau):=\frac{(1+x)^\a}{(1+x)^\a+\tau (1+y)^\a}$. It can be proved similarly that $p_{\sst \Ical}^c \leq p_{\sst I}^c$. 

\subsection{Step attenuation function} Consider the step attenuation function $\ell(r)=p \one_{r\leq d}$ as in~\ref{step-fun}. The SINR is defined as 
\[
\Scal\equiv \mbox{SINR}:= \frac{\rho \, p \one_{H\leq d}}{\s^2+\sum_{i} \rho_i\, p \one_{H\leq D_i\leq d}}.
\]
Then the coverage probability with respect to the Palm probability distribution of epoch of interest is
\begin{align}
p^c(\tau,\mu,\la,\a)&= \E^0_{H}[\P(\Scal>\tau\vert H)],
\label{eq:st1}
\end{align}
where $f_H$ is the density of the distance $H$ to the serving station. Conditioned on $H=h$ the inner term (\ref{eq:st1}) is
\begin{align}
\P(\Scal>\tau\vert h)
&= \E_{I_{h}}\left[\P(\Scal>\tau\vert h, I_h)\right] \left(1-e^{-\la\pi d^2}\right)\nn\\
&=\E_{I_{h}}\left[\P\left(\rho >\frac{\tau}{p} \left(\s^2+I_{h}\right)\right)\right]\left(1-e^{-\la\pi d^2}\right) \nn\\
&=\E_{I_{h}}\left[e^{-\frac{\mu\tau}{p}\left(\s^2+I_{ h}\right)}\right] \left(1-e^{-\la\pi d^2}\right)
= e^{-\frac{\mu\tau\s^2}{p}}\!\! \left(1{-}e^{-\la\pi d^2}\right)\! \E_{I_{h}}\!\left[e^{-\frac{\mu\tau}{p} I_h}\right].
\nn
\end{align}
Using the PGFl of the max-shot noise of a Poisson point process, we get
\begin{align}
\E_{I_{h}}\left[e^{-\frac{\mu\tau}{p} I_h}\right]&=
\begin{cases}
    \frac{1}{1+\tau} e^{-2\la\pi(d-h)\tau/(1+\tau)}, \text{ for typical ms-MI,}\\
    e^{-2\la\pi(d-h)\tau/(1+\tau)}, \text{ for typical MS,}\\
    e^{-2\la\pi(d-h)\tau/(1+\tau)}, \text{ for typical time}.
\end{cases}\nn
\end{align}
As a result the coverage probabilities are
\begin{align}
p^c(\tau,\mu,\la,\a) &=
\begin{cases}
    \frac{C}{1+\tau}   \E^0_{H_{\sst \Vcal}}\left[e^{\g H_{\sst \Vcal}}\right], \text{ for typical ms-MI,}\\
    C \E^0_{H_{\sst S}}\left[e^{\g H_{\sst S}}\right], \text{ for typical MS,}\\
    C \E^0_{R}\left[e^{\g R}\right], \text{ for typical time},
\end{cases}\nn
\end{align}
where $\g= 2\la\pi\tau/(1+\tau)$ and 
\[
C\equiv C(\tau,d,\mu,\la,\s):= e^{-\frac{\mu\tau\s^2}{p}}\left(1-e^{-\la\pi d^2}\right) e^{-2\la\pi \tau/(1+\tau)}.
\]
Even though we have $H_{\sst \Vcal}\geq_{\textnormal{mgf}} R\geq_{\textnormal{mgf}} H_{\sst S}$ from Lemma~\ref{lemma:Laplace_order}, this shows that no distinct ordering holds true among $p_{\sst \Vcal}^c, p_{t}^c$ and $p_{\sst S}^c$, for all values of $\tau$. Similarly using 
\[
\mbox{SINR}:= \frac{\rho \, p \one_{R\leq d}}{\s^2+\sum_{i} \rho_i\, p \one_{H\leq D_i\leq d}}, 
\]
with $R\in \{R_{\sst I}, R_{\sst \Ical}\}$ and $H\in \{H_{\sst I}, H_{\sst \Ical}\}$ at typical mI and MI, respectively, one can show different ordering between $p^c_{\sst \Ical} $ and $ p^c_{\sst I}$ depending for different values of $\tau$, even though we have $H_{\sst\Ical}\geq_{\textnormal{mgf}} H_{\sst I}$ and $R_{\sst\Ical}\geq_{\textnormal{mgf}} R_{\sst I}$, from Lemma~\ref{lemma:Laplace_order}. 

It is also evident that in the interference limited regime, the coverage probabilities are not scale invariant  with respect to $\la$, as the constant 
\[
C(\tau,d,\mu,\la,\s)\vert_{\s=0}= \left(1-e^{-\la\pi d^2}\right) e^{-2\la\pi \tau/(1+\tau)}
\]
is not. On the other hand, the $\mu$-scale invariance holds true. 
\section{Special integrals}
\subsection{The integral in (\ref{eq:kappa-a})}
%
We determine the integral in (\ref{eq:kappa-a}) of the form $\int_{h}^\infty\frac{\tau h^{\a}}{r^{\a}+\tau h^{\a}}r{\rm d}r$ for some $h$, $\tau$ and $\a$ positive.
Taking a change of variable $\tau^{-2/\a} h^{-2}r^2=z$ we obtain that 
\begin{align}
\int_{h}^\infty\frac{\tau h^{\a}}{r^{\a}+\tau h^{\a}}r{\rm d}r &= \half \tau^{2/\a} h^2 \int_{\tau^{-2/\a}}^\infty\frac{1}{1+z^{\a/2}} {\rm d}z
= \half h^2 \kappa(\tau,\a),\nn
\end{align}
where 
\begin{equation}
\kappa(\tau,\a):=\tau^{2/\a} \int_{\tau^{-2/\a}}^\infty\frac{1}{1+z^{\a/2}} {\rm d}z.
\label{eq:kappa}
\end{equation}
\subsection{The integral (\ref{eq:int-1})}
We compute the integral $\int_h^\infty\frac{\g r^{-\a}}{\mu+\g r^{-\a}}r{\rm d}r$ in the exponent in (\ref{eq:int-1}) by a change of variable $(\g/\mu)^{-2/\a} r^2=z$
\begin{align}
\int_h^\infty\frac{\g r^{-\a}}{\mu+\g r^{-\a}}r{\rm d}r
&= \frac{1}{\mu} \int_h^\infty\frac{\g/\mu}{\g/\mu+ r^{\a}}r{\rm d}r  \nn\\
&= \frac{1}{2\mu}  (\g/\mu)^{2/\a}\int_{(\g/\mu)^{-2/\a}h^2}^\infty  (1+ z^{\a/2})^{-1}{\rm d}z\label{eq:Int1}\\
&= \frac{h^2}{2\mu}  (\g h^{-\a}/\mu)^{2/\a}\int_{(\g h^{-\a}/\mu)^{-2/\a}}^\infty  (1+ z^{\a/2})^{-1}{\rm d}z
= \frac{1}{2\mu} h^2 \kappa(\g h^{-\a}/\mu, \a),
\label{eq:Int2}
\end{align}
from (\ref{eq:kappa}). Note from (\ref{eq:Int1}) that $h^2 \kappa(\g h^{-\a}/\mu, \a)$ is a non-increasing function of $h$ for fixed $\a, \g$ and $\mu$.
\subsection{The integral in (\ref{eq:hralpha})} 
For the integral $\int_h^\infty\frac{\tau r^{\a}}{u^{\a}+\tau r^{\a}}u{\rm d}u$ in (\ref{eq:hralpha}), by the change of variable $\tau^{-2/\a} r^{-2}u^2=z$ we have
\begin{align}
\int_h^\infty\!\!\! \frac{\tau r^{\a}}{u^{\a}+\tau r^{\a}}u{\rm d}u &{=}  \frac{\tau^{2/\a} r^2}{2}\!\! \int_{\tau^{-2/\a} (r/h)^{-2}}^\infty\!\frac{1}{z^{\a/2}+1}{\rm d}z\nn\\
&{=}  \frac{h^2}{2} (\tau (r/h)^\a)^{2/\a} \!\! \int_{\tau (r/h)^\a)^{-2/\a}}^\infty\!\frac{1}{z^{\a/2}+1}{\rm d}z 
\stackrel{(\ref{eq:kappa})}{=}  \frac{h^2}{2} \kappa(\tau (r/h)^\a, \a).
\label{eq:kappar-h}
\end{align}
\subsection{The integral in (\ref{eq:STINR-h})} For the integral  $\int_{h}^{\infty}e^{-\mu r^\a x}r{\rm d}r$ in (\ref{eq:STINR-h})  we have
\begin{align}
\int_{h}^{\infty}e^{-\mu r^\a x}r{\rm d}r &= \frac{1}{\a} \int_{\mu h^\a x}^{\infty}e^{-z} \left(\frac{z}{\mu x}\right)^{2/\a} \frac{1}{z}{\rm d}z\nn\\
&= \frac{1}{\a} (\mu x)^{-2/\a} \int_{\mu h^\a x}^{\infty} z^{2/\a-1}e^{-z}{\rm d}z\nn\\
&= \frac{1}{\a} (\mu x)^{-2/\a} \Gamma\left(2/\a,\mu h^\a x\right)
:=K_\a(\mu,x, h),
\label{eq:kalpha}
\end{align}
where $\Gamma(\cdot,\cdot)$ is the upper incomplete gamma function defined as $\Gamma(a,b):=\int_b^\infty z^{a-1} e^{-z}{\rm d}z$
\subsection{The integral in (\ref{eq:LI})}
The integral $\int_{h}^\infty\left(1-e^{-\nu r^{-\a}}\right) r{\rm d}r$ in (\ref{eq:LI}) is computed as
\begin{align}
\int_{h}^\infty\left(1-e^{-\nu r^{-\a}}\right) r{\rm d}r
=  \frac{1}{\a} \nu^{2/\a}\int_0^{\nu h^{-\a}}\left(1-e^{-z}\right) z^{-\frac{2}{\a}-1} {\rm d}z
&=  -\frac{1}{2}  h^2 (1-e^{-\nu h^{-\a}}) + \frac{1}{2} \nu^{2/\a}\gamma\left(1-\frac{2}{\a}, \nu h^{-\a}\right)\nn\\
&:= \half L_\nu(h, \a),
\label{eq:NFI}
\end{align}
where $\gamma(\cdot,\cdot)$ is the lower incomplete gamma function defined as $\gamma(a,b):=\int_0^b z^{a-1} e^{-z}{\rm d}z$. 
\section*{Acknowledgement}
The authors would like to thank Jeffrey Andrews for many valuable suggestions, which improved the article substantially. Many thanks to the anonymous reviewers for reading article with patience and making many important comments, which has significantly improved the results, impact and readability of the manuscript. The authors are thankful to the ERC-NEMO grant, under the European Union’s Horizon 2020 research and innovation program, grant agreement number 788851 to INRIA Paris. This research was also funded in part by the France 2030 BPI ``5G NTN mmWave'' project to T{\'e}l{\'e}com Paris.
\bibliographystyle{IEEEtran}
\bibliography{ref}
\end{document}